\documentclass{cpamart1}     
\received{Month 200X}       
\volume{000}
\startingpage{1}                      

\authorheadline{G. Amir, I. Corwin, J. Quastel}
\titleheadline{Continuum Random Polymer}

\usepackage{graphicx}
\usepackage{amssymb}
\usepackage{amsthm}
\usepackage{bm}
\usepackage{mathrsfs}
\usepackage{yfonts}
\usepackage{enumerate}
\usepackage{lscape}

\usepackage{mathptmx}

\newtheorem{theorem}{Theorem}[section]
\newtheorem{lemma}[theorem]{Lemma}
\newtheorem{corollary}[theorem]{Corollary}
\newtheorem{proposition}[theorem]{Proposition}
\theoremstyle{definition}
\newtheorem{definition}[theorem]{Definition}
\theoremstyle{remark}
\newtheorem{remark}[theorem]{Remark}

\newtheorem{conj}[theorem]{Conjecture}
\newtheorem{claim}[theorem]{Claim}

\DeclareMathOperator*{\tr}{tr}
\DeclareMathOperator*{\sgn}{sgn}

\newcommand{\re}{\ensuremath{\mathrm{Re}}}
\newcommand{\im}{\ensuremath{\mathrm{Im}}}

\newcommand{\C}{\ensuremath{\mathbb{C}}}

\newcommand{\Z}{\ensuremath{\mathbb{Z}}}

\newcommand{\R}{\ensuremath{\mathbb{R}}}

\newcommand{\Hi}[0]{\mathcal{H}}
\newcommand{\kapi}{\omega}
\newcommand{\Ai}{\ensuremath{\mathrm{Ai}}}

\newcommand{\e}{\epsilon}

\newcommand{\ep}{\epsilon}
\newcommand{\eps}{\epsilon}

\begin{document}                        


\title{Probability  Distribution of the Free Energy of the Continuum Directed Random Polymer in $1+1$ dimensions}

\author{Gideon Amir}{University of Toronto}
\author{Ivan Corwin}{Courant Institute}
\author{Jeremy Quastel}{University of Toronto}





\begin{abstract}
We consider the solution of the stochastic heat equation 
\begin{equation*}
 \partial_T \mathcal{Z}  =  \frac12 \partial_X^2 \mathcal{Z} - \mathcal{Z} \dot{\mathscr{W}}
\end{equation*} with delta function initial condition
\begin{equation*}
 \mathcal{Z}  (T=0,X)= \delta_{X=0}
\end{equation*}
whose logarithm, with appropriate normalization, is the free energy of the continuum directed polymer, or the Hopf-Cole solution of the Kardar-Parisi-Zhang equation with narrow wedge initial conditions.  

We obtain explicit formulas for the one-dimensional marginal distributions, the {\it crossover distributions}, which interpolate between a standard Gaussian distribution (small time) and the GUE Tracy-Widom distribution (large time).   
 
 The proof is via a rigorous steepest descent analysis of the Tracy-Widom formula for the asymmetric simple exclusion with anti-shock initial data, which is shown to converge to the continuum equations in an appropriate weakly asymmetric limit. The limit also describes the crossover behaviour between the symmetric and asymmetric exclusion processes. 
\end{abstract}

\maketitle   






\section{Introduction}
\subsection{KPZ/Stochastic Heat Equation/Continuum Random Polymer}\label{KPZ_SHE_Poly}  
Despite its popularity as perhaps {\it the} default model of stochastic growth of a one dimensional interface, we
are still far from a satisfactory theory of the Kardar-Parisi-Zhang (KPZ) equation,
\begin{equation}\label{KPZ0}
\partial_T h = -\frac12(\partial_Xh)^2 + \frac12 \partial_X^2 h +  \dot{\mathscr{W}}
\end{equation}
where $\dot{\mathscr{W}}(T,X)$\footnote{We attempt to use capital letters for all variables (such as $X$, $T$) on the macroscopic level of the stochastic PDEs and polymers. Lower case letters (such as $x$, $t$) will denote WASEP variables, the microscopic discretization of these SPDEs.} is Gaussian space-time white noise,
 \begin{equation*}
E[ \dot{\mathscr{W}}(T,X) \dot{\mathscr{W}}(S,Y)] = \delta(T-S)\delta(Y-X).
\end{equation*}
The reason is that, even for nice initial data, the solution at time $T>0$ will look locally like a Brownian motion in $X$. Hence the nonlinear term is ill-defined.  Naturally, one expects that an appropriate Wick ordering of the non-linearity can lead to well defined solutions.  However, numerous attempts have led to non-physical answers \cite{TChan}.  By a physical answer one means that the 
solution should be close to discrete growth models.  In particular, for a large class of initial data, the solution $h(T,X)$ should look like 
\begin{equation}\label{five}
h(T,X) \sim C(T) + T^{1/3} \zeta(X)
\end{equation}
where $C(T)$ is deterministic and where the statistics of $\zeta$ fits into various universality classes depending on the regime of initial data one is looking at.  More precisely, one expects that the variance scales as 
\begin{equation}\label{five2}
{\rm Var}(h(T,X)) \sim CT^{2/3}.
\end{equation}
The scaling exponent is the result of
extensive Monte Carlo simulations and a few theoretical arguments \cite{fns,vBKS,KPZ,K}. 

The correct interpretation of (\ref{KPZ0}) appears to be that of \cite{BG}, where $h(T,X)$ is simply {\it defined} by the Hopf-Cole transform:
\begin{equation}\label{hc}
h(T,X) = -\log \mathcal{Z} (T,X)
\end{equation}
where $\mathcal{Z} (T,X)$ is the well-defined \cite{W} solution of the stochastic heat equation, 
\begin{equation}\label{she}
 \partial_T \mathcal{Z}  =  \frac12 \partial_X^2 \mathcal{Z} - \mathcal{Z} \dot{\mathscr{W}}.
\end{equation}
Recently \cite{bqs} proved upper and lower bounds of the type (\ref{five2}) for this {\it Hopf-Cole solution} $h$ {\it of KPZ} defined through (\ref{hc}), in the {\it equilibrium regime}, corresponding to starting  (\ref{KPZ0}) with a two sided
Brownian motion.  Strictly speaking, this is not an equilibrium solution for KPZ, but for the stochastic Burgers equation 
\begin{equation*}
\partial_T u = -\frac12 \partial_X u^2 + \frac12 \partial_X^2 u +  \partial_X\dot{\mathscr{W}},
\end{equation*}
formally satisfied by its derivative $ u(T,X)=\partial_X h(T,X)$. See also \cite{Timo} for a similar bound for the free energy of 
a particular discrete polymer model.

In this article, we will be interested in a very different regime, far from equilibrium.  It is most convenient to state in terms of the stochastic heat equation (\ref{she}) for which we will have as initial condition a delta function,
\begin{equation}\label{sheinit}
 \mathcal{Z}  (T=0,X)= \delta_{X=0}.
\end{equation}
This initial condition is natural  for  the interpretation in terms of random polymers, where it corresponds to the point-to-point free energy. The free energy of the continuum directed random polymer in $1+1$ dimensions is
\begin{equation}\label{FK}
\mathcal{F}(T,X) = \log E_{T,X}\left[ :\!\exp\!: \left\{-\int_0^T \dot{\mathscr{W}}(t,b(t)) dt\right\}\right]
\end{equation}
where $E_{T,X}$ denotes expectation over the Brownian bridge $b(t)$, $0\le t\le T$ with $b(0)=0$ and $b(T)=X$.
The expectation of the Wick ordered exponential $ :\!\exp\!:$  is defined using the $n$ step probability densities $p_{ t_1, \ldots, t_n}(x_1,\ldots,x_n)$ of the bridge in terms of a series of multiple It\^o integrals;
\begin{eqnarray}\label{nine}
&& E_{T,X}\left[ :\!\exp\!: \left\{-\int_0^T \dot{\mathscr{W}}(t,b(t)) dt \right\}\right]
\\ && 
= \sum_{n=0}^\infty \int_{\Delta_n(T)} \int_{\mathbb R^n}(-1)^n
p_{ t_1, \ldots, t_n}(x_1,\ldots,x_n) \mathscr{W} (dt_1 dx_1) \cdots \mathscr{W} (dt_n dx_n),\nonumber
\end{eqnarray}
where $\Delta_n(T)=\{(t_1,\ldots,t_n):0\le t_1\le \cdots\le t_n \le T\}$.
Note that the series is convergent in $\mathscr{L}^2(\mathscr{W})$ as one can check that
\begin{equation*}
\int_{\Delta_n(T)} \int_{\mathbb R^n}
p^2_{ t_1, \ldots, t_n}(x_1,\ldots,x_n) dt_1 dx_1 \cdots dt_n dx_n \le C (n!)^{-1/2}
\end{equation*}
and hence the square of the norm,
$
\sum_{n=0}^\infty \int_{\Delta_n(T)} \int_{\mathbb R^n}
p^2_{ t_1, \ldots, t_n}(x_1,\ldots,x_n) dt_1 dx_1 \cdots dt_n dx_n
$, is finite.
Let 
\begin{equation}\label{heat_kernel}
p(T,X) = \frac{1}{\sqrt{2\pi T}} e^{ -X^2/2T }
\end{equation} 
denote the heat kernel.  Then we have 
\begin{equation}\label{zed}
\mathcal{Z}(T,X) = p(T,X) \exp\{ \mathcal{F}(T,X) \}
\end{equation}
as can be seen by writing the integral equation for $\mathcal{Z}(T,X)$;
\begin{equation}\label{intshe}
\mathcal{Z}(T,X) = p(T,X) + \int_0^T\int_{-\infty}^\infty p(T-S,X-Y)\mathcal{Z}(S,Y) \mathscr{W} (dY,dS)
\end{equation}
and then iterating.  The factor $p(T,X)$ in (\ref{zed}) represents the difference between conditioning on the bridge going to $X$, as in (\ref{nine}), and having a delta function initial condition, as in (\ref{sheinit}). The  initial condition corresponds to
\begin{equation*}
\mathcal{F}(0,X) = 0, \qquad X\in \mathbb{R}.
\end{equation*}
In terms of KPZ (\ref{KPZ0}), there is no precise mathematical statement of the initial conditions; what one sees as $T\searrow 0$ is
a narrowing parabola.  In the physics literature this is referred as the {\it narrow wedge initial conditions}.

 We can now state our main result which is an exact formula for the probability distribution for the free energy of the continuum directed random polymer in 1+1 dimensions, or, equivalently, the one-point distribution for the stochastic heat equation with delta initial condition, or the KPZ equation with narrow wedge initial conditions. 
 
  For a function $\sigma(t)$, define the operator $K_{\sigma}$ 
through its kernel,
\begin{equation}\label{sigma_K_def}
K_{\sigma}(x,y) = \int_{-\infty}^{\infty} \sigma(t)\Ai(x+t)\Ai(y+t)dt,
\end{equation}
where     $\mathrm{Ai}(x) = \frac{1}{\pi} \int_0^\infty \cos\left(\tfrac13t^3 + xt\right)\, dt$ is the Airy function.

\begin{theorem}\label{main_result_thm} The {\em crossover distributions}, \begin{equation}\label{defofF}
F_T(s) \stackrel{\rm def}{=} P(\mathcal{F}(T,X)+\tfrac{T}{4!} \le s)
\end{equation} are given by  the following equivalent formulas,

\begin{enumerate}[i.]
\item The {\em crossover Airy kernel formula},
\begin{equation}\label{sigma_Airy_kernel_formula}
F_{T}(s) =\int_{\mathcal{\tilde C}}\frac{d\tilde\mu}{\tilde\mu} e^{-\tilde\mu} \det(I-K_{\sigma_{T,\tilde\mu}})_{L^2(\kappa_T^{-1}a,\infty)},
\end{equation}
where $\mathcal{\tilde C}$ is defined in Definition \ref{thm_definitions}, and $ K_{\sigma_{T,\tilde\mu}}$ is as above with
\begin{equation}\label{airy_like_kernel}
 \sigma_{T,\tilde\mu}(t) = \frac{\tilde\mu}{\tilde\mu-e^{-\kappa_T t }},
\end{equation}
and
\begin{equation*}
 a=a(s)=s-\log\sqrt{2\pi T}, \quad \textrm{and}\quad \kappa_T=2^{-1/3}T^{1/3}.
\end{equation*}

Alternatively, if $\sqrt{z}$ is defined by taking the branch cut of the logarithm on the negative real axis, then
\begin{eqnarray}\label{sym_F_eqn}
F_{T}(s) &=&\int_{\mathcal{\tilde C}}\frac{d\tilde\mu}{\tilde\mu} e^{-\tilde\mu} \det(I-\hat{K}_{\sigma_{T,\tilde\mu}})_{L^2(-\infty,\infty)}\\
\hat{K}_{\sigma_{T,\tilde\mu}}(x,y) &=& \sqrt{\sigma_{T,\tilde\mu}(x-s)}K_{\Ai}\sqrt{\sigma_{T,\tilde\mu}(y-s)}
\end{eqnarray}
where $K_{\Ai}(x,y)$ is the Airy kernel, ie.  $K_{\Ai}=K_{\sigma}$  with $\sigma(t)=\mathbf{1}_{[0,\infty)}(t)$.

\mbox{}
\item The {\em Gumbel convolution formula},
\begin{equation*}
 F_T(s) =1- \int_{-\infty}^{\infty} G(r) f(a-r) dr,
\end{equation*}
where $G(r)$ is given by $G(r) =e^{-e^{-r}}$
and where 
\begin{equation*}
f(r) =\kappa_T^{-1} \det(I-K_{\sigma_T})\mathrm{tr}\left((I-K_{\sigma_T})^{-1}\rm{P}_{\Ai}\right),
\end{equation*}
where the operators $K_{\sigma_T}$ and $\rm{P}_{\Ai}$ act on $L^2(\kappa_T^{-1}r,\infty)$ and are given by their kernels with
\begin{equation*}
\rm{P}_{\Ai}(x,y) =  \Ai(x)\Ai(y),\qquad
 \sigma_T(t) = \frac{1}{1-e^{-\kappa_T t}}.
\end{equation*}
For $\sigma_T$ above, the integral in (\ref{sigma_K_def}) should be intepreted as a principal value integral. 
The operator $K_{\sigma_T}$ contains a  Hilbert transform of the product of Airy functions which can be partially computed with the result that
\begin{equation*}
 K_{\sigma_T}(x,y) = \int_{-\infty}^{\infty} \tilde\sigma_T(t)\Ai(x+t)\Ai(y+t)dt + \kappa_T^{-1} \pi G_{\frac{x-y}{2}}(\frac{x+y}{2})
\end{equation*}
where
\begin{eqnarray}\label{tilde_sigma_form}
 \tilde\sigma_T(t) &=& \frac{1}{1-e^{-\kappa_T t}} -\frac{1}{\kappa_T t}\\
\nonumber G_a(x) &=& \frac{1}{2\pi^{3/2}}\int_0^{\infty} \frac{\sin(x\xi+\tfrac{\xi^3}{12}-\tfrac{a^2}{\xi}+\tfrac{\pi}{4})}{\sqrt{\xi}}d\xi.
\end{eqnarray}
\mbox{}
\item The {\em cosecant kernel formula},
\begin{equation*}
F_T(s)= \int_{\mathcal{\tilde C}} e^{-\tilde\mu}\det(I-K^{\csc}_{a})_{L^2(\tilde\Gamma_{\eta})} \frac{d\tilde\mu}{\tilde\mu},
\end{equation*}
where the contour $\mathcal{\tilde C}$, the contour $\tilde\Gamma_{\eta}$ and the operator $K_a^{\csc}$ are defined in Definition \ref{thm_definitions}.
\end{enumerate}
\end{theorem}

The proof of the theorem relies on the explicit limit calculation for the weakly asymmetric simple exclusion process (WASEP) contained in Theorem \ref{epsilon_to_zero_theorem},
as well as the relationship between WASEP and the stochastic heat equation stated in Theorem \ref{BG_thm}. Combining these two theorems proves the cosecant kernel formula. The alternative versions of the formula are proved in Section \ref{kernelmanipulations}

We also have the following representation for the Fredholm determinant involved in the above theorem. One should compare this result to the formula for the GUE Tracy-Widom distribution given in terms of the Painlev\'{e} II equation (see \cite{TW0,TWAiry} or the discussion of Section \ref{integro_differential}).

\begin{proposition}\label{prop2}
Let $\sigma_{T,\tilde\mu}$ be as in (\ref{airy_like_kernel}).  Then
\begin{eqnarray*}
 \nonumber\frac{d^2}{dr^2} \log{\det}(I - K_{\sigma_{T,\tilde\mu}})_{L^2(r,\infty)} &=& -\int_{-\infty}^{\infty} \sigma^{\prime}_{T,\tilde\mu}(t) q_t^2(r) dt\\
 \det(I - K_{\sigma_{T,\tilde\mu}})_{L^2(r,\infty)} &=& \exp\left(-\int_r^{\infty}(x-r)\int_{-\infty}^{\infty} \sigma^{\prime}_{T,\tilde\mu}(t) q_t^2(x)dtdx\right)
\end{eqnarray*}
where
\begin{equation*}
\frac{d^2}{d r^2}q_t(r) = \left(r + t + 2\int_{-\infty}^{\infty} \sigma^{\prime}_{T,\tilde\mu}(t) q_t^2(r)dt \right) q_t(r)
\end{equation*}
with $q_t(r) \sim \Ai(t+r)$ as $r\to \infty$ and where $\sigma^{\prime}_{T,\tilde\mu}(t)$ is the derivative of the function in (\ref{airy_like_kernel}).
\end{proposition}
This proposition is proved in Section \ref{integro_differential} and follows from a more general theory developed in Section \ref{int_int_op_sec} about a class of generalized integrable integral operators.

 It is not hard to show from the formulas in Theorem \ref{main_result_thm} that
$\lim_{s\to\infty}F_T(s)=~1$, but we do not at the present time know how show directly from the determinental formulas that $\lim_{s\to-\infty}F_T(s)=0$, or even that $F_T$ is non-decreasing in $s$.  However, for each $T$, $\mathcal{F}(T,X)$ is an almost surely finite random variable, and hence we know from the definition (\ref{defofF}) that  $F_T$ is indeed a non-degenerate distribution function. 

The formulas in Theorem \ref{main_result_thm} suggest that in the limit as $T$ goes to infinity, under $T^{1/3}$ scaling, we recover the celebrated $F_{\mathrm{GUE}}$ distribution (sometimes written as $F_2$) which is the GUE Tracy-Widom distribution, i.e., the limiting distribution of the scaled and centered largest eigenvalue in the Gaussian unitary ensemble.
 
\begin{corollary}\label{TW}
As $\lim_{T\nearrow\infty} F_T\left(T^{1/3} s\right)=F_{\mathrm{GUE}}(2^{1/3} s)$
\end{corollary} 

This is most easily seen from the cosecant kernel formula for $F_T(s)$. Formally, as $T$ goes to infinity, the kernel $K_{a}^{\csc}$ behaves as $K_{T^{1/3}s}^{\csc}$ and making a change of variables to remove the $T$ from the exponential argument of the kernel, this approaches the Airy kernel on a complex contour, as given in \cite{TW3} equation (33).  The full proof is given in Section \ref{twasymp}.

An inspection the formula for $F_T$ given in Theorem \ref{main_result_thm} immediately reveals that there is no dependence on $X$.  In fact, one can check directly from (\ref{nine}) that

\begin{proposition}   For each $T\ge 0$, $\mathcal{F}(T, X)$ is stationary in $X$.
\end{proposition}
This is simply because the Brownian bridge transition probabilities are affine transformations of each other. Performing the change of variables, the white noise remains invariant in distribution.  The following conjecture is thus natural:

\begin{conj} For each fixed $T>0$, as $T\nearrow \infty$,
\begin{equation*}
2^{1/3}T^{-1/3}\left(\mathcal{F}(T,T^{2/3}X) + \frac{T}{4!}\right)\to{\mathscr{A}}_2(X)
\end{equation*}
where ${\mathscr{A}}_2(X)$ is the ${\rm Airy}_2$ process (see \cite{PS}).
\end{conj}

Unfortunately, the very recent extensions of the Tracy-Widom formula for ASEP (\ref{TW_prob_equation}) to multipoint distributions \cite{TWmulti} appear not to be conducive to the asymptotic analysis necessary to obtain this conjecture following the program of this article.  Corollary \ref{TW} immediately implies the  convergence of  one point distributions,

\begin{corollary}
$\lim_{T\nearrow \infty} P\left(\frac{\mathcal{F}(T,T^{2/3}X) + \frac{T}{4!}}{T^{1/3}} \leq s\right) = F_{\mathrm{GUE}}(2^{1/3}s)
$.
\end{corollary}

It is elementary to add a temperature $\beta^{-1}$ into the model.  Let
\begin{equation*}
\mathcal{F}_\beta(T,X) =  \log E_{T,X}\left[ :\!\exp\!: \left\{-\beta\int_0^T \dot{\mathscr{W}}(t,b(t)) dt\right\}\right].
\end{equation*}
The corresponding function $\mathcal{Z}_\beta(T,X) = p(T,X) \exp\{ \mathcal{F}_\beta(T,X) \}$  is the solution of  $\partial_T \mathcal{Z}_\beta =\frac12\partial^2_X \mathcal{Z}_\beta  -\beta \dot{\mathscr{W}} \mathcal{Z}_\beta $ with $\mathcal{Z}_\beta(0,X) = \delta_0(X)$ and hence
\begin{equation*}
\mathcal{Z}_\beta(T,X)\stackrel{\rm distr.}{=} \beta^{2} \mathcal{Z}(\beta^{4} T, \beta^{2} X),
\end{equation*}giving the relationship
\begin{equation*}
\beta\sim T^{1/4},
\end{equation*}
between the time $T$ and the temperature $\beta^{-1}$.
Now, just as in (\ref{zed}) we define $\mathcal{F}_{\beta}(T,X)$ in terms of $\mathcal{Z}_{\beta}(T,X)$ and $p(T,X)$. 
From this we see that 
\begin{equation*}
 \mathcal{F}_{\beta}(T,X)\stackrel{\rm distr.}{=} \mathcal{F}(\beta^4 T,\beta^2 X). 
\end{equation*}
Hence the following result about the low temperature limit is, just like Corollary \ref{TW}, a consequence of Theorem \ref{main_result_thm}:
\begin{corollary}  For each fixed $X\in \mathbb R$ and $T>0$,
\begin{equation*}
 \lim_{\beta\rightarrow \infty} P\left( \frac{\mathcal{F}_{\beta}(T,\beta^{2/3}T^{2/3}X) + \frac{\beta^4 T}{4!}}{\beta^{4/3} T^{1/3}}\leq s\right) = F_{\mathrm{GUE}}(2^{1/3}s).
\end{equation*}
\end{corollary}

Now we turn to the behavior as $T$ or $\beta\searrow0$.

\begin{proposition}\label{gauss_limit_thm}
As $T\beta^{4} \searrow0$, 
\begin{equation*}
2^{1/2}\pi^{-1/4}\beta^{-1}T^{-1/4}\mathcal{F}_\beta(T,X) 
\end{equation*}
converges in distribution to a standard Gaussian.
\end{proposition}
This proposition is proved in Section \ref{Gaussian_asymptotics}. As an application, taking $\beta=1$ the above theorem shows that 
\begin{equation*}
\lim_{T\searrow 0} F_T (2^{-1/2}\pi^{1/4} T^{1/4} s) = \int_{-\infty}^s \frac{ e^{ - x^2/2} }{\sqrt{2\pi}} dx. 
\end{equation*}
Proposition \ref{gauss_limit_thm} and Corollary \ref{TW} show that, under appropriate scalings, the family of distributions $F_T$ {\it cross over} from the Gaussian distribution for small $T$ to the GUE Tracy-Widom distribution for large $T$.

The main physical prediction (\ref{five2}) is based on the exact computation \cite{K},
\begin{equation}\label{phys101}
\lim_{T\to \infty} T^{-1} \log E[ \mathcal{Z}^n(T,0) ] =  \frac{1}{4!} n(n^2-1),
\end{equation}
which can be performed rigorously \cite{BC} by expanding the Feynman-Kac formula
(\ref{FK}) for $\mathcal{Z}(T,0)$ into an  expectation over $n$ independent copies (replicas)
of the Brownian bridge.  In the physics literature, the computation is done by 
noting that the correlation functions
\begin{equation}
E[ \mathcal{Z}(T,X_1)\cdots\mathcal{Z}(T,X_n) ] 
\end{equation}
can be computed using the Bethe ansatz \cite{LL} for a system of particles on the line
interacting via an attractive delta function potential.
(\ref{phys101}) suggests, but does not imply, the scaling (\ref{five2}).  The problem is that the moments in (\ref{phys101}) 
grow far too quickly to uniquely determine the underlying distribution.  It is interesting
to note that the Tracy-Widom formula for ASEP (\ref{TW_prob_equation}), which is our main tool,
is also based on the same idea that hard core interacting systems in one dimension can be
rigorously solved via the Bethe ansatz. H. Spohn has pointed out, however, that the analogy is at best partial
because the interaction is attractive in the case of the $\delta-$Bose gas.

The probability distribution for the free energy of the continuum directed random polymer, as well as for the solution to the stochastic heat equation and the KPZ equation has been a subject of interest for many years, with a large physics literature (see, for example, \cite{Calabrese-et-al}, \cite{Dotsenko}, \cite{KolKor}  and references therein.)  The reason why we can now compute the distribution is because of the exact formula of Tracy and Widom for the asymmetric simple exclusion process (ASEP) with step initial condition. Once we observe that the weakly asymmetric simple exlusion process (WASEP) with these initial conditions converges to the solution of the stochastic heat equation with delta initial conditions, the calculation of the probability distribution boils down to a careful asymptotic analysis of the Tracy-Widom ASEP formula. This connection is made in Theorem \ref{BG_thm} and the WASEP asymptotic analysis is recorded by Theorem \ref{epsilon_to_zero_theorem}.

\begin{remark} During the preparation of this article, we learned that T. Sasamoto and H. Spohn \cite{SaSp1,SaSp2,SaSp3} independently obtained a formula equivalent  to (\ref{intintop}) for the distribution function $F_T$.  They also use a steepest descent analysis on the Tracy-Widom ASEP formula.  Note that their argument is at the level of formal asymptotics of operator kernels and they have not attempted a mathematical proof. Very recently, two groups of physicists (\cite{Calabrese-et-al}, \cite{Dotsenko}) have successfully employed the Bethe Ansatz for the attractive $\delta-$Bose gas and the replica trick to rederive the formulas for  $F_T$. 
These methods are non-rigorous, employing divergent series.  However, they suggest a deeper relationship between the work of Tracy and Widom for ASEP and the traditional methods of the Bethe Ansatz for the $\delta-$Bose gas.\end{remark}

\subsubsection{Outline} 
There are three main results in this paper. The first pertains to the KPZ / stochastic heat equation / continuum directed polymer and is contained in the theorems and corollaries above in Section \ref{KPZ_SHE_Poly}. The proof of the equivalence of the formulas of Theorem \ref{main_result_thm} is given in Section \ref{kernelmanipulations}.  The Painlev\'{e} II like formula of Proposition \ref{prop2} is proved in Section \ref{integro_differential} along with the formulation of a general theory about a class of generalized integrable integral operators. The other results of the above section are proved in Section \ref{corollary_sec}.
The second result is about the WASEP. In Section \ref{asepscalingth} we introduce the fluctuation scaling theory of the ASEP and motivate the second main result which is contained in Section \ref{WASEP_crossover}. The Tracy-Widom ASEP formula is reviewed in Section \ref{TW_ASEP_formula_sec} and then a formal explanation of the result is given in Section \ref{formal_calc_subsec}. A full proof of this result is contained in Section \ref{epsilon_section} and its various subsections.
The third result is about the connection between the first (stochastic heat equation, etc.) and second (WASEP) and is stated in Section \ref{WASEP_SHE} and proved in Section \ref{BG}. 

\subsection{ASEP scaling theory}\label{asepscalingth}
The simple exclusion process with parameters $p,q\geq0$ (such that $p+q=1$) is a continuous time Markov process on the discrete lattice $\Z$ with state space $\{0,1\}^{\Z}$.  The 1's are thought of as particles and the 0's as holes. The dynamics for this process are given as follows: Each particle has an independent exponential alarmclock which rings at rate one. When the alarm goes off the particle flips a coin and with probability $p$ attempts to jump one site to the right and with probability $q$ attempts to jump one site to the left. If there is a particle at the destination, the jump is suppressed and the alarm is reset (see \cite{Liggett} for a rigorous construction of the process). If $q=1,p=0$ this process is the totally asymmetric simple exclusion process (TASEP); if $q>p$ it is the asymmetric simple exclusion process (ASEP); if $q=p$ it is the  symmetric simple exclusion process (SSEP). Finally, if we introduce a parameter into the model, we can let $q-p$ go to zero with that parameter, and then this class of processes are known as the  weakly asymmetric simple exclusion process (WASEP). It is the WASEP that is of central interest to us. ASEP is often thought of as a discretization of KPZ (for the height function) or stochastic Burgers (for the particle density).  For WASEP the connection can be made precise (see Sections \ref{WASEP_SHE} and \ref{BG}).

There are many ways to initialize these exclusion processes (such as stationary, flat, two-sided Bernoulli, etc.) analogous to the various initial conditions for KPZ/Stochastic Burgers. We consider a very simple initial condition known as {\it step initial condition}  where every positive integer lattice site (i.e. $\{1,2,3,\ldots\}$) is initially occupied by a particle and every other site is empty. Associated to the ASEP are occupation variables $\eta(t,x)$ which equal 1 if there is a particle at position $x$ at time $t$ and 0 otherwise. From these we define $\hat{\eta}=2\eta-1$ which take values $\pm1$ and define the height function for WASEP with asymmetry $\gamma=q-p$ by,
\begin{equation}\label{defofheight}
 h_{\gamma}(t,x) = \begin{cases}
                2N(t) + \sum_{0<y\leq x}\hat{\eta}(t,y), & x>0,\\
	  	2N(t), & x=0,\\
		2N(t)-  \sum_{x<y\leq 0}\hat{\eta}(t,y), & x<0,
               \end{cases}
\end{equation}
where $N(t)$ is equal to the net number of particles which crossed from the site 1 to the site 0 in time $t$. Since we are dealing with step initial conditions $h_{\gamma}$ is initially given by (connecting the points with slope $\pm 1$ lines) $h_{\gamma}(0,x)=|x|$.  It is easy to show that because of step initial conditions, the following three events are equivalent:
\begin{equation*}
\left\{h_{\gamma}(t,x)\geq 2m-x\right\}  =\{\tilde J_{\gamma}(t,x)\geq m\} = \{\mathbf{x}_{\gamma}(t,m)\leq x)
\end{equation*}
where $\mathbf{x}_{\gamma}(t,m)$ is the location at time $t$ of the particle which started at $m>0$ and where $\tilde J_{\gamma}(t,x)$ is a random variable which records the number of particles which started to the right of the origin at time 0 and ended to the left or at $x$ at time $t$. For this particular initial condition $\tilde J_{\gamma}(t,x) = J_{\gamma}(t,x) + x\vee 0$ where $J_{\gamma}(t,x)$ is the usual time integrated current which measures the signed number of particles which cross the bond $(x,x+1)$ up to  time $t$ (positive sign for jumps from $x+1$ to $x$ and negative for jumps from $x$ to $x+1$). The $\gamma$ throughout emphasizes the strength of the asymmetry.

In the case of the ASEP ($q>p$, $\gamma\in (0,1)$) and the TASEP ($q=1,p=0$, $\gamma=1$) there is a well developed fluctuation theory for the height function. We briefly review this,  since it  motivates the time/space/fluctuation scale we will use throughout the paper, and also since we are ultimately interested in understanding the transition in behaviour from WASEP to ASEP.

The following result was proved for $\gamma=1$ (TASEP) by Johansson \cite{J} and for $0<\gamma<1$ (ASEP) by Tracy and Widom \cite{TW3}:
\begin{equation*}
\lim_{t\rightarrow \infty} P\left(\frac{h_{\gamma}(\frac{t}{\gamma},0)-\frac{1}{2}t}{t^{1/3}} \geq -s\right) =F_{\mathrm{GUE}}(2^{1/3}s).
\end{equation*}

In the case of  TASEP, the one point distribution limit has been extended to a process level limit. Consider a time $t$, a space scale of order $t^{2/3}$ and a fluctuation scale of order $t^{1/3}$. Then, as $t$ goes to infinity, the spatial fluctuation process, scaled by $t^{1/3}$ converges to the $\textrm{Airy}_2$ process (see \cite{Borodin:Ferrari,Corwin:Ferrari:Peche} for this result for TASEP, \cite{JDTASEP} for DTASEP and \cite{PS} for the closely related PNG model). Precisely, for $m\geq 1$ and real numbers $x_1,\ldots,x_m$ and $s_1,\ldots,s_m$:
\begin{equation*}
\lim_{t\rightarrow \infty} P\left(h_{\gamma}(t,x_k t^{2/3})\geq \frac{1}{2}t +(\frac{x_k^2}{2}-s_k)t^{1/3}, ~k\in [m]\right) = P\left(\mathcal{A}_2(x_k)\leq 2^{1/3}s_k, ~k\in [m]\right)
\end{equation*}
where $[m]=\{1,\ldots, m\}$ and where $\mathcal{A}_2$ is  the $\textrm{Airy}_2$ process (see, for example, \cite{Borodin:Ferrari,Corwin:Ferrari:Peche}) and has one-point marginals $F_{\mathrm{GUE}}$. In \cite{JDTASEP}, it is proved that this process has a continuous version and that (for DTASEP) the above convergence can be strengthened due to tightness. Notice that in order to get this process limit, we needed to deal with the parabolic curvature of the height function above the origin by including $(\frac{x_k^2}{2}-s_k)$ rather than just $-s_k$. In fact, if one were to replace $t$ by $t T$ for some fixed $T$, then the parabola would become $\frac{x_k^2}{2T}$. We shall see that this parabola comes up again soon.

An important take away from the result above is the relationship between the exponents for time, space and fluctuations --- their $3:2:1$ ratio. It is only with this ratio that we encounter a non-trivial limiting spatial process. For the purposes of this paper, it is more convenient for us to introduce a parameter $\e$ which goes to zero, instead of the parameter $t$ which goes to infinity.

Keeping in mind the $3:2:1$ ratio of time, space and fluctuations we define scaling variables
\begin{equation*}
 t=\e^{-3/2}T, \qquad x=\e^{-1}X,
\end{equation*}
where $T>0$ and $X\in \R$.  With these variables the height function fluctuations around the origin are written as 
\begin{equation*}
\e^{1/2}\left(h_{\gamma}(\tfrac{t}{\gamma},x)-\tfrac{1}{2}t\right).
\end{equation*}
Motivated by the relationship we will establish in Section \ref{WASEP_SHE}, we are interested in studying the Hopf-Cole transformation of the height function fluctuations given by 
\begin{equation*}\exp\left\{-\e^{1/2}\left(h_{\gamma}(\tfrac{t}{\gamma},x)-\tfrac{1}{2}t\right)\right\}.
\end{equation*} 
When $T=0$ we would like this transformed object to become, in some sense, a delta function at $X=0$. Plugging in $T=0$ we see that the height function is given by $|\e^{-1}X|$ and so the exponential becomes $\exp\{-\e^{-1/2}|X|\}$. If we introduce a factor of $\e^{-1/2}/2$ in front of this, then the total integral in $X$ is 1 and this does approach a delta function as $\e$ goes to zero. Thus we consider
\begin{equation}\label{above_eqn}
 \frac{\e^{-1/2}}{2} \exp\left\{-\e^{1/2}\left(h_{\gamma}(\tfrac{t}{\gamma},x)-\tfrac{1}{2}t\right)\right\}.
\end{equation}
As we shall explain in Section \ref{WASEP_crossover}, the correct scaling of $\gamma$ to see the crossover behaviour between  ASEP and SSEP  is  $\gamma=b\e^{1/2}$. We can set $b=1$, as other values of $b$ can be recovered by scaling. This corresponds with setting
\begin{equation*}
\gamma=\e^{1/2},\qquad p=\tfrac{1}{2}-\tfrac{1}{2}\e^{1/2}, \qquad q=\tfrac{1}{2}+\tfrac{1}{2}\e^{1/2}. 
\end{equation*}
Under this scaling, the WASEP is related to the KPZ equation and stochastic heat equation. To help facilitate this connection, define
\begin{eqnarray*}
\label{nu} \nu_\eps =& p+q-2\sqrt{qp} &= \tfrac{1}{2} \eps + \tfrac{1}{8}\epsilon^2+ O(\eps^3),\\
 \label{lambda} \lambda_\eps =& \tfrac{1}{2} \log (q/p) &= \epsilon^{1/2} + \tfrac{1}{3}\epsilon^{3/2}+O(\epsilon^{5/2}),
\end{eqnarray*}
and the discrete Hopf-Cole transformed height function
\begin{equation}\label{rescaledheight}
Z_\eps(T,X) ={\scriptstyle\frac12} \eps^{-1/2}\exp\left \{ - \lambda_\eps h_{\gamma}(\tfrac{t}{\gamma},x) + \nu_\eps \eps^{-1/2}t\right\}.
\end{equation}
Observe that this differs from the expression in (\ref{above_eqn}) only to second order in $\e$. This second order difference, however, introduces a shift of $T/4!$ which we will see now. Note that the same factor appears in \cite{BG}. With the connection to the polymer free energy in mind, write
\begin{equation*}
Z_{\e}(T,X) = p(T,X)\exp\{F_{\e}(T,X)\}.
\end{equation*}
where $p(T,X)$ is the heat kernel defined in (\ref{heat_kernel}). This implies that the field should be defined by,
\begin{equation*}
 F_{\e}(T,X) =\log(\e^{-1/2}/2) -\lambda_{\eps}h_{\gamma}(\tfrac{t}{\gamma},x)+\nu_\eps \eps^{-1/2}t+\frac{X^2}{2T}+\log\sqrt{2\pi T}.
\end{equation*}
We are interested in understanding the behavior of  $P(F_{\e}(T,X)\le s)$  as $\e$ goes to zero. This probability can be translated into a probability for the height function, the current and finally the position of a tagged particle:
\begin{eqnarray}\label{string_of_eqns}
 && \qquad P(F_{\e}(T,X)+\tfrac{T}{4!}\leq s)= \\
\nonumber && P\left(\log(\e^{-1/2}/2) -\lambda_{\eps}h_{\gamma}(\tfrac{t}{\gamma},x)+\nu_\eps \eps^{-1/2}t+\frac{X^2}{2T}+\log\sqrt{2\pi T} + \tfrac{T}{4!}\leq s\right)=\\
 \nonumber&& P\left(h_{\gamma}(\tfrac{t}{\gamma},x) \geq \lambda_{\e}^{-1}[-s+\log\sqrt{2\pi T}+\log(\e^{-1/2}/2)+\frac{X^2}{2T}+\nu_\eps \e^{-1/2}t+\tfrac{T}{4!}]\right)=\\
\nonumber&& P\left(h_{\gamma}(\tfrac{t}{\gamma},x) \geq \e^{-1/2}\left[-a+\log(\e^{-1/2}/2)+\frac{X^2}{2T}\right]+\frac{t}{2}\right)=\\
 \nonumber&& P(\tilde J_{\gamma}(\tfrac{t}{\gamma},x) \geq m) = P(\mathbf{x}_{\gamma}(\tfrac{t}{\gamma},m)\leq x),
\end{eqnarray}
where $m$ is defined as
\begin{eqnarray}\label{m_eqn}
 m&=&\frac{1}{2}\left[\e^{-1/2}\left(-a+\log(\e^{-1/2}/2)+\frac{X^2}{2T}\right)+\frac{1}{2}t+x\right]\\
\nonumber a &=& s-\log\sqrt{2\pi T}.
\end{eqnarray}

\subsection{WASEP crossover regime}\label{WASEP_crossover}
We now turn to the question of how $\gamma$ should vary with $\e$.
The simplest heuristic argument is to use the KPZ equation 
\begin{equation*}\partial_T h_\gamma = -\frac{\gamma}2(\partial_Xh_\gamma)^2 + \frac12 \partial_X^2 h_\gamma +  \dot{\mathscr{W}}.
\end{equation*}
as a proxy for its discretization ASEP, and rescale
\begin{equation*}
h_{\ep,\gamma}(t,x) = \eps^{1/2} h_\gamma(t/\gamma,x) 
\end{equation*}
to obtain
\begin{equation*}\partial_t h_{\ep,\gamma} = -\frac{1}2(\partial_xh_{\ep,\gamma})^2 + \frac{\ep^{1/2}\gamma^{-1}}2 \partial_x^2 h_{\ep,\gamma} +  \eps^{1/4}\gamma^{-1/2} \dot{\mathscr{W}},
\end{equation*}
from which we conclude that we want $\gamma =b\epsilon^{1/2}$ for some $b\in (0,\infty)$.  We expect Gaussian behavior as $b\searrow 0$ and $F_{GUE}$ behavior as $b\nearrow \infty$.  On the other hand, a simple rescaling reduces everything to the case $b=1$. Thus it suffices to consider 
 \begin{equation*}
 \gamma:=\e^{1/2}.
 \end{equation*} 
From now on we will assume that $\gamma=\e^{1/2}$ unless we state explicitly otherwise. In particular, $F_{\e}(T,X)$ should be considered with respect to $\gamma$ as defined above.

The following theorem is proved in Section \ref{epsilon_section} though an informative but non-rigorous derivation is given in Section \ref{formal_calc_subsec}.

\begin{theorem}\label{epsilon_to_zero_theorem}
For all $s\in \R$, $T>0$ and $X\in\R$ we have the following convergence:
\begin{equation}\label{intintop}
F_T(s):=\lim_{\e\rightarrow 0}P(F_\e(T,X)+\tfrac{T}{4!}\leq s) = \int_{\mathcal{\tilde C}} e^{-\tilde\mu}\det(I-K^{\csc}_{a})_{L^2(\tilde\Gamma_{\eta})} \frac{d\tilde\mu}{\tilde\mu},
\end{equation}
where $a=a(s)$ is given as in the statement of Theorem \ref{main_result_thm} and where the contour $\mathcal{\tilde C}$, the contour $\tilde\Gamma_{\eta}$ and the operator $K_a^{\csc}$ is defined below in Definition \ref{thm_definitions}.
\end{theorem}
\begin{remark}
The limiting distribution function $F_T(s)$ above is, a priori, unrelated to the crossover distribution function (notated suggestively as $F_T(s)$ too) defined in Theorem \ref{main_result_thm} which pretains to KPZ, etc., and not to WASEP. Theorem \ref{BG_thm} below, however, establishes that these two distribution function definitions are, in fact, equivalent.
\end{remark}  
\begin{definition}\label{thm_definitions}
The contour $\mathcal{\tilde C}$ is defined as
\begin{equation*}
\mathcal{\tilde C}=\{e^{i\theta}\}_{\pi/2\leq \theta\leq 3\pi/2} \cup \{x\pm i\}_{x>0}
\end{equation*}
The contours $\tilde\Gamma_{\eta}$, $\tilde\Gamma_{\zeta}$ are defined as
\begin{eqnarray*}
\tilde\Gamma_{\eta}&=&\{\frac{c_3}{2}+ir: r\in (-\infty,\infty)\}\\
\tilde\Gamma_{\zeta}&=&\{-\frac{c_3}{2}+ir: r\in (-\infty,\infty)\},
\end{eqnarray*}
where the constant $c_3$ is defined henceforth as
\begin{equation*}
 c_3=2^{-4/3}.
\end{equation*}
The kernel $K_a^{\csc}$ acts on the function space $L^2(\tilde\Gamma_{\eta})$ through its kernel:
\begin{equation}\label{k_csc_definition}
 K_a^{\csc}(\tilde\eta,\tilde\eta') = \int_{\tilde\Gamma_{\zeta}} e^{-\frac{T}{3}(\tilde\zeta^3-\tilde\eta'^3)+2^{1/3}a(\tilde\zeta-\tilde\eta')}   \left(2^{1/3}\int_{-\infty}^{\infty} \frac{\tilde\mu e^{-2^{1/3}t(\tilde\zeta-\tilde\eta')}}{e^{t}-\tilde\mu}dt\right) \frac{d\tilde\zeta}{\tilde\zeta-\tilde\eta}.
\end{equation}
\end{definition}
\begin{remark}
It is very important to observe that our choice of contours for $\tilde\zeta$ and $\tilde\eta'$ ensure that $\re(-2^{1/3}(\tilde\zeta-\tilde\eta'))=1/2$. This ensures that the integral in $t$ above converges for all $\tilde\zeta$ and $\tilde\eta'$. In fact, the convergence holds as long as we keep $\re(-2^{1/3}(\tilde\zeta-\tilde\eta'))$  in a closed subset of $(0,1)$. The inner integral in (\ref{k_csc_definition}) can be evaluated and we find that following equivalent expression:
\begin{equation*}
 K_a^{\csc}(\tilde\eta,\tilde\eta') = \int_{\tilde\Gamma_{\zeta}} e^{-\frac{T}{3}(\tilde\zeta^3-\tilde\eta'^3)+2^{1/3}a(\tilde\zeta-\tilde\eta')} \frac{\pi 2^{1/3} (-\tilde\mu)^{-2^{1/3}(\tilde\zeta-\tilde\eta')}}{ \sin(\pi 2^{1/3}(\tilde\zeta-\tilde\eta'))} \frac{d\tilde\zeta}{\tilde\zeta-\tilde\eta}.
\end{equation*}
This serves as an analytic extension of the first kernel to a larger domain of $\tilde\eta$, $\tilde\eta'$ and $\tilde\zeta$. We do not, however, make use of this analytic extension, and simply record it as a matter of interest.
\end{remark}

\subsection{The connection between WASEP and the stochastic heat equation}\label{WASEP_SHE}
We now state a result about the convergence of the $Z_{\e}(T,X)$ from (\ref{rescaledheight}) to the solution $\mathcal{Z}(T,X)$ of the stochastic heat equation 
(\ref{she}) with delta initial data (\ref{sheinit}).
First we take the opportunity to state (\ref{she}) precisely:  $\mathscr{W}(T)$, $T \ge 0$ is the cylindrical Wiener process, i.e. the continuous Gaussian process taking values in  $H^{-1/2-}_{\rm loc} (\mathbb R)=\cap_{\alpha<-1/2} H^{\alpha}_{\rm loc}(\mathbb R)$  with
\begin{equation*}
E[ \langle \varphi ,\mathscr{W}(T)\rangle \langle \psi, \mathscr{W}(S)\rangle ] =\min(T,S) \langle \varphi, \psi\rangle
\end{equation*}
for any $\varphi,\psi\in C_c^\infty(\mathbb R)$, the smooth functions with compact support in $\mathbb R$.
Here $H^{\alpha}_{\rm loc}(\mathbb R)$, $\alpha<0$,  consists of
distributions $f$ such that for any $\varphi\in C_c^\infty(\mathbb R)$,
$\varphi f$ is in the standard Sobolev space $H^{-\alpha}(\mathbb R)$, i.e. the dual of $H^{\alpha}(\mathbb R)$ under the $L^2$ pairing.
$H^{-\alpha}(\mathbb R)$ is the closure of $C_c^\infty(\mathbb R)$ under
the norm $\int_{-\infty}^{\infty} (1+ |t|^{-2\alpha}) |\hat f(t)|^2dt$ where $\hat f$ denotes
the Fourier transform.
The distributional time derivative $\dot{\mathscr{W}}(T,X)$
 is the space-time white noise,\begin{equation*}
E[ \dot{\mathscr{W}}(T,X) \dot{\mathscr{W}}(S,Y)] = \delta(T-S)\delta(Y-X).
\end{equation*}
  Note the mild abuse of notation for the sake of clarity; we write $\dot{\mathscr{W}}(T,X)$ even though it is a distribution on $(T,X)\in [0,\infty)\times \mathbb R$ as opposed to  a classical function of $T$ and $X$.
 Let $\mathscr{F}(T)$, $T\ge 0$, be the natural filtration, i.e. the smallest $\sigma$-field
 with respect to which  $\mathscr{W}(S)$ are measurable for all $0\le S\le T$.
 
 The stochastic heat equation is then shorthand for its integrated version (\ref{intshe})
 where the stochastic integral is interpreted in the It\^o sense \cite{W}, so that, in particular,
 if $f(T,X)$ is any non-anticipating integrand,
 \begin{eqnarray*}&
 E[ (\int_0^T \int_{-\infty}^\infty f(S,Y)\mathscr{W}(dY, dS) )^2]=E[ (\int_0^T \int_{-\infty}^\infty f^2(S,Y)dY dS].&
 \end{eqnarray*}
 The  awkward notation is inherited from stochastic partial differential equations: $\mathscr{W}$ for (cylindrical) Wiener process, 
 $\dot{\mathscr{W}}$ for white noise, and stochastic integrals are taken with respect to white noise $\mathscr{W}(dY, dS)$.
  
 Note that the solution can be written explicitly as a series of multiple Wiener integrals.  With $X_0=0$ and $X_{n+1}=X$,
 \begin{equation}\label{soln}
 \mathcal{Z}(T,X)= \sum_{n=0}^\infty (-1)^n \int_{\Delta'_n(T)} \int_{\mathbb R^n}\prod_{i=0}^n
p(T_{i+1}-T_{i}, X_{i+1}-X_{i}) \prod_{i=1}^n\mathscr{W} (dT_i dX_i) \end{equation}
 where $\Delta'_n(T)= \{(T_1,\ldots,T_n) : 0=T_0 \le T_1\le \cdots \le T_n\le T_{n+1}=T\}$.
  
Returning now to the WASEP, the random functions $Z_{\e}(T,X)$ from (\ref{rescaledheight}) have discontinuities both in 
space and in time.  If desired, one can linearly interpolate in space so that they become 
a jump process taking values in the space of continuous functions.  But it does not really make things easier.  The key point is that the jumps are small, so we use instead  the space $D_u([0,\infty); D_u(\mathbb R))$ where $D_u$ refers to right continuous paths with left limits with the topology  of uniform convergence on compact sets.  Let $\mathscr{P}_\e$ denote the probability measure on $D_u([0,\infty); D_u(\mathbb R))$ corresponding to the process  $Z_{\e}(T,X)$.

\begin{theorem}\label{BG_thm}
$\mathscr{P}_\e$, $\ep\in (0,1/4)$ are a tight family of measures and the unique limit point is supported on $C([0,\infty); C(\mathbb R))$ and corresponds to the solution (\ref{soln}) of the stochastic heat equation (\ref{she}) with delta function initial conditions (\ref{sheinit}). \end{theorem}
In particular, for each fixed $X,T$ and $s$,
\begin{equation*}
\lim_{\e\searrow 0} P( F_{\e}(T,X)\le s) = P(\mathcal{F}(T,X)\le s) .
\end{equation*}

The result is motivated by, but does not follow directly from, the results of \cite{BG}. This is because of the delta function initial conditions, and the consequent difference in the scaling.  It requires a certain amount of work to show that their basic computations are applicable to the present case.  This is done in Section \ref{BG}.

\subsection{The Tracy-Widom Step Initial Condition ASEP formula}\label{TW_ASEP_formula_sec}
Due to the process level convergence of WASEP to the stochastic heat equation, exact information about WASEP can be, with care, translated into information about the stochastic heat equation. Until recently, very little exact information was known about ASEP or WASEP. The work of Tracy and Widom in the past few years, however, has changed the situation significantly. The key tool in determining the limit as $\e$ goes to zero of $P(F_{\e}(T,X)+\tfrac{T}{4!}\leq s)$ is their  exact formula for the transition probability for a tagged particle in ASEP started from step initial conditions. This formula was stated in \cite{TW3} in the form below, and was developed in the three papers \cite{TW1,TW2,TW3}. We will apply it to the last line of  (\ref{string_of_eqns}) to give us an exact formula for $P(F_{\e}(T,X)+\tfrac{T}{4!}\le s)$.

Recall that $\mathbf{x}_{\gamma}(t,m)$ is the location at time $t$ of the particle which started at $m>0$. Consider $q>p$ such that $q+p=1$ and let $\gamma=q-p$ and $\tau=p/q$. For $m>0$, $t\geq 0$ and $x\in \Z$, it is shown in \cite{TW3} that,
\begin{equation}\label{TW_prob_equation}
P(\mathbf{x}(\gamma^{-1}t,m)\leq x) = \int_{S_{\tau^+}}\frac{d\mu}{\mu} \prod_{k=0}^{\infty} (1-\mu\tau^k)\det(I+\mu J_{t,m,x,\mu})_{L^2(\Gamma_{\eta})}
\end{equation}
where $S_{\tau^+}$ is a circle centered at zero of radius strictly between $\tau$ and 1, and where the kernel of the Fredholm determinant (see Section \ref{pre_lem_ineq_sec}) is given by
\begin{equation}\label{J_eqn_def}
J_{t,m,x,\mu}(\eta,\eta')=\int_{\Gamma_{\zeta}} \exp\{\Psi_{t,m,x}(\zeta)-\Psi_{t,m,x}(\eta')\}\frac{f(\mu,\zeta/\eta')}{\eta'(\zeta-\eta)}d\zeta
\end{equation}
where $\eta$ and $\eta'$ are on $\Gamma_{\eta}$, a circle centered at zero of radius $\rho$ strictly between $\tau$
and $1$, and the $\zeta$ integral is on $\Gamma_{\zeta}$, a circle centered at zero of radius strictly between $1$ and $\rho\tau^{-1}$ (so as to ensure that $|\zeta/\eta|\in (1,\tau^{-1})$), and where, for fixed $\xi$,
\begin{eqnarray*}
\nonumber f(\mu,z)&=&\sum_{k=-\infty}^{\infty} \frac{\tau^k}{1-\tau^k\mu}z^k,\\
\Psi_{t,m,x}(\zeta) &=& \Lambda_{t,m,x}(\zeta)-\Lambda_{t,m,x}(\xi),\\
 \nonumber \Lambda_{t,m,x}(\zeta) &=& -x\log(1-\zeta) + \frac{t\zeta}{1-\zeta}+m\log\zeta
 .
\end{eqnarray*}
\begin{remark}
 Throughout the rest of the paper we will only include the subscripts on $J$, $\Psi$ and $\Lambda$ when we want to emphasize their dependence  on a given variable. 
\end{remark}

\subsection{The weakly asymmetric limit of the Tracy-Widom ASEP formula}\label{formal_calc_subsec}

The Tracy and Widom ASEP formula (\ref{TW_prob_equation}) provides an exact expression for the probability $P(F_{\e}(T,X)+\tfrac{T}{4!}\leq s)$ by interpreting it in terms of a probability of the location of a tagged particle (\ref{string_of_eqns}). It is of great interest to understand this limit ($F_T(s)$) since, as we have seen, it describes a number of interesting limiting objects.

We will now present a formal computation of the expressions given in Theorem \ref{epsilon_to_zero_theorem} (see Section \ref{WASEP_crossover}) for  $F_T(s)$. After presenting the formal argument, we will stress that there are a number of very important technical points which arise during this argument, many of which require serious work to resolve. In Section \ref{epsilon_section} we will provide a rigorous proof of Theorem \ref{epsilon_to_zero_theorem}  in which we deal with all of the possible pitfalls.

\begin{definition}\label{quantity_definitions}
Recall the definitions for the relevant quantities in this limit:
\begin{eqnarray*}
&& p=\frac{1}{2}-\frac{1}{2}\e^{1/2},\qquad q=\frac{1}{2}+\frac{1}{2}\e^{1/2}\\
&& \gamma=\e^{1/2},\qquad \tau=\frac{1-\e^{1/2}}{1+\e^{1/2}}\\
&& x=\e^{-1}X,\qquad t=\e^{-3/2}T\\
&& m=\frac{1}{2}\left[\e^{-1/2}\left(-a+\log(\e^{-1/2}/2)+\frac{X^2}{2T}\right)+\frac{1}{2}t+x\right]\\
&& \left\{F_{\e}(T,X)+\frac{T}{4!}\leq s\right\} = \left\{\mathbf{x}(\frac{t}{\gamma},m)\leq x\right\},
\end{eqnarray*}
where $a=a(s)$ is defined in the statement of Theorem \ref{main_result_thm}. We also define the contours $\Gamma_{\eta}$ and $\Gamma_{\zeta}$ to be
\begin{equation*}
 \Gamma_{\eta}=\{z:|z|=1-\tfrac{1}{2}\e^{1/2}\} \qquad \textrm{and} \qquad \Gamma_{\zeta}=\{z:|z|=1+\tfrac{1}{2}\e^{1/2}\} 
\end{equation*}

\end{definition}
The first term in the integrand of (\ref{TW_prob_equation}) is the infinite product $\prod_{k=0}^{\infty}(1-\mu \tau^k)$. Observe that $\tau\approx 1-2\e^{1/2}$ and that $S_{\tau^+}$, the contour on which $\mu$ lies, is a circle centered at zero of radius between $\tau$ and 1. The infinite product is not well behaved along most of this contour, so we will deform the contour to one along which the product is not highly oscillatory. Care must be taken, however, since the Fredholm determinant has poles at every $\mu=\tau^k$. The deformation must avoid passing through them. Observe now that
\begin{equation*}
 \prod_{k=0}^{\infty}(1-\mu \tau^k) = \exp\{\sum_{k=0}^{\infty} \log(1-\mu \tau^k)\},
\end{equation*}
and that for small $|\mu|$,
\begin{eqnarray}\nonumber
\sum_{k=0}^{\infty} \log(1-\mu (1-2\e^{1/2})^k)& \approx& \e^{-1/2} \int_0^{\infty} \log(1-\mu e^{-2 r}) dr \\ & \approx & \e^{-1/2}\mu \int_0^{\infty} e^{-2 r} dr = -\frac{\e^{-1/2}\mu}{2}.
\end{eqnarray}
With this in mind define
\begin{equation*}
\tilde\mu = \e^{-1/2}\mu,
\end{equation*}
from which we see that if the Riemann sum approximation is reasonable then the infinite product converges to $e^{-\tilde\mu/2}$. We make the $\mu \mapsto \e^{-1/2}\tilde\mu$ change of variables and find that the above approximations are reasonable if we consider a $\tilde\mu$ contour 
\begin{equation*}
\mathcal{\tilde C}_{\e}=\{e^{i\theta}\}_{\pi/2\leq \theta\leq 3\pi/2} \cup \{x\pm i\}_{0<x<\e^{-1/2}-1}.
\end{equation*}
Thus the infinite product goes to $e^{-\tilde\mu/2}$.

Now we turn to the Fredholm determinant. We determine a candidate for the pointwise limit of the kernel.  That the combination of these two pointwise limits gives the actual limiting formula as $\e$ goes to zero is, of course, completely unjustified at this point. Also, the pointwise limits here disregard the existence of a number of singularities encountered during the argument. 

The kernel $J(\eta,\eta')$ is given by an integral and the integrand has three main components: An exponential term
\begin{equation*}
 \exp\{\Lambda(\zeta)-\Lambda(\eta')\},
\end{equation*}
a rational function term (we include the differential with this term for scaling purposes)
\begin{equation*}
\frac{d\zeta}{\eta'(\zeta-\eta)},
\end{equation*}
and the term
\begin{equation*}
\mu f(\mu,\zeta/\eta').
\end{equation*}
We will proceed by the method of steepest descent, so in order to determine the region along the $\zeta$ and $\eta$ contours which affects the asymptotics we consider the exponential term first. The argument of the exponential is given by $\Lambda(\zeta)-\Lambda(\eta')$ where
\begin{equation*}
\Lambda(\zeta)=-x\log(1-\zeta) + \frac{t\zeta}{1-\zeta}+m\log(\zeta),
\end{equation*}
and where, for the moment, we take $m=\frac{1}{2}\left[\e^{-1/2}(-a+\frac{X^2}{2T})+\frac{1}{2}t+x\right]$. The real expression for $m$ has a $\log(\e^{-1/2}/2)$ term which we define in with the $a$ for the moment (recall that $a$ is defined in the statement of Theorem \ref{main_result_thm}).
Recall that $x, t$ and $m$ all depend on $\e$. For small $\e$, $\Lambda(\zeta)$ has a critical point in an $\e^{1/2}$ neighborhood of -1. For purposes of having a nice ultimate answer, we choose to center in on  \begin{equation*}
\xi=-1-2\e^{1/2}\frac{X}{T}.
\end{equation*}
We can rewrite the argument of the exponential as $(\Lambda(\zeta)-\Lambda(\xi))-(\Lambda(\eta')-\Lambda(\xi))=\Psi(\zeta)-\Psi(\eta')$. The idea in \cite{TW3} for extracting asymptotics of this term is  to deform the $\zeta$ and $\eta$ contours to lie along curves such that outside the scale $\e^{1/2}$ around $\xi$, $\re\Psi(\zeta)$ is large and negative, and $\re\Psi(\eta')$ is large and positive. Hence we can ignore those parts of the contours. Then, rescaling around $\xi$ to blow up this $\e^{1/2}$ scale, we obtain the asymptotic exponential term. This final change of variables then sets the scale at which we should analyze the other two terms in the integrand for the $J$ kernel.

Returning to $\Psi(\zeta)$, we make a  Taylor expansion  around $\xi$ and find that in a neighborhood of $\xi$,
\begin{equation*}
 \Psi(\zeta) \approx -\frac{T}{48} \e^{-3/2}(\zeta-\xi)^3 + \frac{a}{2}\e^{-1/2}(\zeta-\xi).
\end{equation*}
This suggests the change of variables,
\begin{equation}\label{change_of_var_eqn}
 \tilde\zeta = 2^{-4/3}\e^{-1/2}(\zeta-\xi) \qquad \tilde\eta' = 2^{-4/3}\e^{-1/2}(\eta'-\xi),
\end{equation}
and likewise for $\tilde\eta$. After this our Taylor expansion takes the form
\begin{equation}\label{taylor_expansion_term}
 \Psi(\tilde\zeta) \approx -\frac{T}{3} \tilde\zeta^3 +2^{1/3}a\tilde\zeta.
\end{equation}
In the spirit of steepest descent analysis, we would like the $\zeta$ contour to leave $\xi$ in a direction where this Taylor expansion is decreasing rapidly. This is accomplished by leaving at an angle $\pm 2\pi/3$. Likewise, since $\Psi(\eta)$ should increase rapidly, $\eta$ should leave $\xi$ at angle $\pm\pi/3$. The $\zeta$ contour was originally centred at zero and of radius $1+\e^{1/2}/2$ and the $\eta$ contour of radius $1-\e^{1/2}/2$. In order to deform these contours without changing the value of the determinant, care must be taken since there are  poles of $f$ whenever $\zeta/\eta'=\tau^k$, $k\in \Z$. We ignore this issue for the formal calculation, and deal with it carefully in Section \ref{epsilon_section} by using different contours.

Let us now assume that we can deform our contours to curves along which $\Psi$ rapidly decays in $\zeta$ and increases in $\eta$, as we move along them away from $\xi$. If we apply the change of variables in (\ref{change_of_var_eqn}), the straight part of our contours become infinite at angles $\pm2\pi/3$ and $\pm \pi/3$ which we call $\tilde\Gamma_{\zeta}$ and $\tilde\Gamma_{\eta}$. Note that this is {\em not} the actual definition of these contours which we use in the statement and proof of Theorem \ref{main_result_thm}  because of the singularity problem mentioned above. 

Applying this change of variables to the kernel of the Fredholm determinant changes the $L^2$ space and hence we must multiply the kernel by the Jacobian term $2^{4/3}\e^{1/2}$. We will include this term with the $\mu f(\mu,z)$ term and take the $\e\to0$  limit of that product.

As noted before, the term $2^{1/3}a\tilde\zeta$ should have been $2^{1/3}(a-\log(\e^{-1/2}/2))\tilde\zeta$ in the Taylor expansion above, giving
\begin{equation*}
 \Psi(\tilde\zeta) \approx -\frac{T}{3} \tilde\zeta^3 +2^{1/3}(a-\log(\e^{-1/2}/2))\tilde\zeta,
\end{equation*}
which would appear to blow up as $\e$ goes to zero. We now show how the extra $\log\e$ in the exponent can be absorbed into the $2^{4/3}\e^{1/2}\mu f(\mu,\zeta/\eta')$ term.
Recall
\begin{equation*}
 \mu f(\mu,z) = \sum_{k=-\infty}^{\infty} \frac{\mu \tau^k}{1-\tau^k \mu}z^k.
\end{equation*}
If we let $n_0=\lfloor \log(\e^{-1/2}) /\log(\tau)\rfloor$, then observe that for $1 < |z| < \tau^{-1}$,
\begin{equation*}
 \mu f(\mu,z) = \sum_{k=-\infty}^{\infty} \frac{ \mu \tau^{k+n_0}}{1-\tau^{k+n_0}\mu}z^{k+n_0} =z^{n_0} \tau^{n_0}\mu \sum_{k=-\infty}^{\infty} \frac{ \tau^{k}}{1-\tau^{k}\tau^{n_0}\mu}z^{k}.
\end{equation*}
By the choice of $n_0$, $\tau^{n_0}\approx \e^{-1/2}$ so 
\begin{equation*}
 \mu f(\mu,z) \approx z^{n_0} \tilde\mu f(\tilde\mu,z).
\end{equation*}
The discussion on the exponential term indicates that it suffices to understand the behaviour of this function when $\zeta$ and $\eta'$ are within $\e^{1/2}$ of $\xi$. Equivalently, letting $z=\zeta/\eta'$, it suffices to understand $ \mu f(\mu,z) \approx z^{n_0} \tilde\mu f(\tilde\mu,z)$ for
\begin{equation*}
 z= \frac{\zeta}{\eta'}=\frac{\xi +  2^{4/3}\e^{1/2}\tilde\zeta}{\xi +  2^{4/3}\e^{1/2}\tilde\eta'}\approx 1-\e^{1/2}\tilde z, \qquad \tilde z=2^{4/3}(\tilde\zeta-\tilde\eta').
\end{equation*}
Let us now consider $z^{n_0}$ using the fact that $\log(\tau)\approx -2\e^{1/2}$:
\begin{equation*}
z^{n_0} \approx (1-\e^{1/2}\tilde z)^{\e^{-1/2}(\frac{1}{4}\log\e)} \approx e^{-\frac{1}{4}\tilde z \log(\e)}.
\end{equation*}
Plugging back in the value of $\tilde z$ in terms of $\tilde\zeta$ and $\tilde\eta'$ we see that this prefactor of $z^{n_0}$ exactly cancels the $\log\e$ term which accompanies $a$ in the exponential.

What remains is to determine the limit of $2^{4/3}\e^{1/2}\tilde\mu f(\tilde\mu, z)$ as $\e$ goes to zero, for $z\approx 1-\e^{1/2} \tilde z$. This can be found by interpreting the infinite sum as a Riemann sum approximation for a certain integral. Define $t=k\e^{1/2}$ and observe that
\begin{equation}\label{Riemann_limit} 
 \e^{1/2}\tilde\mu f(\tilde\mu,z) = \sum_{k=-\infty}^{\infty} \frac{ \tilde\mu \tau^{t\e^{-1/2}}z^{t\e^{-1/2}}}{1-\tilde\mu \tau^{t\e^{-1/2}}}\e^{1/2} \rightarrow \int_{-\infty}^{\infty} \frac{\tilde\mu e^{-2t}e^{-\tilde z t}}{1-\tilde\mu e^{-2t}}dt.
\end{equation}
This used the fact that $\tau^{t\e^{-1/2}}\rightarrow e^{-2t}$ and that $z^{t\e^{-1/2}}\rightarrow e^{-\tilde z t}$, which hold at least pointwise in $t$. For (\ref{Riemann_limit}) to hold , we must have $\re \tilde z$ bounded inside $(0,2)$, but we disregard this difficulty for  the heuristic proof. If we change variables of $t$ to $t/2$ and multiply the top and bottom by $e^{-t}$ then we find that
\begin{equation*}
 2^{4/3}\e^{1/2}\mu f(\mu,\zeta/\eta') \rightarrow 2^{1/3} \int_{-\infty}^{\infty} \frac{\tilde\mu e^{-\tilde zt/2}}{e^{t}-\tilde\mu}dt.
\end{equation*}
As far as the final term, the rational expression, under the change of variables and zooming in on $\xi$, the factor of $1/\eta'$ goes to -1 and the $\frac{d\zeta}{\zeta-\eta'}$ goes to $\frac{d\tilde\zeta}{\tilde\zeta-\tilde\eta'}$.

Thereby we formally obtain from $\mu J$ the kernel $-K_{a'}^{\csc}(\tilde\eta,\tilde\eta')$ acting on $L^2(\tilde\Gamma_{\eta})$, where
\begin{equation*}
 K_{a'}^{\csc}(\tilde\eta,\tilde\eta') = \int_{\tilde\Gamma_{\zeta}} e^{-\frac{T}{3}(\tilde\zeta^3-\tilde\eta'^3)+2^{1/3}a'(\tilde\zeta-\tilde\eta')} \left(2^{1/3}\int_{-\infty}^{\infty} \frac{\tilde\mu e^{-2^{1/3}t(\tilde\zeta-\tilde\eta')}}{e^{t}-\tilde\mu}dt\right) \frac{d\tilde\zeta}{\tilde\zeta-\tilde\eta},
\end{equation*}
with $a'=a+\log2$. Recall that the $\log2$ came from the $\log(\e^{-1/2}/2)$ term.

We have the identity
\begin{equation}\label{cscid}
 \int_{-\infty}^{\infty} \frac{\tilde\mu e^{-\tilde zt/2}}{e^{t}-\tilde\mu}dt =(-\tilde\mu)^{-\tilde z/2}\pi \csc(\pi \tilde z/2),
\end{equation}
where the branch cut in $\tilde\mu$ is taken along the positive real axis, hence $(-\tilde\mu)^{-\tilde z/2} =e^{-\log(-\tilde\mu)\tilde z/2}$ where $\log$ is taken with the standard branch cut along the negative real axis.
We may use the identity to rewrite the kernel as
\begin{equation*}
 K_{a'}^{\csc}(\tilde\eta,\tilde\eta') = \int_{\tilde\Gamma_{\zeta}} e^{-\frac{T}{3}(\tilde\zeta^3-\tilde\eta'^3)+2^{1/3}a'(\tilde\zeta-\tilde\eta')} \frac{\pi 2^{1/3}(-\tilde\mu)^{-2^{1/3}(\tilde\zeta-\tilde\eta')}}{ \sin(\pi 2^{1/3}(\tilde\zeta-\tilde\eta'))} \frac{d\tilde\zeta}{\tilde\zeta-\tilde\eta}.
\end{equation*}
Therefore we have shown formally that
\begin{equation*}
 \lim_{\e\rightarrow 0} P(F_{\e}(T,X)+\tfrac{T}{4!}\leq s) := F_T(s) = \int_{\mathcal{\tilde C}}e^{-\tilde \mu/2}\frac{d\tilde\mu}{\tilde\mu}\det(I-K_{a'}^{\csc})_{L^2(\tilde\Gamma_{\eta})},
\end{equation*}
where $a'=a+\log 2$. To make it cleaner we replace $\tilde\mu/2$  with  $\tilde\mu$. This only affects the $\tilde\mu$ term above given now by $(-2\tilde\mu)^{-\tilde z/2}$$=$$(-\tilde\mu)^{-2^{1/3}(\tilde\zeta-\tilde\eta')} e^{-2^{1/3}\log2(\tilde\zeta-\tilde\eta')}$. This can be absorbed and cancels the $\log2$ in $a'$ and thus we obtain,
\begin{equation*}
F_T(s) = \int_{\mathcal{\tilde C}}e^{-\tilde \mu}\frac{d\tilde\mu}{\tilde\mu}\det(I-K_{a}^{\csc})_{L^2(\tilde\Gamma_{\eta})},
\end{equation*}
which, up to the definitions of the contours $\tilde\Gamma_{\eta}$ and $\tilde\Gamma_{\zeta}$, is the desired limiting formula.

We now briefly note some of the problems and pitfalls of the preceeding formal argument, all of which will be addressed in the real proof of Section \ref{epsilon_section}.

Firstly, the pointwise convergence of both the prefactor infinite product and the Fredholm determinant is certainly not enough to prove convergence of the $\tilde\mu$ integral. Estimates must be made to control this convergence or to show that we can cut off the tails of the $\tilde\mu$ contour at negligible cost and then show uniform convergence on the trimmed contour.

Secondly, the deformations of the $\eta$ and $\zeta$ contours to the steepest descent curves is {\it entirely} illegal, as it involves passing through many poles of the kernel (coming from the $f$ term). In the case of \cite{TW3} this problem could be dealt with rather simply by just slightly modifying the descent curves. However, in our case, since $\tau$ tends to $1$ like $\e^{1/2}$, such a patch is much harder and involves very fine estimates to show that there exists suitable contours which stay close enough together, yet along which $\Psi$ displays the necessary descent and ascent required to make the argument work. This issues also comes up in the convergence of (\ref{Riemann_limit}). In order to make sense of this we must ensure that $1 < |\zeta/\eta'| < \tau^{-1}$ or else the convergence and the resulting expression make no sense.

Finally, one must make precise tail estimates to show that the kernel convergence is in the sense of trace-class norm. The Riemann sum approximation argument  can in fact be made rigorous (following the proof of Proposition \ref{originally_cut_mu_lemma}). We choose, however, to give an alternative proof of the validity of that limit in which we identify and prove the limit of $f$ via analysis of singularities and residues.

\section{Proof of the weakly asymmetric limit of the Tracy-Widom ASEP formula}\label{epsilon_section}

In this section we give a proof of Theorem \ref{epsilon_to_zero_theorem}, for which a formal derivation was presented in Section \ref{formal_calc_subsec}. The heart of the argument is Proposition \ref{uniform_limit_det_J_to_Kcsc_proposition} which is proved in Section \ref{J_to_K_sec} and also relies on a number of technical lemmas. These lemmas as well as all of the other propositions are proved in Section \ref{props_and_lemmas_sec}.

\subsubsection{Proof of Theorem \ref{epsilon_to_zero_theorem}}\label{proof_of_WASEP_thm_sec}

The expression given in equation (\ref{TW_prob_equation}) for $P(F_{\e}(T,X)+\tfrac{T}{4!}\leq s)$ contains an integral over a $\mu$ contour of a product of a prefactor infinite product and a Fredholm determinant. The first step towards taking the limit of this as $\e$ goes to zero is to control the prefactor, $\prod_{k=0}^{\infty} (1-\mu\tau^k)$. Initially $\mu$ lies on a contour $S_{\tau^+}$ which is centered at zero and of radius between $\tau$ and 1.  Along this contour the partial products (i.e., product up to $N$) form a highly oscillatory sequence and hence it is hard to control the convergence of the sequence.

\begin{figure}
\begin{center}
 \includegraphics[scale=.17]{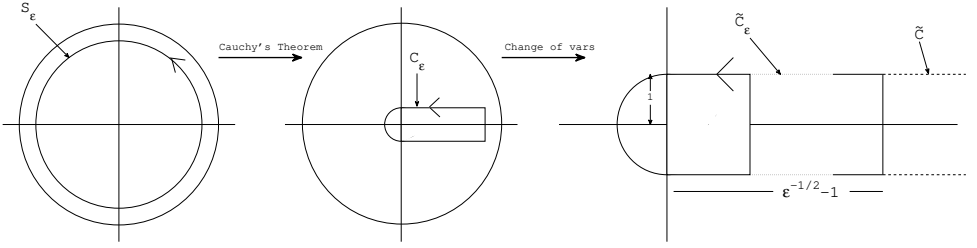}
\caption{The $S_{\e}$ contour is deformed to the $C_{\e}$ contour via Cauchy's theorem and then a change of variables leads to $\tilde{C}_{\e}$, with its infinite extension $\tilde{C}$.}\label{deform_to_c}
\end{center}
\end{figure}

The first step in our proof is to deform the $\mu$ contour $S_{\tau^+}$ to\begin{equation*}
 \mathcal{C}_{\e} = \{\e^{1/2}e^{i\theta}\}\cup \{x\pm i \e^{1/2}\}_{0<x\leq 1-\e^{1/2}}\cup \{1-\e^{1/2}+\e^{1/2}iy\}_{-1<y<1},
\end{equation*}
a long, skinny cigar shaped contour 
(see Fig. \ref{deform_to_c}.)
We orient $ \mathcal{C}_{\e}$ counter-clockwise. Notice that this new contour still includes all of the poles at $\mu=\tau^{k}$ associated with the $f$ function in the $J$ kernel.

In order to justify replacing $S_{\tau^+}$ by $\mathcal{C}_{\e}$ we need the following (for the proof see Section \ref{proofs_sec}):

\begin{lemma}\label{deform_mu_to_C}
In equation (\ref{TW_prob_equation}) we can replace the contour $S_\e$ with $\mathcal{C}_\e$ as the contour of integration for $\mu$ without affecting the value of the integral.
\end{lemma}

Having made this deformation of the $\mu$ contour, we now observe that the natural scale for $\mu$ is on order $\e^{1/2}$. With this in mind we make the change of variables
\begin{equation*}
 \mu = \e^{1/2}\tilde\mu.
\end{equation*}
\begin{remark}
Throughout the proof of this theorem and its lemmas and propositions, we will use the tilde to denote variables which are $\e^{1/2}$ rescaled versions of the original, untilded variables.
\end{remark}

The $\tilde\mu$ variable now lives on the contour 
\begin{equation*}
 \mathcal{\tilde C}_{\e} = \{e^{i\theta}\}\cup \{x\pm i\}_{0<x\leq \e^{-1/2}-1}\cup \{\e^{-1/2}-1+iy\}_{-1<y<1}.
\end{equation*}
which grow and ultimately approach
\begin{equation*}
 \mathcal{\tilde C} = \{e^{i\theta}\}\cup \{x\pm i\}_{x>0}.
\end{equation*}
In order to show convergence of the integral as $\e$ goes to zero, we must consider two things, the convergence of the integrand for $\tilde\mu$ in some compact region around the origin on $\mathcal{\tilde{C}}$, and the controlled decay of the integrand on $\mathcal{\tilde C}_{\e}$ outside of that compact region. This second consideration will allow us to approximate the integral by a finite integral in $\tilde\mu$, while the first consideration will tell us what the limit of that integral is. When all is said and done, we will paste back in the remaining part of the $\tilde\mu$ integral and have our answer.
With this in mind we give the following bound which is proved in Section \ref{proofs_sec},

\begin{lemma}\label{mu_inequalities_lemma}
 Define two regions, depend on a fixed parameter $r\geq 1$,
\begin{eqnarray*}
 R_1 &=& \{\tilde\mu : |\tilde\mu|\leq \frac{r}{\sin(\pi/10)}\}\\
 R_2 &=& \{\tilde\mu : \re(\tilde\mu)\in [\frac{r}{\tan(\pi/10)},\e^{-1/2}], \textrm{ and } \im(\tilde\mu)\in [-2,2]\}.
\end{eqnarray*}
$R_1$ is compact and $R_1 \cup R_2$ contains all of the contour $\mathcal{\tilde{C}}_{\e}$. Furthermore define the function (the infinite product after the change of variables)
\begin{equation*}
 g_{\e}(\tilde\mu) = \prod_{k=0}^{\infty} (1-\e^{1/2}\tilde\mu \tau^k).
\end{equation*}
Then uniformly in $\tilde\mu\in R_1$,
\begin{equation}\label{g_e_ineq1}
g_{\e}(\mu)\to e^{-\tilde\mu/2}
\end{equation}
Also, for all $\e<\e_0$ (some positive constant) there exists a constant $c$ such that for all $\tilde\mu\in R_2$ we have the following tail bound:
\begin{equation}\label{g_e_ineq2}
|g_{\e}(\tilde\mu)| \leq |e^{-\tilde\mu/2}| |e^{-c\e^{1/2}\tilde\mu^2}|.
\end{equation}
(By the choice of $R_2$, for all $\tilde\mu\in R_2$, $\re(\tilde\mu^2)>\delta>0$ for some fixed $\delta$. The constant $c$ can be taken to be $1/8$.)
\end{lemma}

We now turn our attention to the Fredholm determinant term in the integrand. Just as we did for the prefactor infinite product in Lemma \ref{mu_inequalities_lemma} we must establish uniform convergence of the determinant for $\tilde\mu$ in a fixed compact region around the origin, and a suitable tail estimate valid outside that compact region. The tail estimate must be such that for each finite $\e$, we can combine the two tail estimates (from the prefactor and from the determinant) and show that their integral over the tail part of $\mathcal{\tilde C}_{\e}$ is small and goes to zero as we enlarge the original compact region. For this we have the following two propositions (the first is the most substantial and is proved in Section \ref{J_to_K_sec}, while the second is proved in Section \ref{proofs_sec}).

\begin{proposition}\label{uniform_limit_det_J_to_Kcsc_proposition}
Fix $s\in \R$, $T>0$ and $X\in \R$. Then for any compact subset of $\mathcal{\tilde C}$ we have that for all $\delta>0$ there exists an $\e_0>0$ such that for all $\e<\e_0$ and all $\tilde\mu$ in the compact subset,
\begin{equation*}
\left|\det(I+\e^{1/2}\tilde\mu J_{\e^{1/2}\tilde\mu})_{L^2(\Gamma_{\eta})} - \det(I-K^{\csc}_{a'})_{L^2(\tilde\Gamma_{\eta})}\right|<\delta.
\end{equation*}
Here $a'=a+\log2$ and $K_{a'}^{\csc}$ is defined in Def. \ref{k_csc_definition} and depends implicitly on $\tilde\mu$.
\end{proposition}
\begin{proposition}\label{originally_cut_mu_lemma}
There exist $c,c'>0$ and $\e_0>0$ such that for all $\e<\e_0$ and all $\tilde\mu\in\mathcal{\tilde C}_{\e}$, 
\begin{equation*}
\left|g_{\e}(\tilde\mu)\det(I+\e^{1/2}\tilde\mu J_{\e^{1/2}\tilde\mu})_{L^2(\Gamma_{\eta})}\right| \leq c'e^{-c|\tilde\mu|}.
\end{equation*} 
\end{proposition}
This exponential decay bound on the integrand shows that that, by choosing a suitably large (fixed) compact region around zero along the contour $\mathcal{\tilde C}_{\e}$, it is possible to make the $\tilde\mu$ integral outside of this region arbitrarily small, uniformly in $\e\in (0,\e_0)$. This means that we may assume henceforth  that $\tilde\mu$ lies in a compact subset of $\mathcal{\tilde C}$.

Now that we are on a fixed compact set of $\tilde\mu$, the first part of  Lemma \ref{mu_inequalities_lemma} and Proposition \ref{uniform_limit_det_J_to_Kcsc_proposition} combine to show that the integrand converges uniformly to
\begin{equation*}
 \frac{e^{-\tilde\mu/2}}{\tilde\mu} \det(I-K^{\csc}_{a'})_{L^2(\tilde\Gamma_{\eta})}
\end{equation*}
and hence the integral converges to the integral with this integrand.

To finish the proof of the limit in Theorem \ref{epsilon_to_zero_theorem}, it is necessary that for any $\delta$ we can find a suitably small $\e_0$ such that the difference between the two sides of the limit differ by less than $\delta$ for all $\e<\e_0$. Technically we are in the position of a $\delta/3$ argument. One portion of $\delta/3$ goes to the cost of cutting off the $\tilde\mu$ contour outside of some compact set. Another $\delta/3$ goes to the uniform convergence of the integrand. The final portion goes to repairing the $\tilde\mu$ contour. As $\delta$ gets smaller, the cut for the $\tilde\mu$ contour must occur further out. Therefore the limiting integral will be over the limit of the $\tilde\mu$ contours, which we called $\mathcal{\tilde C}$. The final $\delta/3$ is spent on the following Proposition, whose proof is given in Section \ref{proofs_sec}.

\begin{proposition}\label{reinclude_mu_lemma}
There exists $c,c'>0$ such that for all $\tilde\mu\in \mathcal{\tilde C}$ with $|\tilde\mu|\geq 1$,
\begin{equation*}
\left|\frac{e^{-\tilde\mu/2}}{\tilde\mu} \det(I-K^{\csc}_{a})_{L^2(\tilde\Gamma_{\eta})}\right| \leq |c'e^{-c\tilde\mu}|.
\end{equation*}
\end{proposition}
Recall that the kernel $K^{\csc}_{a}$ is a function of $\tilde\mu$.  
The argument used to prove this proposition immediately shows that $K_a^{\csc}$ is a trace class operator on $L^2(\tilde\Gamma_{\eta})$.

It is an immediate corollary of this exponential tail bound that for sufficiently large compact sets of $\tilde\mu$,  the cost to include the rest of the $\tilde\mu$ contour is less than $\delta/3$. This, along with the change of variables in $\tilde\mu$ described at the end of Section \ref{formal_calc_subsec} finishes the proof of Theorem \ref{epsilon_to_zero_theorem}.

\subsection{Proof of Proposition \ref{uniform_limit_det_J_to_Kcsc_proposition}}\label{J_to_K_sec}

In this section we provide all of the steps necessary to prove Proposition \ref{uniform_limit_det_J_to_Kcsc_proposition}. To ease understanding of the argument we relegate more technical points to lemmas whose proof we delay to Section \ref{JK_proofs_sec}.

During the proof of this proposition, it is important to  keep in mind that we are assuming that  $\tilde\mu$ lies in a fixed compact subset of $\mathcal{\tilde C}$. Recall that $\tilde\mu = \e^{-1/2}\mu$. We proceed via the following strategy  to find the limit of the Fredholm determinant as $\e$ goes to zero. The first step is to deform the contours $\Gamma_{\eta}$ and $\Gamma_{\zeta}$ to suitable curves along which there exists a small region outside of which the kernel of our operator is exponentially small. This justifies cutting the contours off outside of this small region. We may then rescale everything so this small region becomes order one in size. Then  we show uniform convergence of the kernel to the limiting kernel on the compact subset. Finally we need to show that we can complete the finite contour on which this limiting object is defined to an infinite contour without significantly changing the value of the determinant. 

Recall now that $\Gamma_{\zeta}$ is defined to be a circle centered at zero of radius $1+\e^{1/2}/2$ and $\Gamma_{\eta}$ is a circle centered at zero of radius $1-\e^{1/2}/2$ and that
\begin{equation*}
 \xi = -1 - 2\e^{1/2}\frac {X}{T}.
\end{equation*}
The function $f(\mu,\zeta/\eta')$ which shows up in the definition of the kernel for $J$ has poles as every point $\zeta/\eta'=z=\tau^k$ for $k\in \Z$. 

As long as we simultaneously deform the $\Gamma_{\zeta}$ contour as we deform $\Gamma_{\eta}$ so as to keep $\zeta/\eta'$ away from these poles, we may use Proposition \ref{TWprop1} (Proposition 1 of \cite{TW3}), to justify the fact that the determinant does not change under this deformation. In this way we may deform our contours to the following modified contours $\Gamma_{\eta,l},\Gamma_{\zeta,l}$:

\begin{figure}
\begin{center}
\includegraphics[scale=.6]{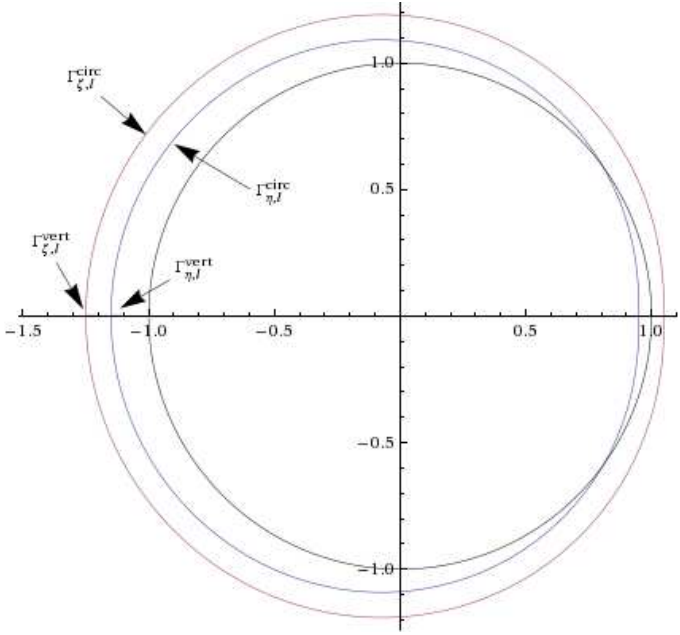}
\caption{$\Gamma_{\zeta,l}$ (the outer most curve) is composed of a small verticle section near $\xi$ labeled $\Gamma_{\zeta,l}^{vert}$ and a large almost circular (small modification due to the function $\kappa(\theta)$) section labeled $\Gamma_{\zeta,l}^{circ}$. Likewise $\Gamma_{\eta,l}$ is the middle curve, and the inner curve is the unit circle. These curves depend on $\e$ in such a way that $|\zeta/\eta|$ is bounded between $1$ and $\tau^{-1}\approx 1+2\e^{1/2}$.}\label{kappa_contour}
\end{center}
\end{figure}

\begin{definition}
Let $\Gamma_{\eta,l}$ and $\Gamma_{\zeta,l}$ be two families (indexed by $l>0$) of simple closed contours in $\C$ defined as follows. Let 
\begin{equation}\label{kappa_eqn}
\kappa(\theta) = \frac{2X}{T} \tan^2\left(\frac{\theta}{2}\right)\log\left(\frac{2}{1-\cos\theta}\right).
\end{equation}
Both $\Gamma_{\eta,l}$ and $\Gamma_{\zeta,l}$ will be symmetric across the real axis, so we need only define them on the top half. $\Gamma_{\eta,l}$ begins at $\xi+\e^{1/2}/2$ and moves along a straight vertical line for a distance $l\e^{1/2}$ and then joins the curve 
\begin{equation}\label{kappa_param_eqn}
 \left[1+\e^{1/2}(\kappa(\theta)+\alpha)\right]e^{i\theta}
\end{equation}
parametrized by $\theta$ from $\pi-l\e^{1/2} + O(\e)$ to $0$, and where $\alpha = -1/2 + O(\e^{1/2})$ (see Figure \ref{kappa_contour} for an illustration of these contours). The small errors are necessary to make sure that the curves join up at the end of the vertical section of the curve. We extend this to a closed contour by reflection through the real axis and orient it clockwise. We denote the first, vertical part, of the contour by $\Gamma_{\eta,l}^{vert}$ and the second, roughly circular part by $\Gamma_{\eta,l}^{circ}$. This means that $\Gamma_{\eta,l}=\Gamma_{\eta,l}^{vert}\cup \Gamma_{\eta,l}^{circ}$,
and along this contour we can think of parametrizing $\eta$ by $\theta\in [0,\pi]$.

We define $\Gamma_{\zeta,l}$ similarly, except that it starts out at $\xi-\e^{1/2}/2$ and joins the curve given by equation (\ref{kappa_param_eqn})
where the value of $\theta$ ranges from $\theta=\pi-l\e^{1/2} + O(\e)$ to $\theta=0$ and where $\alpha = 1/2 + O(\e^{1/2})$. We similarly denote this contour by the union of  $\Gamma_{\zeta,l}^{vert}$ and $\Gamma_{\zeta,l}^{circ}$.
\end{definition}

By virtue of these definitions, it is clear that $\e^{-1/2}|\zeta/\eta'-\tau^k|$ stays bounded away from zero for all $k$, and that $|\zeta/\eta'|$ is bounded in an closed set contained in $(1,\tau^{-1})$ for all $\zeta\in \Gamma_{\zeta,l}$ and $\eta\in \Gamma_{\eta,l}$. Therefore, for any $l>0$ we may, by deforming both the $\eta$ and $\zeta$ contours simultaneously, assume that our operator acts on $L^2(\Gamma_{\eta,l})$ and that its kernel is defined via an integral along $\Gamma_{\zeta,l}$. It is critical that we now show that, due to our choice of contours, we are able to forget about everything except for the vertical part of the contours. To formulate this we have the following:

\begin{definition}
Let $\chi_l^{vert}$ and $\chi_l^{circ}$ be projection operators acting on $L^2(\Gamma_{\eta,l})$ which project onto $L^2(\Gamma_{\eta,l}^{vert})$ and $L^2(\Gamma_{\eta,l}^{circ})$ respectively. Also define two operators $J_l^{vert}$ and $J_l^{circ}$ which act on $L^2(\Gamma_{\eta,l})$ and have kernels identical to $J$ (see equation (\ref{J_eqn_def})) except the $\zeta$ integral is over $\Gamma_{\zeta,l}^{vert}$ and $\Gamma_{\zeta,l}^{circ}$ respectively.
Thus we have a family (indexed by $l>0$) of decompositions of our operator $J$ as follows:
\begin{equation*}
J = J_l^{vert}\chi_{l}^{vert} +J_l^{vert}\chi_{l}^{circ}+J_l^{circ}\chi_{l}^{vert}+J_l^{circ}\chi_{l}^{circ}.
\end{equation*}
\end{definition}

We now show that it suffices to just consider the first part of this decomposition ($J_l^{vert}\chi_{l}^{vert}$) for sufficiently large $l$.

\begin{proposition}\label{det_1_1_prop}  Assume that $\tilde\mu$ is restricted to a bounded subset of the contour $\mathcal{\tilde C}$.
For all $\delta>0$ there exist  $\e_0>0$ and $l_0>0$ such that for all $\e<\e_0$ and all $l>l_0$,
\begin{equation*}
|\det(I+\mu J)_{L^2(\Gamma_{\eta,l})} - \det(I+J_{l}^{vert})_{L^2(\Gamma_{\eta,l}^{vert})}|<\delta.
\end{equation*}
\end{proposition}

\begin{proof}
As was explained in the introduction, if we let
\begin{equation}\label{n_0_eqn}
n_0=\lfloor \log(\e^{-1/2}) /\log(\tau)\rfloor
\end{equation}
then it follows from the invariance of the doubly infinite sum for $f(\mu,z)$ that
\begin{equation*}
 \mu f(\mu,z) = z^{n_0} (\tilde\mu f(\tilde\mu,z) + O(\e^{1/2})).
\end{equation*}
Note that the $O(\e^{1/2})$  does not play a significant role in what follows so we drop it.

Using the above argument and the following two lemmas (which are proved in Section \ref{JK_proofs_sec}) we will be able to complete the proof of Proposition \ref{det_1_1_prop}.

\begin{lemma}\label{kill_gamma_2_lemma}
For all $c>0$ there exist $l_0>0$ and $\e_0>0$ such that for all $l>l_0$, $\e<\e_0$ and $\eta\in \Gamma_{\eta,l}^{circ}$,
\begin{equation*}
 \re(\Psi(\eta)+n_0\log(\eta))\geq c|\xi-\eta|\e^{-1/2},
\end{equation*}
where $n_0$ is defined in  (\ref{n_0_eqn}). Likewise, for all $\e<\e_0$ and $\zeta\in \Gamma_{\zeta,l}^{circ}$,
\begin{equation*}
 \re(\Psi(\zeta)+n_0\log(\zeta))\leq -c|\xi-\zeta|\e^{-1/2}.
\end{equation*}
\end{lemma}

\begin{lemma}\label{mu_f_polynomial_bound_lemma}
For all $l>0$ there exist $\e_0>0$ and $c>0$ such that for all $\e<\e_0$, $\eta'\in \Gamma_{\eta,l}$ and $\zeta\in \Gamma_{\zeta,l}$,
\begin{equation*}
 |\tilde\mu f(\tilde\mu,\zeta/\eta')|\leq \frac{c}{|\zeta-\eta'|}.
\end{equation*}
\end{lemma}

It now follows that for any $\delta>0$, we can find $l_0$ large enough that $||J_l^{vert}\chi_{l}^{circ}||_1$, $||J_l^{circ}\chi_{l}^{vert}||_1$ and $||J_l^{circ}\chi_{l}^{circ}||_1$ are all bounded by $\delta/3$. This is because we may factor these various operators into a product of Hilbert-Schmidt operators and then use the exponential decay of Lemma \ref{kill_gamma_2_lemma} along with the polynomial control of Lemma \ref{mu_f_polynomial_bound_lemma} and the remaining term $1/(\zeta-\eta)$ to prove that each of the Hilbert-Schmidt norms goes to zero (for a similar argument, see the bottom of page 27 of \cite{TW3}).
This completes the proof of Proposition \ref{det_1_1_prop}.
\end{proof}

We  now return to the  proof of  Proposition \ref{uniform_limit_det_J_to_Kcsc_proposition}. We have successfully restricted ourselves to considering $J_{l}^{vert}$ acting on $L^2(\Gamma_{\eta,l}^{vert})$. Having focused on the region of asymptotically non-trivial behavior, we can now rescale and show that the kernel converges to its limit, uniformly on the compact contour.

\begin{definition}\label{change_of_var_tilde_definitions}
Recall $c_3=2^{-4/3}$ and let
\begin{equation*}
\eta = \xi + c_3^{-1}\e^{1/2}\tilde\eta, \qquad \eta' = \xi + c_3^{-1}\e^{1/2}\tilde\eta', \qquad \zeta = \xi + c_3^{-1}\e^{1/2}\tilde\zeta.
\end{equation*}
Under these change of variables the contours $\Gamma_{\eta,l}^{vert}$ and  $\Gamma_{\zeta,l}^{vert}$ become
\begin{eqnarray*}
 \tilde\Gamma_{\eta,l} = \{c_3/2+ir:r\in (-c_3l,c_3l)\},\\
 \tilde\Gamma_{\zeta,l} = \{-c_3/2+ir:r\in (-c_3l,c_3l)\}.
\end{eqnarray*}
As  $l$ increases to infinity, these contours approach their infinite versions,
\begin{eqnarray*}
 \tilde\Gamma_{\eta} = \{c_3/2+ir:r\in (-\infty,\infty)\},\\
 \tilde\Gamma_{\zeta} = \{-c_3/2+ir:r\in (-\infty,\infty)\}.
\end{eqnarray*}
With respect to the change of variables define an operator $\tilde J$ acting on $L^2(\tilde\Gamma_{\eta})$ via the kernel:
\begin{equation*}
\mu \tilde J_l (\tilde\eta,\tilde\eta') = c_{3}^{-1}\e^{1/2} \int_{\tilde\Gamma_{\zeta,l}} e^{\Psi(\xi+c_3^{-1} \e^{1/2}\tilde\zeta)-\Psi(\xi+c_3^{-1} \e^{1/2}\tilde\eta')} \frac{\mu f(\mu,\frac{\xi+c_3^{-1}\e^{1/2}\tilde\zeta}{\xi+c_3^{-1}\e^{1/2}\tilde\eta'})}{(\xi+c_3^{-1}\e^{1/2}\tilde\eta')(\tilde\zeta-\tilde\eta)}d\tilde\zeta.
\end{equation*}
Lastly, define the operator $\tilde\chi_l$ which projects $L^2(\tilde\Gamma_{\eta})$ onto $L^2(\tilde\Gamma_{\eta,l})$.
\end{definition}

It is clear that applying the change of variables, the Fredholm determinant  $\det(I+J_{l}^{vert})_{L^2(\Gamma_{\eta,l}^{vert})}$ becomes
$\det(I+\tilde\chi_l \mu\tilde J_l \tilde\chi_l)_{L^2(\tilde\Gamma_{\eta,l})}$.

We now state a proposition which gives, with respect to these fixed contours $\tilde\Gamma_{\eta,l}$ and $\tilde\Gamma_{\zeta,l}$, the limit of the determinant in terms of the uniform limit of the kernel. Since all contours in question are finite, uniform convergence of the kernel suffices to show trace class convergence of the operators and hence convergence of the determinant.

Recall the definition of the operator $K_a^{\csc}$ given in Definition \ref{thm_definitions}. For the purposes of this proposition, modify the kernel so that the integration in $\zeta$ occurs now only over $\tilde\Gamma_{\zeta,l}$ and not all of $\tilde\Gamma_{\zeta}$. Call this modified operator $K_{a',l}^{\csc}$.

\begin{proposition}\label{converges_to_kcsc_proposition}
For all $\delta>0$ there exist $\e_0>0$ and $l_0>0$ such that for all $\e<\e_0$, $l>l_0$, and  $\tilde\mu$ in our fixed compact subset of $\mathcal{\tilde C}$,
\begin{equation*}
 \left|\det(I+\tilde\chi_l \mu\tilde J_l \tilde\chi_l)_{L^2(\tilde\Gamma_{\eta,l})} - \det(I-\tilde\chi_l K_{a',l}^{\csc}\tilde\chi_l)_{L^2(\tilde\Gamma_{\eta,l})}\right|< \delta,
\end{equation*}
where $a'=a+\log2$.
\end{proposition}
\begin{proof}
The proof of this proposition relies on showing the uniform convergence of the kernel of $\mu\tilde J$ to the kernel of $K_{a',l}^{\csc}$, which suffices because of the compact contour. Furthermore, since the $\zeta$ integration is itself over a compact set, it  suffices to show uniform convergence of this integrand. The two lemmas stated below will imply such uniform convergence and hence complete this proof.

First, however, recall that $\mu f(\mu,z) = z^{n_0} (\tilde\mu f(\tilde\mu,z) + O(\e^{1/2}))$ where $n_0$ is defined in equation (\ref{n_0_eqn}). We are interested in having $z=\zeta/\eta'$, which, under the change of variables can be written as
\begin{equation*}
 z=1-\e^{1/2}\tilde z +O(\e), \qquad \tilde z = c_3^{-1}(\tilde\zeta-\tilde\eta')=2^{4/3}(\tilde\zeta-\tilde\eta').
\end{equation*}
Therefore, since $n_0= -\frac{1}{2}\log(\e^{-1/2})\e^{-1/2}+ O(1)$ it follows that
\begin{equation*}
 z^{n_0} = \exp\{-2^{1/3}(\tilde\zeta-\tilde\eta')\log(\e^{-1/2})\}(1+o(1)).
\end{equation*}
This expansion still contains an $\e$ and hence the argument blows up as $\e$ goes to zero. However, this exactly counteracts the $\log(\e^{-1/2})$ term in the definition of $m$ which goes into the argument of the exponential of the integrand. We make use of this cancellation in the proof of this first lemma and hence include the $n_0\log(\zeta/\eta')$ term into the exponential argument.

The following two lemmas are proved in Section \ref{JK_proofs_sec}.
\begin{lemma}\label{compact_eta_zeta_taylor_lemma}
For all $l>0$ and all $\delta>0$ there exists $\e_0>0$ such that for all $\tilde\eta'\in \tilde\Gamma_{\eta,l}$ and $\tilde\zeta\in \tilde\Gamma_{\zeta,l}$ we have for $0<\e\le\e_0$,
\begin{equation*}
\left|\left(\Psi(\tilde\zeta)-\Psi(\tilde\eta') + n_0\log(\zeta/\eta')\right) - \left(-\frac{T}{3}(\tilde\zeta^3-\tilde\eta'^3) + 2^{1/3}a'(\tilde\zeta-\tilde\eta)\right)\right|<\delta,
\end{equation*}
where $a=a'+\log2$. Similarly we have
\begin{equation*}
\left|e^{\Psi(\tilde\zeta)-\Psi(\tilde\eta') + n_0\log(\zeta/\eta')} - e^{-\frac{T}{3}(\tilde\zeta^3-\tilde\eta'^3) + 2^{1/3}a'(\tilde\zeta-\tilde\eta')}\right|<\delta.
\end{equation*}
\end{lemma}

\begin{lemma}\label{muf_compact_sets_csc_limit_lemma}
For all $l>0$ and all $\delta>0$ there exists $\e_0>0$ such that for all $\tilde\eta'\in \tilde\Gamma_{\eta,l}$ and $\tilde\zeta\in \tilde\Gamma_{\zeta,l}$ we have for $0<\e\le\e_0$,
\begin{equation*}
 \left|\e^{1/2}\tilde\mu f\left(\tilde\mu, \frac{\xi+c_3^{-1}\e^{1/2}\tilde\zeta}{\xi+c_3^{-1}\e^{1/2}\tilde\eta'}\right) - \int_{-\infty}^{\infty} \frac{\tilde\mu e^{-2^{1/3}t(\tilde\zeta-\tilde\eta')}}{e^{t}-\tilde\mu}dt\right|<\delta.
\end{equation*}
\end{lemma}
As explained in Definition \ref{thm_definitions}, the final integral converges since our choices of $\tilde\zeta$ and $\tilde\eta'$ ensure that $\re(-2^{1/3}(\tilde\zeta-\tilde\eta'))=1/2$. 
Note that the above integral also has a representation (\ref{cscid}) in terms of the $\csc$ function. This gives the analytic extension of the integral to all $\tilde z\notin 2\Z$ where $\tilde{z}=2^{4/3}(\tilde\zeta-\tilde\eta')$.

Finally, the sign change in front of the kernel of the Fredholm determinant comes from the $1/\eta'$ term which, under the change of variables converges uniformly to $-1$.
\end{proof}

Having successfully taken the $\e$ to zero limit, all that now remains is to paste the rest of the contours $\tilde\Gamma_{\eta}$ and $\tilde\Gamma_{\zeta}$ to their abbreviated versions $\tilde\Gamma_{\eta,l}$ and $\tilde\Gamma_{\zeta,l}$. To justify this we must show that the inclusion of the rest of these contours does not significantly affect the Fredholm determinant.
Just as in the proof of Proposition \ref{det_1_1_prop} we have three operators which we must re-include at provably small cost. Each of these operators, however, can be factored into the product of Hilbert Schmidt operators and then an analysis similar to that done following  Lemma \ref{mu_f_polynomial_bound_lemma}  (see in particular page 27-28 of \cite{TW3}) shows that because $\re(\tilde\zeta^3)$ grows like $|\tilde\zeta|^2$ along $\tilde\Gamma_{\zeta}$ (and likewise but opposite for $\eta'$) there is sufficiently strong exponential decay to show that the trace norms of these three additional kernels can be made arbitrarily small by taking $l$ large enough.

This last estimate completes the proof of Proposition \ref{uniform_limit_det_J_to_Kcsc_proposition}.

\subsection{Technical lemmas, propositions and proofs}\label{props_and_lemmas_sec}

\subsubsection{Properties of Fredholm determinants}\label{pre_lem_ineq_sec}

Before beginning the proofs of the propositions and lemmas, we give the definitions and some important propeties for Fredholm determinants, trace class operators and Hilbert-Schmidt operators. For a more complete treatment of this theory see, for example, \cite{BS:book}. 

Consider a (separable) Hilbert space $\Hi$ with bounded linear operators $\mathcal{L}(\Hi)$. If $A\in \mathcal{L}(\Hi)$, let $|A|=\sqrt{A^*A}$ be the unique positive square-root. We say that $A\in\mathcal{B}_1(\Hi)$, the trace class operators, if the trace norm $||A||_1<\infty$. Recall that this norm is defined relative to an orthonormal basis of $\Hi$ as $||A||_1:= \sum_{n=1}^{\infty} (e_n,|A|e_n)$.
This norm is well defined as it does not depend on the choice of orthonormal basis $\{e_n\}_{n\geq 1}$. For $A\in\mathcal{B}_1(\Hi)$, one can then define the trace $\tr A :=\sum_{n=1}^{\infty} (e_n,A e_n)$.  We say that $A\in \mathcal{B}_{2}(\Hi)$, the Hilbert-Schmidt operators,  if the Hilbert-Schmidt norm $||A||_2 := \sqrt{\tr(|A|^2)}<\infty$. 

\begin{lemma}[Pg. 40 of \cite{BOO}, from Theorem 2.20 from \cite{BS:book}]\label{trace_convergence_lemma}
\mbox{}
The following conditions are equivalent:
\begin{enumerate}
 \item $||K_n-K||_1\to 0$;
 \item $\tr K_n\to \tr K$ and $K_n\to K$ in the weak operator topology.
\end{enumerate}
\end{lemma}

For $A\in~\mathcal{B}_1(\Hi)$ we can also define a Fredholm determinant $\det(I+~A)_{\Hi}$.  Consider $u_i\in \Hi$ and define the tensor product $u_1\otimes \cdots \otimes u_n$ by its action on $v_1,\ldots, v_n \in\Hi$ as
\begin{equation*}
u_1\otimes \cdots \otimes u_n (v_1,\ldots, v_n) = \prod_{i=1}^{n} (u_i,v_i).
\end{equation*}
Then $\bigotimes_{i=1}^{n}\Hi$ is the span of all such tensor products.
There is a vector subspace of this space which is known as the alternating product:
\begin{equation*}
\bigwedge^n(\Hi) = \{h\in\bigotimes_{i=1}^{n} \Hi : \forall \sigma\in S_n, \sigma h =-h\},
\end{equation*}
where $\sigma u_1\otimes \cdots \otimes u_n = u_{\sigma(1)}\otimes \cdots \otimes u_{\sigma(n)}$.
If $e_1,\ldots,e_n$  is a basis for $\Hi$ then $e_{i_1}\wedge \cdots \wedge e_{i_k}$ for $1\leq i_1<\ldots<i_k\leq n$ form a basis of $\bigwedge^n(\Hi)$.
Given an operator $A\in \mathcal{L}(\Hi)$, define
\begin{equation*}
\Gamma^n(A)(u_1\otimes \cdots \otimes u_n) := Au_1\otimes \cdots \otimes Au_n.
\end{equation*}
Note that any element in $\bigwedge^n(\Hi)$ can be written as an antisymmetrization of tensor products. Then it follows that $\Gamma^n(A)$ restricts to an operator from $\bigwedge^n(\Hi)$ into $\bigwedge^n(\Hi)$. If $A\in~\mathcal{B}_1(\Hi)$, then $\tr \Gamma^{(n)}(A)\leq ||A||_1^n/n!$, and 
we can define 
\begin{equation*}
\det(I+~A)= 1 + \sum_{k=1}^{\infty} \tr(\Gamma^{(k)}(A)).
\end{equation*}
As one expects, $\det(I+~A)=\prod_j (1+\lambda_j)$ where $\lambda_j$ are the eigenvalues of $A$ counted with algebraic multiplicity (Thm XIII.106, \cite{RS:book}).

\begin{lemma}[Ch. 3 \cite{BS:book}]\label{fredholm_continuity_lemma}
$A\mapsto \det(I+A)$ is a continuous function on $\mathcal{B}_1(\Hi)$. Explicitly,
\begin{equation*}
 |\det(I+A)-\det(I+B)|\leq ||A-B||_{1}\exp(||A||_1+||B||_1+1).
\end{equation*}
If $A\in \mathcal{B}_1(\Hi)$ and $A=BC$ with $B,C\in \mathcal{B}_2(\Hi)$ then
\begin{equation*}
 ||A||_1\leq ||B||_2||C||_2.
\end{equation*}
For $A\in \mathcal{B}_1(\Hi)$,
\begin{equation*}
 |\det(I+A)|\leq e^{||A||_1}.
\end{equation*}
If $A\in \mathcal{B}_2(\Hi)$ with kernel $A(x,y)$ then 
\begin{equation*}
||A||_2 = \left(\int |A(x,y)|^2 dx dy\right)^{1/2}.
\end{equation*}
\end{lemma}

\begin{lemma}\label{projection_pre_lemma}
 If $K$ is an operator acting on a contour $\Sigma$ and $\chi$ is a projection operator unto a subinterval of $\Sigma$ then
\begin{equation*}
 \det(I+K\chi)_{L^2(\Sigma,\mu)}=\det(I+\chi K\chi)_{L^2(\Sigma,\mu)}.
\end{equation*}
\end{lemma}

In performing steepest descent analysis on Fredholm determinants, the following proposition allows one to deform contours to descent curves.

\begin{lemma}[Proposition 1 of \cite{TW3}]\label{TWprop1}
Suppose $s\to \Gamma_s$ is a deformation of closed curves and a kernel $L(\eta,\eta')$ is analytic in a neighborhood of $\Gamma_s\times \Gamma_s\subset \C^2$ for each $s$. Then the Fredholm determinant of $L$ acting on $\Gamma_s$ is independent of $s$.
\end{lemma}

The following lemma, provided to us by Percy Deift, with proof provided in Appendix \ref{PD_appendix}, allows us to use Cauchy's theorem when manipulating integrals which involve Fredholm determinants in the integrand.

\begin{lemma}\label{Analytic_fredholm_det_lemma}
Suppose $A(z)$ is an analytic map from a region $D\in \C$ into the trace-class operators on a (separable) Hilbert space $\mathcal{\Hi}$. Then $z\mapsto \det(I+A(z))$ is analytic on $D$.
\end{lemma}

\subsubsection{Proofs from Section \ref{proof_of_WASEP_thm_sec}}\label{proofs_sec}
We now turn to the proofs of the previously stated lemmas and propsitions.
\begin{proof}[Proof of Lemma \ref{deform_mu_to_C}]
The lemma follows from Cauchy's theorem once we  show that for fixed $\epsilon$, the integrand $\mu^{-1} \prod_{k=0}^{\infty} (1-\mu\tau^k)\det(I+\mu J_{\mu})$ is analytic in $\mu$ between $S_\e$ and $\mathcal{C}_\e$ (note that we now include a subscript $\mu$ on $J$ to emphasize the dependence of the kernel on $\mu$). It is clear that the infinite product and the $\mu^{-1}$ are analytic in this region. In order to show that $\det(I+\mu J_{\mu})$ is analytic in the desired region we
may appeal to Lemma \ref{Analytic_fredholm_det_lemma}. Therefore it suffices to show that the map $J(\mu)$ defined by $\mu\mapsto J_{\mu}$
is an analytic map from this region of $\mu$ between $S_\e$ and $\mathcal{C}_\e$ into the trace class operators (this suffices since the multiplication by $\mu$ is clearly analytic).  
The rest of this proof is devoted to the proof of this fact.

In order to prove this, we need to show that $J_{\mu}^h=\frac{J_{\mu+h}-J_{\mu}}{h}$ converges to some trace class operator  as $h\in \C$ goes to zero. By the criteria of Lemma \ref{trace_convergence_lemma} it suffices to prove that the kernel associated to $J_{\mu}^h$ converges uniformly in $\eta,\eta'\in\Gamma_{\eta}$ to the kernel of $J'_{\mu}$.  This will prove both the convergence of traces as well as the weak convergence of operators necessary to prove trace norm convergence and complete this proof. The operator $J'_{\mu}$ acts on $\Gamma_{\eta}$, the circle centered at zero and of radius $1-\tfrac{1}{2}\e^{1/2}$, as
\begin{equation*}
 J'_{\mu}(\eta,\eta')=\int_{\Gamma_{\zeta}} \exp\{\Psi(\zeta)-\Psi(\eta')\}\frac{f'(\mu,\zeta/\eta')}{\eta'(\zeta-\eta)}d\zeta
\end{equation*}
where
\begin{equation*}
 f'(\mu,z)=\sum_{k=-\infty}^{\infty} \frac{\tau^{2k}}{(1-\tau^k\mu)^2}z^k.
\end{equation*}
Our desired convergence will follow if we can show that 
\begin{equation*}
 \left|h^{-1}\left(f(\mu+h,\zeta/\eta')-f(\mu,\zeta/\eta')\right) - f'(\mu,\zeta/\eta')\right|
\end{equation*}
tends to zero uniformly in $\zeta\in \Gamma_{\zeta}$ and $\eta'\in \Gamma_{\eta}$ as $|h|$ tends to zero. 
Expanding this out and taking the absolute value inside of the infinite sum we have
\begin{equation}\label{sum_eqn}
 \sum_{k=-\infty}^{\infty} \left| h^{-1} \left(\frac{\tau^k}{1-\tau^k(\mu+h)}-\frac{\tau^k}{1-\tau^k(\mu)}\right) - \frac{\tau^{2k}}{(1-\tau^k(\mu))^2}\right| z^k
\end{equation}
where $z=|\zeta/\eta'|\in (1,\tau^{-1})$.
For $\e$ and $\mu$ fixed there is a $k^*$ such that for $k\ge k^*$,
\begin{equation*}
\left|\frac{\tau^k h}{1-\tau^k \mu}\right|<1. 
\end{equation*}
Furthermore, by choosing $|h|$ small enough we can make $k^*$ negative. As a result we also have that for small enough $|h|$, for all $k<k^*$,
\begin{equation*}
\left|\frac{h}{\tau^{-1}-\mu}\right|<1. 
\end{equation*}
Splitting our sum into  $k<k^*$ and $k\ge k^*$, and using the fact that $1/(1-w)=1+w+O(w^2)$ for $|w|<1$ we can Taylor expand as follows: For $k\geq k^*$
\begin{equation*}
\frac{\tau^k}{1-\tau^k(\mu+h)} = \frac{\tau^k}{1-\tau^k\mu}\frac{1}{1-\frac{\tau^k h}{1-\tau^k \mu}} = \frac{\tau^k \left(1+\frac{\tau^k h}{1-\tau^k \mu} + \left(\frac{\tau^k}{1-\tau^k\mu}\right)^2O(h^2)\right)}{1-\tau^k \mu}. 
\end{equation*}
Similarly, expanding the second term inside the absolute value in equation (\ref{sum_eqn}) and canceling with the third term we are left with
\begin{equation*}
 \sum_{k=k^*}^{\infty} \frac{\tau^{3k}}{(1-\tau^k \mu)^3} O(h) z^k.
\end{equation*}
The sum converges since $\tau^3 z <1$ and thus behaves like $O(h)$ as desired.
Likewise for $k<k^*$, by multiplying the numerator and denominator by $\tau^{-k}$, the same type of expansion works and we find that the error is given by the same summand as above but over $k$ from $-\infty$ to $k^*-1$. Again, however,
the sum converges since the numerator and denominator cancel each other for $k$ large negative, and $z^k$ is a convergent series for $k$ going to negative infinity. Thus this error series also behaves like $O(h)$ as desired. This shows the needed uniform convergence and completes the proof.
\end{proof}

\begin{proof}[Proof of Lemma \ref{mu_inequalities_lemma}]
We prove this with the scaling parameter $r=1$ as the general case follows in a similar way.
Consider
\begin{equation*}
\log(g_{\e}(\tilde\mu))=\sum_{k=0}^{\infty} \log(1-\e^{1/2}\tilde\mu \tau_\e^k).
\end{equation*}
 We have $\sum_{k=0}^{\infty} \epsilon^{1/2} \tau^k = \frac12(1 + \e^{1/2}c_\e )$ where $c_\e=O(1)$.  So for $\tilde\mu\in R_1$ we have
\begin{eqnarray*}
\nonumber|\log(g_{\e}(\tilde\mu))+\frac{\tilde\mu}2(1+ \e^{1/2}c_\e )| &=& \left|\sum_{k=0}^{\infty}\log(1-\epsilon^{1/2}\tilde\mu \tau^k)+ \epsilon^{1/2}\tilde\mu \tau^k\right|\\
 &\leq & \sum_{k=0}^{\infty} |\log(1-\epsilon^{1/2}\tilde\mu \tau^k)+ \epsilon^{1/2}\tilde\mu \tau^k|\\
\nonumber &\leq & \sum_{k=0}^{\infty} |\epsilon^{1/2}\tilde\mu \tau^k|^2 = \frac{\e|\tilde\mu|^2}{1-\tau^2}=
\frac{\e^{1/2} |\tilde\mu|^2}{4-4\e^{1/2}}\\ &\leq& c\e^{1/2}|\tilde\mu|^2
 \nonumber\leq  c' \e^{1/2}.
\end{eqnarray*}
 The second inequality uses the fact that for $|z|\leq 1/2$,
$|\log(1-z)+z|\leq |z|^2$. Since $\tilde\mu\in R_1$ it follows that $|z|=\e^{1/2}|\tilde\mu|$ is  bounded by $1/2$ for small enough $\e$. The constants here are finite and do not depend on any of the parameters.  This proves equation (\ref{g_e_ineq1}) and shows that the convergence is uniform in $\tilde\mu$ on $R_1$.

We now turn to the second inequality, equation (\ref{g_e_ineq2}). Consider the region, \begin{equation*}
 D=\{z:\arg(z)\in [-{\scriptstyle\frac{\pi}{10}},{\scriptstyle\frac{\pi}{10}}]\}\cap \{z:\im(z)\in (-{\scriptstyle\frac{1}{10}},{\scriptstyle\frac{1}{10}})\}\cap \{z:\re(z)\leq 1\}.
 \end{equation*}
For all $z\in D$,
\begin{equation}\label{ineq2}
 \re(\log(1-z))\leq \re(-z-z^2/2).
\end{equation}
For $\tilde\mu\in R_2$, it is clear that $\e^{1/2}\tilde\mu\in D$. Therefore, using (\ref{ineq2}),
\begin{eqnarray*}
 \nonumber\re(\log(g_{\e}(\tilde\mu))) &=& \sum_{k=0}^{\infty}\re[\log(1-\epsilon^{1/2}\tilde\mu \tau^k)]\\
&\leq & \sum_{k=0}^{\infty} \left(-\re[\e^{1/2}\tilde\mu \tau^k]-\re[(\e^{1/2}\tilde\mu \tau^k)^2/2]\right)\\
\nonumber&\leq & -\re(\tilde\mu/2)-\frac{1}{8}\e^{1/2}\re(\tilde\mu^2).
\end{eqnarray*}
This proves equation (\ref{g_e_ineq2}). Note that from the definition of $R_2$ we can calculate the argument of $\tilde\mu$ and we see that $|\arg \tilde\mu|\leq \arctan(2\tan(\tfrac{\pi}{10}))<\tfrac{\pi}{4}$ and $|\tilde\mu|\geq r\geq 1$. Therefore $\re(\tilde\mu^2)$ is positive and bounded away from zero for all $\tilde\mu\in R_2$.
\end{proof}

\begin{proof}[Proof of Proposition \ref{originally_cut_mu_lemma}]
This proof proceeds in a similar manner to the proof of Proposition \ref{reinclude_mu_lemma}, however, since in this case we have to deal with $\e$ going to zero and changing contours, it is, by necessity, a little more complicated. For this reason we encourage readers to first study the simpler proof of Proposition \ref{reinclude_mu_lemma}. 

In that proof we factor our operator into two pieces. Then, using the decay of the exponential term, and the control over the size of the $\csc$ term, we are able to show that the Hilbert-Schmidt norm of the first factor is finite and that for the second factor it is bounded by $|\tilde\mu|^{\alpha}$ for $\alpha<1$ (we show it for $\alpha=1/2$ though any $\alpha>0$ works, just with constant getting large as $\alpha\searrow 0$). This gives an estimate on the trace norm of the operator, which, by exponentiating, gives an upper bound $e^{c|\tilde\mu|^{\alpha}}$ on the size of the determinant. This upper bound is beat by the exponential decay in $\tilde\mu$ of the prefactor term $g_{\e}$.

For the proof of Proposition \ref{originally_cut_mu_lemma}, we do the same sort of factorization of our operator into $AB$, where here, 
\begin{equation*}
 A(\zeta,\eta)=\frac{e^{c[\Psi(\zeta)+n_0\log(\zeta)]}}{\zeta-\eta}
\end{equation*}
with $n_0$ as explained before the statement of Lemma \ref{kill_gamma_2_lemma}, and $0<c<1$ fixed,  and 
\begin{equation*}
B(\eta,\zeta) = e^{-c[\Psi(\zeta)+n_0\log(\zeta)]}e^{\Psi(\zeta)-\Psi(\eta)}\mu f(\mu,\zeta/\eta)\frac{1}{\eta}.     \end{equation*}
We must be careful in keeping track of the contours on which these operators act. As we have seen we may assume that the $\eta$ variables are on $\Gamma_{\eta,l}$ and the $\zeta$ variables on $\Gamma_{\zeta,l}$ for any fixed choice of $l\geq 0$. Now using the estimates of Lemmas \ref{kill_gamma_2_lemma} and \ref{compact_eta_zeta_taylor_lemma}, we compute that $||A||_2<\infty$ (uniformly in $\e<\e_0$ and, trivially, also in $\tilde\mu$). Here we calculate the Hilbert-Schmidt norm using Lemma \ref{fredholm_continuity_lemma}. Intuitively this norm is uniformly bounded as $\e$ goes to zero, because, while the denominator blows up as badly as $\e^{-1/2}$, the numerator is roughly supported only on a region of measure $\e^{1/2}$ (owing to the exponential decay of the exponential when $\zeta$ differs from $\xi$ by more than order $\e^{1/2}$).

We wish to control $||B||_2$ now. Using the discussion before Lemma  \ref{kill_gamma_2_lemma} we may rewrite $B$ as 
\begin{equation*}
B(\eta,\zeta) = e^{-c[\Psi(\zeta)+n_0\log(\zeta)]}e^{(\Psi(\zeta)+n_0\log(\zeta))-(\Psi(\eta)-n_0\log(\eta))}\tilde\mu f(\tilde\mu,\zeta/\eta)\frac{1}{\eta}
\end{equation*}
Lemmas \ref{kill_gamma_2_lemma} and \ref{compact_eta_zeta_taylor_lemma} apply and tell us that the exponential terms decay at least as fast as $\exp\{-\e^{-1/2}c'|\zeta-\eta|\}$. So the final ingredient in proving our proposition is  control of $|\tilde\mu f(\tilde\mu, z)|$ for $z=\zeta/\eta'$. We break it up into  two regions of $\eta',\zeta$: The first (1) when $|\eta'-\zeta|\leq c$ for a very small constant $c$ and the second (2) when $|\eta'-\zeta|> c$. We will compute $||B||_2$ as the square root of
\begin{equation}\label{case1case2}
 \int_{\eta,\zeta\in \textrm{Case (1)}} |B(\eta,\zeta)|^2 d\eta d\zeta + \int_{\eta,\zeta\in \textrm{Case (2)}} |B(\eta,\zeta)|^2 d\eta d\zeta.
\end{equation}
We will show that the first term can be bounded by $C|\tilde\mu|^{2\alpha}$ for any $\alpha<1$, while the second term can be bounded by a large constant. As a result $||B||_2\leq C|\tilde\mu|^{\alpha}$ which is exactly as desired since then $||AB||_1\leq e^{c|\tilde\mu|^{\alpha}}$.

Consider case (1) where $|\eta'-\zeta|\leq c$ for a constant $c$ which is positive but small (depending on $T$). One may easily check from the defintion of the contours that $\e^{-1/2}(|\zeta/\eta|-1)$ is contained in a compact subset of $(0,2)$. In fact, $\zeta/\eta'$ almost exactly lies along the curve $|z|=1+\e^{1/2}$ and in particular (by taking $\e_0$ and $c$ small enough) we can assume that $\zeta/\eta$ never leaves the region bounded by $|z|=1+(1\pm r)\e^{1/2}$ for any fixed $r<1$. Let us call this region $R_{\e,r}$. Then we have
\begin{lemma}\label{final_estimate}
Fix $\e_0$ and $r\in (0,1)$. Then for all $\e<\e_0$, $\tilde\mu\in\mathcal{\tilde C}_{\e}$ and $z\in R_{\e,r}$,
\begin{equation*}
|\tilde\mu f(\tilde\mu,z)| \leq c|\tilde\mu|^{\alpha}/|1-z|
\end{equation*}
for some $\alpha\in (0,1)$, with $c=c(\alpha)$  independent of $z$, $\tilde\mu$ and $\e$. 
\end{lemma}

\begin{remark}
By changing the value of $\alpha$ in the definition of $\kappa(\theta)$ (which then goes into the definition of $\Gamma_{\eta,l}$ and $\Gamma_{\zeta,l}$) and also focusing the region $R_{\e,r}$ around $|z|=1+2\alpha \e^{1/2}$, we can take $\alpha$ arbitrarily small in the above lemma at a cost of increasing the constant $c=c(\alpha)$ (the same also applies for Proposition \ref{reinclude_mu_lemma}). The $|\tilde\mu|^{\alpha}$ comes from the fact that $(1+2\alpha \e^{1/2})^{\tfrac{1}{2}\e^{-1/2}\log|\tilde\mu|} \approx |\tilde\mu|^{\alpha}$. Another remark is that the proof below can be used to provide an alternative proof of Lemma \ref{muf_compact_sets_csc_limit_lemma} by studying the convergence of the Riemann sum directly rather than by using functional equation properties of $f$ and the analytic continuations.
\end{remark}

We complete the ongoing proof of Proposition \ref{originally_cut_mu_lemma} and then return to the proof of the above lemma. 

Case (1) is now done since we can estimate the first integral in equation (\ref{case1case2}) using Lemma \ref{final_estimate} and the exponential decay of the exponential term outside of $|\eta'-\zeta|=O(\e^{1/2})$. Therefore, just as with the $A$ operator, the $\e^{-1/2}$ blowup of $|\tilde\mu f(\tilde\mu,\zeta/\eta')|$ is countered by the decay of the exponential and we are just left with a large constant time $|\tilde\mu|^{\alpha}$.

Turing to case (2) we need to show that the second integral in equation (\ref{case1case2}) is bounded uniformly in $\e$ and $\tilde\mu\in \tilde\C_{\e}$. This case corresponds to $|\eta'-\zeta|>c$ for some fixed but small constant $c$. Since $\e^{-1/2}(|\zeta/\eta|-1)$ stays bounded in a compact set, using an argument almost identical to the proof of Lemma \ref{mu_f_polynomial_bound_lemma} we can show that $|\tilde\mu f(\tilde\mu,\zeta/\eta)|$ can be bounded by $C|\tilde\mu|^{C'}$ for positive yet finite constants $C$ and $C'$. The important point here is that there is only a finite power of $|\tilde\mu|$. Since $|\tilde\mu|<\e^{-1/2}$ this means that this term can blow up at most polynomially in $\e^{-1/2}$. On the other hand we know that the exponential term decays exponentially fast like $e^{-\e^{-1/2}c}$ and hence the second integral in equation (\ref{case1case2}) goes to zero.

We now return to the proof of Lemma \ref{final_estimate} which will complete the proof of  Proposition~\ref{originally_cut_mu_lemma}.

\begin{proof}[Proof of Lemma \ref{final_estimate}]
We will prove the desired estimate for $z:|z|=1+\e^{1/2}$. The proof for general $z\in R_{\e,r}$ follows similarly.
Recall that 
\begin{equation*}
 \tilde\mu f(\tilde\mu,z) = \sum_{k=-\infty}^{\infty} \frac{\tilde\mu \tau^k}{1-\tilde\mu \tau^k} z^k.
\end{equation*}
Since $\tilde\mu$ has imaginary part 1, the denominator is smallest when $\tau^{k}=1/|\tilde\mu|$, corresponding to 
\begin{equation*}
 k=k^*=\lfloor \tfrac{1}{2}\e^{-1/2} \log |\mu|\rfloor.
\end{equation*}
We start, therefore, by centering our doubly infinite sum at around this value, 
\begin{equation*}
 \tilde\mu f(\tilde\mu,z) =   \sum_{k=-\infty}^{\infty} \frac{\tilde\mu \tau^{k^*}\tau^k}{1-\tilde\mu\tau^{k^*} \tau^k} z^{k^*}z^k.
\end{equation*}
By the definition of $k^*$, 
\begin{equation*}
|z|^{k^*}= |\tilde\mu|^{1/2}(1+O(\e^{1/2}))
\end{equation*}
thus we find that 
\begin{equation*}
|\tilde\mu f(\tilde\mu,z)| =   |\tilde\mu|^{1/2}\left|\sum_{k=-\infty}^{\infty} \frac{\kapi\tau^k}{1-\kapi \tau^k} z^k\right|
\end{equation*}
where \begin{equation*}
\kapi =\tilde\mu\tau^{k^*}
\end{equation*}
and is roughly on the unit circle except for a small dimple near 1. To be more precise, due to the rounding in the definition of $k^*$ the $\kapi$ is not exactly on the unit circle, however we do have the following two properties:
\begin{equation*}
 |1-\kapi|>\e^{1/2}, \qquad |\kapi|-1=O(\e^{1/2}).
\end{equation*}
The section of $\mathcal{\tilde C}_{\e}$ in which $\tilde\mu=\e^{-1/2}-1+iy$ for $y\in (-1,1)$ corresponds to $\kapi$ lying along a small dimple around $1$ (and still respects $|1-\kapi|>\e^{1/2}$). We call the curve on which $\kapi$ lies $\Omega$.

We can bring the $|\tilde\mu|^{1/2}$ factor to the left and split the summation into three parts, so that $|\tilde\mu|^{-1/2}|\tilde\mu f(\tilde\mu,z)|$ equals
\begin{equation}\label{three_term_eqn}
 \left|\sum_{k=-\infty}^{-\e^{-1/2}} \frac{\kapi\tau^k}{1-\kapi \tau^k} z^k+ \sum_{k=-\e^{-1/2}}^{\e^{-1/2}} \frac{\kapi\tau^k}{1-\kapi \tau^k} z^k+ \sum_{k=\e^{-1/2}}^{\infty} \frac{\kapi\tau^k}{1-\kapi \tau^k} z^k\right|.
\end{equation}
We will control each of these term separately. The first and the third are easiest. Consider 
\begin{equation*}
\left|(z-1)\sum_{k=-\infty}^{-\e^{-1/2}} \frac{\kapi\tau^k}{1-\kapi \tau^k} z^k\right|.
\end{equation*}
We wish to show this is bounded by a constant which is independent of $\tilde\mu$ and $\e$. Summing by parts  the argument of the absolute value can be written as
\begin{equation}\label{eqn174}
\frac{\kapi\tau^{-\e^{-1/2}+1}}{1-\kapi\tau^{-\e^{-1/2}+1}}z^{-\e^{-1/2}+1}+(1-\tau)\sum_{k=-\infty}^{-\e^{-1/2}}\frac{\kapi\tau^k}{(1-\kapi\tau^k)(1-\kapi\tau^{k+1})}z^k.
\end{equation}
We have $\tau^{-\e^{-1/2}+1} \approx e^2$ and $|z^{-\e^{-1/2}+1}|\approx e^{-1}$ (where $e\sim 2.718$). The denominator of the first term  is therefore bounded from zero. Thus the absolute value of this term is bounded by a constant. For the second term of (\ref{eqn174}) we can bring the absolute value inside of the summation to get
\begin{equation*}
 (1-\tau)\sum_{k=-\infty}^{-\e^{-1/2}}\left|\frac{\kapi\tau^k}{(1-\kapi\tau^k)(1-\kapi\tau^{k+1})}\right||z|^k.
\end{equation*}
The first term in absolute values stays bounded above by a constant times the value at $k=-\e^{-1/2}$. Therefore, replacing this by a constant, we can sum in $|z|$ and we get $\frac{|z|^{-\e^{-1/2}}}{1-1/|z|}$. The numerator, as noted before, is like $e^{-1}$ but the denominator is like $\e^{1/2}/2$. This is cancelled by the term $1-\tau=O(\e^{1/2})$ in front. Thus the absolute value is bounded.

The argument for the third term of equation (\ref{three_term_eqn}) works in the same way, except rather than multiplying by $|1-z|$ and showing the result is constant, we multiply by $|1-\tau z|$. This is, however, sufficient since $|1-\tau z|$ and $|1-z|$ are effectively the same for $z$ near 1 which is where our desired bound must be shown carefully.

We now turn to the middle term in equation (\ref{three_term_eqn}) which is more difficult.
We will show that
\begin{equation*}
 \left|(1-z)\sum_{k=-\e^{-1/2}}^{\e^{-1/2}} \frac{\kapi\tau^k}{1-\kapi \tau^k} z^k\right|=O(\log|\tilde\mu|).
\end{equation*}
This is of smaller order than $|\tilde\mu|$ raised to any positive real power and thus finishes the proof. For the sake of simplicity  we will first show this with $z=1+\e^{1/2}$. The general argument for points $z$ of the same radius and non-zero angle is very similar as we will observe at the end of the proof.
For the special choice of $z$,  the prefactor $(1-z)=\e^{1/2}$.

The method of proof  is to show that this sum is well approximated by a Riemann sum. This idea was mentioned in the formal proof of the $\e$ goes to zero limit. In fact, the argument below can be used to make that formal observation rigorous, and thus provides an alternative method to the complex analytic approach we take in the proof of  Lemma \ref{muf_compact_sets_csc_limit_lemma}. The sum we have is given by
\begin{equation}\label{first_step_sum}
\e^{1/2} \sum_{k=-\e^{-1/2}}^{\e^{-1/2}}\frac{\kapi\tau^k}{1-\kapi \tau^k} z^k = \e^{1/2} \sum_{k=-\e^{-1/2}}^{\e^{-1/2}}\frac{\kapi(1-\e^{1/2}+O(\e))^k}{1-\kapi (1-2\e^{1/2}+O(\e))^k}
\end{equation}
where we have used the fact that $\tau z = 1-\e^{1/2} + O(\e)$.
Observe that if $k=t\e^{-1/2}$ then this sum is close to a Riemann sum for
\begin{equation}\label{integral_equation_sigma}
 \int_{-1}^{1}\frac{\kapi e^{-t}}{1-\kapi e^{-2t}} dt.
\end{equation}
We use this formal relationship to prove that the sum in equation (\ref{first_step_sum}) is $O(\log|\tilde\mu|)$. We do this in a few steps. The first step is to consider the difference between each term in our sum and the analogous term in a Riemann sum for the integral. After estimating the difference we show that this can be summed over $k$ and gives us a finite error. The second step is to estimate the error of this Riemann sum approximation to the actual integral. The final step is to note that 
\begin{equation*}
\int_{-1}^{1} \frac{\kapi e^{-t}}{1-\kapi e^{-2t}}dt \sim |\log(1-\kapi)|\sim \log|\tilde\mu|
\end{equation*}
for $\kapi\in \Omega$ (in particular where $|1-\kapi|>\e^{1/2}$).  Hence it is easy to check that it is smaller than any power of $|\tilde\mu|$.

A single term in the Riemann sum for the integral looks like 
$
 \e^{1/2} \frac{\kapi e^{-k\e^{1/2}}}{1-\kapi e^{-2k\e^{1/2}}}
$. Thus we are interested in estimating 
\begin{equation}\label{estimating_eqn}
 \e^{1/2}\left|\frac{\kapi(1-\e^{1/2}+O(\e))^k}{1-\kapi (1-2\e^{1/2}+O(\e))^k} - \frac{\kapi e^{-k\e^{1/2}}}{1-\kapi e^{-2k\e^{1/2}}} \right|.
\end{equation}
We claim that there exists $C<\infty$, independent of $\epsilon$ and $k$ satisfying $k\e^{1/2}\leq 1$, such that the previous line is bounded above by
\begin{equation}\label{giac}
  \frac{Ck^2\e^{3/2}}{(1-\kapi+\kapi 2k\e^{1/2})}+\frac{Ck^3\e^{2}}{(1-\kapi+\kapi 2k\e^{1/2})^2}.
\end{equation}
To prove that (\ref{estimating_eqn}) $\le $(\ref{giac}) we expand the powers of $k$ and the exponentials. For the numerator and denominator of the first term inside of the absolute value in  (\ref{estimating_eqn}) we have $\kapi(1-\e^{1/2}+O(\e))^k= \kapi-\kapi k\e^{1/2} + O(k^2\e)$ and 
\begin{eqnarray*}
\nonumber 1-\kapi (1-2\e^{1/2}+O(\e))^k & = & 1-\kapi+\kapi 2k\e^{1/2} -\kapi 2k^2\e+O(k\e)+O(k^3\e^{3/2})\\ & = &
(1-\kapi+\kapi 2k\e^{1/2})(1 - \frac{\kapi 2k^2\e +O(k\e)+O(k^3\e^{3/2})}{1-\kapi+\kapi 2k \e^{1/2}}).
\end{eqnarray*}
Using $1/(1-z)=1+z+O(z^2)$ for $|z|<1$ we see that 
\begin{eqnarray*}
&& \frac{\kapi(1-\e^{1/2}+O(\e))^k}{1-\kapi (1-2\e^{1/2}+O(\e))^k}\\
\nonumber&=&\frac{\kapi-\kapi k\e^{1/2} + O(k^2\e)}{1-\kapi+\kapi 2k\e^{1/2}}\left(1+\frac{\kapi 2k^2\e +O(k\e)+O(k^3\e^{3/2})}{1-\kapi+\kapi 2k \e^{1/2}}\right)\\
\nonumber& =& 
\frac{\left(\kapi-\kapi k \e^{1/2} + O(k^2\e)\right)\left(1-\kapi+\kapi 2k \e^{1/2} + \kapi 2k^2\e +O(k\e)+O(k^3\e^{3/2})\right)}{(1-\kapi+\kapi 2k\e^{1/2})^2}
\end{eqnarray*}
Likewise, the second term from equation (\ref{estimating_eqn}) can be similarly estimated and shown to be
\begin{equation*}
\frac{\kapi e^{-k\e^{1/2}}}{1-\kapi e^{-2k\e^{1/2}}}= \frac{\left(\kapi-\kapi k\e^{1/2} + O(k^2\e)\right)\left(1-\kapi+\kapi 2k \e^{1/2} + \kapi 2k^2\e +O(k^3\e^{3/2})\right)}{(1-\kapi+\kapi 2k\e^{1/2})^2}.
\end{equation*}
Taking the difference of these two terms, and noting the cancellation of a number of the terms in the numerator, gives (\ref{giac}).

To see that the  error in (\ref{giac}) is bounded after the summation over $k$ in the range $\{-\e^{-1/2},\ldots, \e^{-1/2}\}$, note that  this gives
\begin{eqnarray*}
\e^{1/2}\sum_{-\e^{-1/2}}^{\e^{1/2}}\frac{(2k\e^{1/2})^2}{1-\kapi+\kapi (2k\e^{1/2})}+\frac{(2k\e^{1/2})^3}{(1-\kapi+\kapi (2k\e^{1/2}))^2}\\
\nonumber\sim\int_{-1}^{1}\frac{(2t)^2}{1-\kapi+\kapi 2t} +\frac{(2t)^3}{(1-\kapi+\kapi 2t)^2} dt . 
\end{eqnarray*}
The Riemann sums and integrals
 are easily shown to be convergent for our $\kapi$ which  lies on $\Omega$, which is roughly the unit circle, and avoids the point 1 by distance $\e^{1/2}$.

Having completed this first step, we now must show that the Riemann sum for the integral in equation  (\ref{integral_equation_sigma}) converges to the integral. This uses the following estimate,\begin{equation}\label{riemann_approx_max}
\sum_{k=-\e^{-1/2}}^{\e^{-1/2}} \e^{1/2} \max_{(k-1/2)\e^{1/2}\leq t\leq (k+1/2)\e^{1/2}} \left| \frac{\kapi e^{-k\e^{1/2}}}{1-\kapi e^{-2k\e^{1/2}}} - \frac{\kapi e^{-t}}{1-\kapi e^{-2t}}\right|\le C
\end{equation}
 To show this, observe that for $t\in \e^{1/2}[k-1/2,k+1/2]$ we can expand the second fraction as 
\begin{equation}\label{sec_frac}
 \frac{\kapi e^{-k\e^{1/2}}(1+O(\e^{1/2}))}{1-\kapi e^{-2k\e^{1/2}}(1-2l\e^{1/2}+O(\e))} 
\end{equation}
where $l\in [-1/2,1/2]$.
Factoring the denominator as 
\begin{equation}
 (1-\kapi e^{-2k\e^{1/2}})(1+ \frac{\kapi e^{-2k\e^{1/2}}(2l\e^{1/2}+O(\e))}{1-\kapi e^{-2k\e^{1/2}}})
\end{equation}
we can use $1/(1+z)=1-z+O(z^2)$ (valid since $|1-\kapi e^{-2k\e^{1/2}}|>\e^{1/2}$ and $|l|\leq 1$) to rewrite equation (\ref{sec_frac}) as 
\begin{equation*}
 \frac{\kapi e^{-k\e^{1/2}}(1+O(\e^{1/2}))\left(1- \frac{\kapi e^{-2k\e^{1/2}}(2l\e^{1/2}+O(\e))}{1-\kapi e^{-2k\e^{1/2}}} \right)}{1-\kapi e^{-2k\e^{1/2}}}.
\end{equation*}
Canceling terms in this expression with the terms in the first part of equation (\ref{riemann_approx_max}) we find that we are left with terms bounded by
\begin{equation*}
 \frac{O(\e^{1/2})}{1-\kapi e^{-2k\e^{1/2}}} + \frac{O(\e^{1/2})}{(1-\kapi e^{-2k\e^{1/2}})^2}. 
\end{equation*}
These must be summed over $k$ and multiplied by the prefactor $\e^{1/2}$. Summing over $k$ we find that these are approximated by the integrals
\begin{equation*}
\e^{1/2}\int_{-1}^{1}\frac{1}{1-\kapi +\kapi 2t}dt,\qquad \e^{1/2}\int_{-1}^{1}\frac{1}{(1-\kapi +\kapi 2t)^2}dt
\end{equation*}
where $|1-\kapi|>\e^{1/2}$. The first integral has a logarithmic singularity at $t=0$ which gives $|\log(1-\kapi)|$ which is clearly bounded by a constant time $|\log\e^{1/2}|$ for $\kapi\in \Omega$. When multiplied by $\e^{1/2}$ this term is clearly bounded in $\e$. Likewise, the second integral diverges like $|1/(1-\kapi)|$ which is bounded by $\e^{-1/2}$ and again multiplying by the $\e^{1/2}$ factor in front shows that this term is bounded. This proves the Riemann sum approximation.

This estimate completes the proof of the desired bound when $z=1+\e^{1/2}$. The general case of $|z|=1+\e^{1/2}$ is proved along a similar line by letting $z= 1+\rho\e^{1/2}$ for $\rho$ on a suitably defined contour such that $z$ lies on the circle of radius $1+\e^{1/2}$. The prefactor is no longer $\e^{1/2}$ but rather now $\rho\e^{1/2}$ and all estimates must take into account $\rho$. However, going through this carefully one finds that the same sort of estimates as above hold and hence the theorem is proved in general.
\end{proof}
This lemma completes the proof of Proposition \ref{originally_cut_mu_lemma} 
\end{proof}

\begin{proof}[Proof of Proposition \ref{reinclude_mu_lemma}]
We will focus on the growth of the absolute value of the determinant. Recall that if $K$ is trace class then $|\det(I+K)|\leq e^{||K||_1}$. Furthermore, if $K$ can be factored into the product $K=AB$ where $A$ and $B$ are Hilbert-Schmidt, then $||K||_1\leq ||A||_2||B||_2$. We will demonstrate such a factorization and follow this approach to control the size of the determinant.

Define $A:L^2(\tilde\Gamma_{\zeta})\rightarrow L^2(\tilde\Gamma_{\eta})$ and $B:L^2(\tilde\Gamma_{\eta})\rightarrow L^2(\tilde\Gamma_{\zeta})$ via the kernels
\begin{eqnarray*}  && 
A(\tilde\zeta,\tilde\eta) = \frac{e^{-|\im(\tilde\zeta)|}}{\tilde\zeta-\tilde\eta}, \\  B(\tilde\eta,\tilde\zeta)& =& \nonumber e^{|\im(\tilde\zeta)|}e^{-\frac{T}{3}(\tilde\zeta^3-\tilde\eta^3)+a\tilde z} 2^{1/3}\frac{\pi (-\tilde\mu)^{\tilde z}}{\sin(\pi \tilde z)},
\end{eqnarray*}
where we let $\tilde z = 2^{1/3}(\tilde\zeta-\tilde\eta)$. Notice that we have put the factor $e^{-|\im(\tilde\zeta)|}$ into the $A$ kernel and removed it from the $B$ contour. The point of this is to help control the $A$ kernel, without significantly impacting the norm of the $B$ kernel.

Consider first $||A||_2$ which is given by
\begin{equation*}
 ||A||_2^2 = \int_{\tilde\Gamma_{\zeta}}\int_{\tilde\Gamma_{\eta}} d\tilde\zeta d\tilde\eta \frac{e^{-2|\im(\tilde\zeta)|}}{|\tilde\zeta-\tilde\eta|^2}.
\end{equation*}
The integral in $\tilde\eta$ converges and is independent of $\tilde\zeta$ (recall that $|\tilde\zeta-\tilde\eta|$ is bounded away from zero) while the remaining integral in $\tilde\zeta$ is clearly convergent (it is exponentially small as $\tilde\zeta$ goes away from zero along $\tilde\Gamma_{\zeta}$. Thus $||A||_{2}<c$ with no dependence on $\tilde\mu$ at all.

We now turn to computing $||B||_2$. First consider the cubic term $\tilde\zeta^3$. The contour $\tilde\Gamma_{\zeta}$ is parametrized by $-\frac{c_3}{2} + c_3 i r$ for $r\in (-\infty,\infty)$, that is, a straight up and down line just to the left of the $y$ axis. By plugging this parametrization in and cubing it, we see that, $\re(\tilde\zeta^3)$ behaves like $|\im(\tilde\zeta)|^2$. This is  crucial; even though our contours are parallel and only differ horizontally by a small distance, their relative locations lead to very different behavior for the real part of their cube. For $\tilde\eta$ on the right of the $y$ axis, the real part still grows quadratically, however with a negative sign. This is important because this implies that $|e^{-\frac{T}{3}(\tilde\zeta^3-\tilde\eta^3)}|$ behaves like the exponential of the real part of the argument, which is to say, like
\begin{equation*}
e^{-\frac{T}{3}(|\im(\tilde\zeta)|^2+|\im(\tilde\eta)|^2)}.
\end{equation*}
Turning to the $\tilde\mu$ term, observe that
\begin{eqnarray*}
 |(-\tilde\mu)^{-\tilde z}| &=& e^{\re\left[(\log|\tilde\mu|+i\arg(-\tilde\mu))(-\re(\tilde z)-i\im(\tilde z))\right]}\\
&=& e^{-\log|\tilde\mu| \re(\tilde z) +\arg(-\tilde\mu)\im(\tilde z)}.
\end{eqnarray*}
The $\csc$ term behaves, for large $\im(\tilde z)$ like $e^{-\pi |\im(\tilde z)|}$, and putting all these estimates together gives that for $\tilde\zeta$ and $\tilde\eta$ far from the origin on their respective contours, $|B(\tilde\eta,\tilde\zeta)|$ behaves like the following product of exponentials:
\begin{equation*}
e^{|\im(\tilde\zeta)|} e^{-\frac{T}{3}(|\im(\tilde\zeta)|^2+|\im(\tilde\eta)|^2)} e^{-\log|\tilde\mu|\re(\tilde z) + \arg(-\tilde\mu)\im(\tilde z) - \pi |\im(\tilde z)|}.
\end{equation*}
Now observe that, due to the location of the contours, $-\re(\tilde z)$ is constant and less than one (in fact equal to $1/2$ by our choice of contours). Therefore we may factor out the term $e^{-\log|\tilde\mu|\re(\tilde z)} = |\tilde\mu|^\alpha$ for $\alpha=1/2<1$.

The Hilbert-Schmidt norm of what remains is clearly finite and independent of $\tilde\mu$.  This is just due to the strong exponential decay from the quadratic terms  $-\im(\zeta)^2$ and $-\im(\eta)^2$ in the exponential. Therefore we find that $||B||_2\leq   c|\tilde\mu|^\alpha$ for some constant $c$.

This shows that $||K_a^{\csc}||_1$ behaves like $|\tilde\mu|^{\alpha}$ for $\alpha<1$. Using the bound that $|\det(I+K_a^{\csc})|\leq e^{||K_a^{\csc}||}$ we find that $|\det(I+K_a^{\csc})|\leq e^{|\tilde\mu|^{\alpha}}$. Comparing this to $e^{-\tilde\mu}$ we have our desired result. Note that the proof also shows that $K_a^{\csc}$ is trace class.
\end{proof}

\subsubsection{Proofs from Section \ref{J_to_K_sec}}\label{JK_proofs_sec}
\begin{proof}[Proof of Lemma \ref{kill_gamma_2_lemma}]
Before starting this proof, we remark that the choice (\ref{kappa_eqn}) of  $\kappa(\theta)$ was specifically to make the calculations in this proof more tractable. Certainly other choices of contours would do, however, the estimates could be harder. As it is, we used Mathematica as a preliminary tool to assist us in computing the series expansions and simplifying the resulting expressions.

Define $g(\eta)=\Psi(\eta)+n_0\log(\eta)$. We wish to control the real part of this function for both the $\eta$ contour and the $\zeta$ contour. Combining these estimates proves the lemma.

We may expand $g(\eta)$ into powers of $\e$ with the expression for $\eta$ in terms of $\kappa(\theta)$ from (\ref{kappa_eqn})  with $\alpha=-1/2$ (similarly $1/2$ for the $\zeta$ expansion). Doing this we see that the $n_0\log(\eta)$ term plays an important role in canceling the $\log(\e)$ term in the $\Psi$ and we are left with
\begin{equation}\label{psi_real_eqn}
\re(g(\eta)) =\e^{-1/2}\left( -{\scriptstyle\frac14}  \e^{-1/2}T\alpha \cot^2(\tfrac{\theta}{2}) + {\scriptstyle\frac18}T\left[\alpha+\kappa(\theta)\right]^2 \cot^2(\tfrac{\theta}{2}) \right)+ O(1).
\end{equation}
We must show that everything in the parenthesis above is bounded below by a positive constant times $|\eta-\xi|$ for all $\eta$ which start at roughly angle $l\e^{1/2}$. Equivalently we can show that the terms in the parenthesis behave bounded below by a positive constant times $|\pi-\theta|$, where $\theta$ is the polar angle of $\eta$.

The second part of this expression is clearly positive regardless of the value of $\alpha$. What this suggests is that we must show (in order to also be able to deal with $\alpha=1/2$ corresponding to the $\zeta$ estimate) that for $\eta$ starting at angle $l\e^{1/2}$ and going to zero, the first term dominates (if $l$ is large enough).

To see this we first note that since $\alpha=-1/2$, the first term is clearly positive and dominates for $\theta$ bounded away from $\pi$. This proves the inequality for any range of $\eta$ with $\theta$ bounded from $\pi$. Now observe the following asymptotic behavior of the following three functions of $\theta$ as $\theta$ goes to $\pi$:
\begin{eqnarray*}
\cot^2(\tfrac{\theta}{2}) &\approx& {\scriptstyle\frac{1}{4}}(\pi-\theta)^2\\
\tan^2(\tfrac{\theta}{2})&\approx& {4}(\pi-\theta)^{-2}\\
\log^2\left({\scriptstyle\frac{2}{1-\cos(\theta)}}\right) &\approx&{\scriptstyle \frac{1}{16} }(\pi-\theta)^4.
\end{eqnarray*}
The behaviour expressed above is dominant for $\theta$ close to $\pi$. We may expand the square in the second term in (\ref{psi_real_eqn}) and use the above expressions to find that for some suitable constant $C>0$ (which depends on $X$ and $T$ only), we have
\begin{equation*}
\re(g(\eta)) = \e^{-1/2}\left(-{\scriptstyle\frac1{16} }\e^{-1/2} T\alpha (\pi-\theta)^2 + C(\pi-\theta)^2\right) + O(1).
\end{equation*}
Now use the fact that $\pi-\theta\geq l\e^{1/2}$ to give
\begin{equation}\label{eqn_g}
\re(g(\eta)) = \e^{-1/2}\left(-{\scriptstyle\frac1{16} }lT\alpha (\pi-\theta) +{\scriptstyle \frac{X^2}{8T}}(\pi-\theta)^2\right) + O(1).
\end{equation}
Since $\pi-\theta$ is bounded by $\pi$, we see that taking $l$ large enough, the first term always dominates for the entire range of $\theta\in [0,\pi-l\e^{1/2}]$. Therefore since $\alpha=-1/2$, we find that we have have the desired lower bound in $\e^{-1/2}$ and $|\pi-\theta|$.

Turn now to the bound for $\re(g(\zeta))$. In the case of the $\eta$ contour we took $\alpha=-1/2$, however since we now are dealing with the $\zeta$ contour we must take $\alpha=1/2$. This change in the sign of $\alpha$ and the argument above shows that equation (\ref{eqn_g}) implies the desired bound for $\re(g(\zeta))$, for $l$ large enough.
\end{proof}

Before proving Lemma \ref{mu_f_polynomial_bound_lemma} we record the following key lemma on the meromorphic extension of $\mu f(\mu,z)$. Recall that $\mu f(\mu,z)$ has poles at $\mu= \tau^j$, $j\in \Z$.

\begin{lemma}\label{f_functional_eqn_lemma}
For $\mu\neq \tau^j$, $j\in \Z$,  $\mu f(\mu,z)$ is analytic in $z$ for $1<|z|<\tau^{-1}$ and extends analytically to all $z\neq 0$ or $\tau^k$ for $k\in \Z$. This extension is given by first writing $\mu f(\mu,z) = g_+(z)+g_-(z)$ where
\begin{equation*}
 g_+(z)=\sum_{k=0}^{\infty} \frac{\mu \tau^kz^k}{1-\tau^k\mu}\qquad\qquad g_-(z)=\sum_{k=1}^{\infty} \frac{\mu \tau^{-k}z^{-k}}{1-\tau^{-k}\mu},
\end{equation*}
and where $g_+$ is now defined for $|z|<\tau^{-1}$ and $g_-$ is defined for $|z|>1$. These functions satisfy the following two functional equations which imply the analytic continuation:
\begin{equation*}
 g_+(z)=\frac{\mu}{1-\tau z}+\mu g_+(\tau z),\qquad g_-(z)=\frac{1}{1-z}+\frac{1}{\mu}g_-(z/\tau).
\end{equation*}
By repeating this functional equation we find that
\begin{equation*}
 g_+(z)=\sum_{k=1}^{N}\frac{\mu^k}{1-\tau^k z}+\mu^N g_+(\tau^N z),\qquad g_-(z)=\sum_{k=0}^{N-1}\frac{\mu^{-k}}{1-\tau^{-k}z}+\mu^{-N}g_-(z\tau^{-N}).
\end{equation*}
\end{lemma}

\begin{proof}
We prove the $g_+$ functional equation, since the $g_-$ one follows similarly. Observe that
\begin{eqnarray*}
\nonumber  g_+(z) & =&  \sum_{k=0}^{\infty} \mu(\tau z)^k ( 1+\frac{1}{1-\mu\tau^k} -1) \\&= &\frac{\mu}{1-\tau z}+\sum_{k=0}^{\infty} \frac{\mu^2\tau^k}{1-\mu\tau^k} (\tau z)^k = \frac{\mu}{1-\tau z} + \mu g_+(\tau z),
\end{eqnarray*}
which is the desired relation.
\end{proof}

\begin{proof}[Proof of Lemma \ref{mu_f_polynomial_bound_lemma}]
Recall that $\tilde\mu$ lies on a compact subset of $\mathcal{\tilde C}$ and hence that $|1-\tilde\mu \tau^k|$ stays bounded from below as $k$ varies. Also observe that due to our choices of contours for $\eta'$ and $\zeta$, $|\zeta/\eta'|$ stays bounded in $(1,\tau^{-1})$. Write $z=\zeta/\eta'$. Split  $\tilde\mu f(\tilde\mu ,z)$ as $g_+(z)+g_-(z)$ (see Lemma \ref{f_functional_eqn_lemma} above), we see that $g_+(z)$ is bounded by a constant time $1/(1-\tau z)$ and likewise $g_-(z)$ is bounded by a constant time $1/(1-z)$. Writing this in terms of $\zeta$ and $\eta'$ again we have our desired upperbound.
\end{proof}

\begin{proof}[Proof of Lemma \ref{compact_eta_zeta_taylor_lemma}]
By the discussion preceding the statement of this lemma it suffices to consider the expansion without $n_0\log(\zeta/\eta')$ and without the $\log\e$ term in $m$ since, as we will see, they exactly cancel out. Therefore, for the sake of this proof we modify the definition of $m$ given in equation (\ref{m_eqn}) to be
\begin{equation*}
 m=\frac{1}{2}\left[\e^{-1/2}(-a'+\frac{X^2}{2T})+\frac{1}{2}t+x\right].
\end{equation*}
where $a'=a+\log2$.

The argument now amounts to a Taylor series expansion with control over the remainder term.  Let us start by recording the first four derivatives of $\Lambda(\zeta)$:
\begin{eqnarray*}
 \Lambda(\zeta) &=& -x\log(1-\zeta)+\frac{t\zeta}{1-\zeta} + m\log \zeta\\
 \Lambda'(\zeta) &=& \frac{x}{1-\zeta}+\frac{t}{(1-\zeta)^2}+\frac{m}{\zeta}\\
 \Lambda''(\zeta) &=& \frac{x}{(1-\zeta)^2}+\frac{2t}{(1-\zeta)^3}-\frac{m}{\zeta^2}\\
 \Lambda'''(\zeta) &=& \frac{2x}{(1-\zeta)^3}+\frac{6t}{(1-\zeta)^4}+\frac{2m}{\zeta^3}\\
 \Lambda''''(\zeta) &=& \frac{6x}{(1-\zeta)^4}+\frac{24 t}{(1-\zeta)^5} -\frac{6m}{\zeta^4}.
\end{eqnarray*}
We Taylor expand $\Psi(\zeta)=\Lambda(\zeta)-\Lambda(\xi)$ around $\xi$ and then expand in $\e$ as $\e$ goes to zero and find that
\begin{eqnarray*}
 \Lambda'(\xi) &=& \tfrac{a'}{2}\e^{-1/2} + O(1)\\
 \Lambda''(\xi) &=& O(\e^{-1/2})\\
 \Lambda'''(\xi) &=& \tfrac{-T}{8} \e^{-3/2} +O(\e^{-1})\\
 \Lambda''''(\xi) &=& O(\e^{-3/2}).
\end{eqnarray*}
A Taylor series remainder estimate shows then that
\begin{eqnarray*}
\nonumber&& \left|\Psi(\zeta) - \left[\Lambda'(\xi)(\zeta-\xi)+{\scriptstyle\frac1{2!}}\Lambda''(\xi)(\zeta-\xi)^2 + {\scriptstyle\frac1{3!}}\Lambda'''(\xi)(\zeta-\xi)^3\right]\right|\\&&\quad \leq \sup_{t\in B(\xi,|\zeta-\xi|)} {\scriptstyle\frac1{4!}}|\Lambda''''(t)| |\zeta-\xi|^4,
\end{eqnarray*}
where $B(\xi,|\zeta-\xi|)$ denotes the ball around $\xi$ of radius $|\zeta-\xi|$. Now considering the scaling we have that $\zeta-\xi = c_3^{-1}\e^{1/2}\tilde\zeta$ so that when we plug this in along with the estimates on derivatives of $\Lambda$ at $\xi$, we find that the equation above becomes
\begin{equation*}
 \left|\Psi(\zeta) - \left[2^{1/3}a'\tilde\zeta -\tfrac{T}{3}\tilde\zeta^3 \right]\right|=O(\e^{1/2}).
\end{equation*}
From this we see that if we included the $\log\e$ term in with $m$ it would, as claimed, exactly cancel the $n_0\log(\zeta/\eta')$ term. The above estimate therefore proves the desired first claimed result.

The second result follows readily from
 $|e^z-e^w|\leq |z-w|\max\{|e^z|,|e^w|\}$ and the first result, as well as the boundedness of the limiting integrand.
\end{proof}

\begin{proof}[Proof of Lemma \ref{muf_compact_sets_csc_limit_lemma}]

Expanding in $\e$ we have that
\begin{equation*}
 z=\frac{\xi+c_{3}^{-1}\e^{1/2}\tilde\zeta}{\xi+c_{3}^{-1}\e^{1/2}\tilde\eta'} = 1-\e^{1/2}\tilde z + O(\e)
\end{equation*}
where the error is uniform for our range of $\tilde\eta'$ and $\tilde\zeta$ and where
\begin{equation*}
 \tilde z = c_{3}^{-1}(\tilde\zeta-\tilde\eta').
\end{equation*}
We now appeal to the functional equation for $f$,  explained in Lemma \ref{f_functional_eqn_lemma}. Therefore we wish to study $\e^{1/2}g_{+}(z)$ and $\e^{1/2}g_{-}(z)$ as $\e$ goes to 0 and show that they converge uniformly to suitable integrals. First consider the $g_{+}$ case.
Let us, for the moment, assume that $|\tilde\mu|<1$. We know that $|\tau z|<1$, thus for any $N\geq 0$, we have
\begin{equation*}
\e^{1/2} g_{+}(z) = \e^{1/2}\sum_{k=1}^{N} \frac{\tilde\mu^k}{1-\tau^k z} + \e^{1/2}\tilde\mu^N g_{+}(\tau^N z).
\end{equation*}
Since, by assumption, $|\tilde\mu|<1$, the first sum is the partial sum of a convergent series. Each term may be expanded in $\e$. Noting that
\begin{equation*}
 1-\tau^k z = 1-(1-2\e^{1/2}+O(\e))(1-\e^{1/2}\tilde z +O(\e)) = (2k+\tilde z)\e^{1/2} +kO(\e),
\end{equation*}
we find that
\begin{equation*}
\e^{1/2}\frac{\tilde\mu^k}{1-\tau^k z} = \frac{\tilde\mu^k}{2k+\tilde z} + k O(\e^{1/2}).
\end{equation*}
The last part of the expression for $g_{+}$ is bounded in $\e$, thus we end up with the following asymptotics
\begin{equation*}
\e^{1/2} g_{+}(z) = \sum_{k=1}^{N} \frac{\tilde\mu^k}{2k+\tilde z} + N^2 O(\e^{1/2}) + \tilde\mu^N O(1).
\end{equation*}
It is possible to choose $N(\e)$ which goes to infinity, such that $N^2 O(\e^{1/2}) = o(1)$. Then for any fixed compact set contained in $\C\setminus \{-2,-4,-6,\ldots\}$  we have uniform convergence of this sequence of analytic functions to some function, which is necessarily analytic and equals
\begin{equation*}
\sum_{k=1}^{\infty} \frac{\tilde\mu^k}{2k+\tilde z}.
\end{equation*}
This expansion is valid for $|\tilde\mu|<1$ and for all $\tilde z\in \C\setminus\{-2,-4,-6,\ldots\}$.

Likewise for $\e^{1/2}g_{-}(z)$, for $|\tilde\mu|>1$ and for $\tilde z\in \C\setminus\{-2,-4,-6,\ldots\}$, we have uniform convergence to the analytic function
\begin{equation*}
\sum_{k=-\infty}^{0} \frac{\tilde\mu^k}{2k+\tilde z}.
\end{equation*}

We now introduce the Hurwitz Lerch transcendental function and relate some basic properties of it which can be found in \cite{SC:2001s}.
\begin{equation*}
\Phi(a,s,w) = \sum_{k=0}^{\infty} \frac{a^k}{(w+k)^s}
\end{equation*}
for $w>0$ real and either $|a|<1$ and $s\in \C$ or $|a|=1$ and $\re(s)>1$. For $\re(s)>0$ it is possible to analytically extend this function using the integral formula
\begin{equation*}
\Phi(a,s,w) = \frac{1}{\Gamma(s)} \int_0^{\infty} \frac{e^{-(w-1)t}}{e^t-a} t^{s-1} dt,
\end{equation*}
where additionally $a\in \C\setminus[1,\infty)$ and $\re(w)>0$.

Observe that we can express our series in terms of this function as
\begin{eqnarray*}
\sum_{k=1}^{\infty} \frac{\tilde\mu^k}{2k+\tilde z} = \frac{1}{2}\tilde \mu  \Phi(\tilde \mu , 1, 1+\tilde z/2),\\
\sum_{k=-\infty}^{0} \frac{\tilde\mu^k}{2k-\tilde z} = -\frac{1}{2}\Phi(\tilde \mu^{-1} , 1, -\tilde z/2).
\end{eqnarray*}
These two functions can be analytically continued using the integral formula onto the same region where $\re(1+\tilde z/2)>0$ and $\re(-\tilde z/2)>0$ -- i.e. where $\re(\tilde z/2)\in (-1,0)$. Additionally the analytic continuation is valid for all $\tilde\mu$ not along $\R^+$.

We wish now to use Vitali's convergence theorem to conclude that $\tilde\mu f(\tilde\mu, z)$  converges uniformly for general $\tilde\mu$ to the sum of these two analytic continuations. In order to do that we need a priori boundedness of $\e^{1/2}g_+$ and $\e^{1/2}g_-$ for compact regions of $\tilde\mu$ away from $\R^+$. This, however, can be shown directly as follows. By assumption on $\tilde\mu$ we have that $|1-\tau^k \tilde\mu|>c^{-1}$ for some positive constant $c$. Consider $\e^{1/2}g_+$ first.
\begin{equation*}
 |\e^{1/2}g_+(z)| \leq \e^{1/2}\tilde\mu \sum_{k=0}^{\infty} \frac{|\tau z|^k}{|1-\tau^k \tilde\mu|} \leq c\e^{1/2} \frac{1}{1-|\tau z|}.
\end{equation*}
We know that $|\tau z|$ is bounded to order $\e^{1/2}$ away from $1$ and therefore this show that $|\e^{1/2}g_+(z)|$ has an upperbound uniform in $\tilde\mu$. Likewise we can do a similar computation for $\e^{1/2}g_-(z)$ and find the same result, this time using that $|z|$ is bounded to order $\e^{1/2}$ away from $1$.

As a result of this apriori boundedness, uniform in $\tilde\mu$, we have that for compact sets of $\tilde\mu$ away from $\R^+$, uniformly in $\e$, $\e^{1/2}g_+$ and $\e^{1/2}g_-$ are uniformly bounded as $\e$ goes to zero. Therefore Vitali's convergence theorem implies that they converge uniformly to their analytic continuation.

Now observe that
\begin{equation*}
\frac{1}{2}\tilde \mu \Phi(\tilde \mu ,1, 1+\tilde z/2) = \frac{1}{2}\int_0^{\infty} \frac{\tilde\mu e^{-\tilde z t/2}}{e^t-\tilde \mu}dt,
\end{equation*}
and 
\begin{equation*}
-\frac{1}{2}\Phi(\tilde\mu^{-1}, 1,-\tilde z/2)= -\frac{1}{2}\int_0^{\infty} \frac{e^{-(-\tilde z/2 -1)t}}{e^{t} -1/\tilde{\mu}}dt= \frac{1}{2}\int_{-\infty}^{0}  \frac{\tilde\mu e^{-\tilde z t/2}}{e^{t}-\tilde \mu}dt.
\end{equation*}
Therefore, by a simple change of variables in the second integral, we can combine these as a single integral
\begin{equation*}
\frac{1}{2}\int_{-\infty}^{\infty} \frac{\tilde\mu e^{-\tilde z t/2}}{e^t-\tilde \mu}dt  = \frac{1}{2}\int_0^{\infty} \frac{\tilde\mu s^{-\tilde z/2}}{s-\tilde\mu} \frac{ds}{s}.
\end{equation*}
The first of the above equations proves the lemma, and for an alternative expression we use the second of the integrals (which followed from the change of variables $e^t=s$) and thus, on the region $\re(\tilde z/2)\in (-1,0)$ this integral converges and equals
\begin{equation*}
{\scriptstyle\frac{1}{2}}\pi (-\tilde\mu)^{-\tilde z} \csc(\pi \tilde z/2).
\end{equation*}
This function is, in fact, analytic for $\tilde\mu\in \C\setminus[0,\infty)$ and for all $\tilde z\in \C \setminus 2\Z$. Therefore it is the analytic continuation of our asymptotic series.
\end{proof}

\section{Weakly asymmetric limit of the corner growth model}\label{BG}

Recall the definitions in Section \ref{asepscalingth} of WASEP, its height function  (\ref{defofheight}), and, for $X\in  \e\mathbb Z$ 
and $T\ge  0$,
\begin{equation}\label{scaledhgt}
Z_\eps(T,X) =\tfrac{1}{2} \eps^{-1/2}\exp\left \{ - \lambda_\eps h_{\eps^{1/2}}(\eps^{-2}T,[\eps^{-1}X]) + \nu_\eps \eps^{-2}T\right\}
\end{equation}
where, for  $\epsilon\in(0,1/4)$, let $p ={\scriptstyle\frac12} - {\scriptstyle\frac12} \epsilon^{1/2}$, $q ={\scriptstyle\frac12} + {\scriptstyle\frac12}\epsilon^{1/2}$ and $\nu_\e$ and $\lambda_\e$ are as in (\ref{nu}) and (\ref{lambda}), and the closest integer $[x]$ is given by 
\begin{equation*}
[x] = \lfloor x+\tfrac12\, \rfloor.
\end{equation*}   

 Let us describe in simple terms the dynamics in $T$ of $Z_\eps(T,X)$
defined in (\ref{scaledhgt}).   It grows continuously exponentially at rate $\eps^{-2}\nu_\eps $ and jumps at rates \begin{equation*}
r_-(X)=\epsilon^{-2}q(1-\eta(x))\eta(x+1)= \frac14\eps^{-2}q(1-\hat\eta(x))(1+\hat\eta(x+1))  \end{equation*}
to $e^{-2\lambda_\epsilon}Z_\eps$ and  \begin{equation*}r_+(X)=\epsilon^{-2}p\eta(x)(1-\eta(x+1))=\frac14\eps^{-2}p(1+\hat\eta(x))(1-\hat\eta(x+1))   \end{equation*}
to $e^{2\lambda_\epsilon}Z_\eps$, independently at each site $X=\e x\in \eps\mathbb Z$ (recall that $\hat{\eta}=2\eta-1$). We write this as follows,
\begin{eqnarray*}
dZ_\eps(X) & = &\left\{ \eps^{-2}\nu_\eps  + ( e^{-2\lambda_\eps }-1) r_-(X)
+ ( e^{2\lambda_\eps }-1)r_+(X) \right\}  Z_\eps(X) dT\nonumber \\
&&  + ( e^{-2\lambda_\eps }-1) Z_\eps(X) dM_-(X)
+ ( e^{2\lambda_\eps }-1) Z_\eps(X) dM_+(X)
\end{eqnarray*}
where $dM_\pm(X) = dP_\pm(X) -r_\pm (X) dT$ where
$P_-(X), P_+(X)$, $X\in \eps\mathbb{Z}$ are  independent Poisson processes running at 
rates $r_-(X), r_+(X)$, and $d$ always refers to change in macroscopic time $T$.  Let \begin{equation*}\mathcal{D}_\eps = 2\sqrt{pq} = 1-\frac12 \eps + O (\eps^2)\end{equation*} and $\Delta_\eps$ be the $\eps\mathbb Z$ Laplacian, $\Delta f(x) = \eps^{-2}(f(x+\eps) -2f(x) + f(x-\eps))$.  We also have
\begin{equation*}
 {\scriptstyle{\frac12}} \mathcal{D}_\eps \Delta_\eps Z_\eps (X)   =   {\scriptstyle{\frac12}} \eps^{-2} \mathcal{D}_\eps
 ( e^{-\lambda_\eps \hat\eta(x+1)} -2 + e^{\lambda_\eps \hat\eta(x)} ) Z_\eps(X).
 \end{equation*}
 The parameters have been carefully chosen so that 
  \begin{equation*}
 {\scriptstyle{\frac12}} \eps^{-2} \mathcal{D}_\eps
 ( e^{-\lambda_\eps \hat\eta(x+1)} -2 + e^{\lambda_\eps \hat\eta(x)} )= \eps^{-2}\nu_\eps  + ( e^{-2\lambda_\eps }-1) r_-(X)
+ ( e^{2\lambda_\eps }-1)r_+(X).
\end{equation*}
Hence \cite{G},\cite{BG}, 
\begin{equation}\label{sde7}
dZ_\eps = {\scriptstyle{\frac12}} \mathcal{D}_\eps \Delta_\eps Z_\eps +  Z_\eps dM_\eps \end{equation}
where  \begin{equation*}
dM_\eps(X)= ( e^{-2\lambda_\eps }-1)  dM_-(X)
+ ( e^{2\lambda_\eps }-1)  dM_+(X)
\end{equation*} are martingales in $T$ with
\begin{equation*}
d\langle M_\eps(X),M_\eps(Y)\rangle= \eps^{-1}\mathbf{1}(X=Y) b_\eps(\tau_{-[\eps^{-1}X]}\eta) dT.
\end{equation*}
Here $\tau_x\eta(y) = \eta(y-x)$ and 
\begin{equation*}
b_\eps(\eta)
=1- \hat\eta(1) \hat\eta(0)+ \hat{b}_\eps(\eta)
\end{equation*}
where \begin{eqnarray}\nonumber \hat{b}_\eps(\eta) & = & \eps^{-1}\{ [ p( (e^{-2\lambda_\eps}-1)^2-4\eps ) + q( (e^{2\lambda_\eps}-1)^2-4\eps )] 
\\ && + [q(e^{-2\lambda_\eps}-1)^2-p(e^{2\lambda_\eps}-1)^2 ](\hat\eta(1)- \hat\eta(0)) \\&& \nonumber - [q(e^{-2\lambda_\eps}-1)^2+p(e^{2\lambda_\eps}-1)^2 -\eps]\hat\eta(1) \hat\eta(0)\}
.\end{eqnarray}
Clearly $b_\eps,\hat{b}_\eps\ge 0$. It is easy to check that there is a $C<\infty$ such that \begin{equation*}\hat{b}_\eps\le C\eps^{1/2}\end{equation*}
and, for sufficiently small $\eps>0$,
\begin{equation}\label{bdona} 
b_\eps\le 3.
\end{equation}
Note that  (\ref{sde7}) is equivalent to the 
integral equation,
\begin{eqnarray}\label{sde3}
{Z}_\eps(T,X) & = &  \eps\sum_{Y\in \eps \mathbb{Z}} p_\eps(T,X-Y) Z_\eps(0,Y) 
\\ && \nonumber  + \int_0^T \eps\sum_{Y\in \eps \mathbb{Z}} p_\eps(T-S,X-Y)  Z_\eps(S,Y)d{M}_\eps(S,Y) 
\end{eqnarray}
where $p_\eps(T,X) $ are the  (normalized) transition probabilities for the continuous time random 
walk with generator $ {\scriptstyle{\frac12}} \mathcal{D}_\eps \Delta_\eps $.  The normalization is
multiplication of the actual transition probabilities by $\eps^{-1}$ so that 
\begin{equation*}
 p_\eps(T,X) \to p(T,X) = \frac{ e^{-X^2/ 2T} }{\sqrt{2\pi T}}.
\end{equation*}
We need some apriori bounds.
\begin{lemma}\label{apriori} For $0< T\le T_0$, and for each $q=1,2,\ldots$, there is a $C_q= C_q(T_0)<\infty$ such that 
\begin{enumerate}[i.]
\item $E [ Z_\eps^2(T,X) ] \le C_2p_\eps^2(T,X)$;
\item $ E\left[ \left( Z_\eps(T,X)-\eps\sum_{Y\in \eps \mathbb{Z}} p_\eps(T,X-Y) Z_\eps(0,Y)\right)^2 \right] \le C_2 p_\eps^2(T,X)$;
\item $E [ Z^{2q}_\eps(T,X) ] \le C_q p_\eps^{2q}(T,X)$.
\end{enumerate}
\end{lemma}

\begin{proof} Within the proof, $C$ will denote a finite number which does not depend on any
other parameters except $T$ and $q$, but may change from line 
to line.  Also, for ease of notation, we identify functions on $\eps \mathbb{Z}$ with those on $\mathbb{R}$ by $f(x)=f([x])$.

 First, note that 
\begin{equation*}
Z_\eps(0,Y) = {\scriptstyle\frac12}\epsilon^{-1/2} \exp\{ -\eps^{-1}\lambda_\eps |Y| \}= {\scriptstyle\frac12} \epsilon^{-1/2} \exp\{ -\eps^{-1/2} |Y| + O(\eps^{1/2}) \}
\end{equation*}
is an approximate delta function, from which we check that 
\begin{equation}\label{yy}
\eps\sum_{Y\in \eps \mathbb{Z}} p_\eps(T,X-Y) Z_\eps(0,Y) \le C  p_\eps(T,X).
\end{equation}
Let \begin{equation*}
f_\eps(T,X)= E[ Z_\eps^2(T,X)] .
\end{equation*}
From (\ref{yy}), (\ref{sde3}) we get 
\begin{equation}\label{tt}
f_\eps(T,X)\le C p^2_\eps(T,X)
+ C\int_0^T \int_{-\infty}^{\infty} p^2_\eps(T-S,X-Y)  f_\eps(S,Y) dSdY. 
\end{equation}
Iterating we obtain, \begin{equation}\label{tt2}
f_\eps(T,X) \le   \sum_{n=0}^\infty C^n I_{n,\eps}(T,X)
\end{equation}
where, for $\Delta_n=\Delta_n(T)=\{0=t_0\le T_1<\cdots< T_n<T\}$,$X_0=0$,
\begin{equation*}
 I_{n,\eps}(T,X)= \int_{\Delta_n} \int_{\mathbb{R}^n} 
\prod_{i=1}^{n} p^2_\eps(T_i-T_{i-1},X_i-X_{i-1})  p_\eps^2(T-T_n, X-x_n)\prod_{i=1}^{n} dX_i dT_i. 
\end{equation*}
One readily checks that 
\begin{equation*}
 I_{n,\eps}(T,X)\le C^n T^{n/2} (n!)^{-1/2}p^2_\eps(T,X). 
\end{equation*}
From which we obtain $i$,
\begin{equation*}
f_\eps(T,X) \le   C\sum_{n=0}^\infty (CT)^{n/2} (n!)^{-1/2} p^2_\eps(T,X)\le C' p^2_\eps(T,X).
\end{equation*}
Now we turn to $ii$.  From  (\ref{sde3}), the term on right hand side is bounded by a constant multiple of 
\begin{equation*}
   \int_0^T \int_{-\infty}^{\infty} p^2_\eps(T-S,X-Y)  E[Z^2_\eps(S,Y)]dYdS.
\end{equation*}
Using   $i$, this is in turn bounded by $C\sqrt{T} p^2_\eps(T,X) $, 
which proves $ii$.

Finally we prove $iii$.  Fix a $q\ge 2$. By standard methods of martingale analysis and (\ref{bdona}), we have
\begin{eqnarray*}
&& E\Big[\Big(\int_0^T\eps\sum_{Y\in \eps \mathbb{Z}} p_\eps(T-S,X-Y)  Z_\eps(S,Y)d{M}_\eps(S,Y) \Big)^{2q}\Big]
\\&& \qquad  \le C   E\Big[\Big(\int_0^T \eps\sum_{Y\in \eps \mathbb{Z}} p^2_\eps(T-S,X-Y)  Z^2_\eps(S,Y)dS \Big)^{q}\Big].\nonumber
\end{eqnarray*}
Let 
\begin{equation*}
g_\eps(T,X) =E[Z^{2q}_\eps(T,X)]/ p_\eps^{2q}(T,X).
\end{equation*}
From the last inequality, and Schwarz's inequality, we have 
\begin{equation*}
 g_\eps(T,X) \le C (1+  \int_{\Delta'_q(T)}\int_{\mathbb{R}^q} \prod_{i=1}^q p^2_\eps(S_i-S_{i-1},X_i-X_{i-1})p_\eps^{2}(S_i,Y_i)g^{1/q}_\eps(S_i,Y_i)dY_idS_i ).\end{equation*}
Now use the fact that
\begin{equation*}
\prod_{i=1}^q g^{1/q}_\eps(S_i,Y_i) \le C\sum_{i=1}^q  \frac{\prod_{j\neq i} p^{2/(q-1)}_\eps(S_j,Y_j)}{p^{2}_\eps(S_i,Y_i)}g_\eps(S_i,Y_i) \end{equation*}
and iterate the inequality to obtain $iii$.
\end{proof}

 We now turn to the tightness.  In fact, although we are in a different regime, the arguments of
  \cite{BG} actually
extend to our case.  For each $\delta>0$, let $\mathscr{P}^\delta_\eps$ be the distributions of the processes $\{Z_\eps(T,X)\}_{\delta\le T}$ on  $D_u([\delta,\infty); D_u(\mathbb R))$ where $D_u$ refers to right continuous paths with left limits  with  the topology of uniform convergence on compact sets. Because the discontinuities of $Z_\eps(T,\cdot)$ are restricted to
 $\eps(1/2+\Z)$, it is measurable as a $D_u(\R)$-valued random function (see Sec. 18 of 
\cite{Bill}.)   Since the jumps of $Z_\eps(T,\cdot)$ are uniformly small, local uniform convergence works for us just as well the standard Skhorohod topology.    The following summarizes results which are contained \cite{BG} but not explicitly stated there in the form we need.      
\begin{theorem} \cite{BG}  There is an explicit $p<\infty$ such that if there exist $C, c < \infty$ for which 
\begin{equation}\label{deltainit}
\int_{-\infty}^{\infty} Z_\eps^p(\delta,X)d\mathscr{P}^\delta_\eps \le C e^{ c|X| } ,\qquad X \in \eps \mathbb{Z},
\end{equation}
 then $\{\mathscr{P}^\delta_\eps\}_{0\le \eps\le 1/4}$ is tight.  
 Any limit point $\mathscr{P}^\delta$ is supported on $C([\delta,\infty); C(\mathbb R))$ and solves  the martingale problem for the stochastic heat equation (\ref{she}) after time~$\delta$.
\end{theorem}

It appears that $p=10$ works in \cite{BG}, though it almost certainly can be improved to $p=4$.
Note that the process level convergence is more than we need for the one-point function.  However, it could be useful in the future. Although not explicitly stated there the theorem is proved in \cite{BG}. The key point is that all computations in \cite{BG} after the initial time are done using the equation (\ref{sde7}) for $Z_\eps$,
which scales linearly in $Z_\eps$.  So the only input is a bound like (\ref{deltainit}) on the initial data.  In \cite{BG}, this
is made as an assumption, which can easily be checked for initial data close to equilibrium.  In the present case, it follows from 
$iii$ of Lemma \ref{apriori}. 

   The measures $\mathscr{P}^{\delta_1}$ and $\mathscr{P}^{\delta_2}$ for $\delta_1<\delta_2$ can be chosen to be consistent  on $C( [\delta_2,\infty), C(\mathbb{R}) ) $ and because of this there is a limit measure $\mathscr{P}$ on $C((0,\infty), C(\mathbb{R}) ) $ which is consistent with any $\mathscr{P}^{\delta}$   when restricted to $C( [\delta,\infty), C(\mathbb{R}) )$.  From the uniqueness of the martingale problem for $t\ge \delta>0$ and the corresponding martingale representation theorem \cite{KS}  there is a  space-time
white noise $\dot{\mathscr{W}}$, on a possibly  enlarged probability space, $(\bar\Omega,\bar{\mathscr{F}}_T, \bar{\mathscr{P}})$ such that under $\bar{\mathscr{P}}$, for any $\delta>0$,
\begin{eqnarray*}
{Z}(T,X)  & = & \int_{-\infty}^{\infty} p(T-\delta,X-Y) Z(\delta,Y)  dY\nonumber \\ &&
 + \int_\delta^{T} \int_{-\infty}^{\infty} p(T-S,X-Y)  Z(S,Y)\bar{\mathcal{W}} (dY,dS). 
\end{eqnarray*}
Finally $ii$ of Lemma \ref{apriori} shows that under $\bar{\mathscr{P}}$, 
\begin{equation*}
\int_{-\infty}^{\infty} p(T-\delta,X-Y) Z(\delta,Y) dY \to p(T,X)
\end{equation*}
as $\delta\searrow 0$, which completes the proof.

\section{Alternative forms of the crossover distribution function}\label{kernelmanipulations}
We now demonstrate how the various alternative formulas for $F_{T}(s)$ given in Theorem \ref{main_result_thm} are derived from the cosecant kernel formula of Theorem \ref{epsilon_to_zero_theorem}.

\subsection{Proof of the crossover Airy kernel formula}\label{cross_over_airy_sec}
We prove this by showing that 
\begin{equation*}\det(I-K_a^{\csc})_{L^2(\tilde\Gamma_{\eta})} =\det(I - K_{\sigma_{T,\tilde\mu}})_{{L}^2(\kappa_T^{-1}a,\infty)}
\end{equation*}
where $K_{\sigma_{T,\tilde\mu}}$ and $\sigma_{T,\tilde\mu}$ are given in the statement of Theorem \ref{main_result_thm} and $\kappa_T = 2^{-1/3} T^{1/3}$.

The kernel $K^{\csc}_{a}(\tilde\eta,\tilde\eta')$ is given by equation (\ref{k_csc_definition}) as
\begin{equation*}
\int_{\tilde\Gamma_{\zeta}} e^{-\tfrac{T}{3}(\tilde\zeta^3-\tilde\eta'^3)+2^{1/3}a(\tilde\zeta-\tilde\eta')}\bigg(2^{1/3}\int_{-\infty}^{\infty} \frac{\tilde\mu e^{-2^{1/3}t(\tilde \zeta - \tilde \eta')}}{ e^t - \tilde \mu}dt  \bigg)\frac{d\tilde\zeta}{\tilde\zeta-\tilde\eta},
\end{equation*}
where we recall that the inner integral converges since $\re(-2^{1/3}(\tilde\zeta-\tilde\eta'))=1/2$ (see the discussion in Definition \ref{thm_definitions}).
For $\re(z)<0$ we have the following nice identity: 
\begin{equation*}
 \int_a^{\infty} e^{xz} dx= -\frac{e^{az}}{z},
\end{equation*}
which, noting that $\re(\tilde\zeta-\tilde\eta')<0$, we may apply to the above kernel to get
\begin{equation*}
-2^{2/3}\int_{\tilde\Gamma_{\zeta}}\int_{-\infty}^{\infty}\int_{a}^{\infty} e^{-\frac{T}{3}(\tilde\zeta^3-\tilde\eta'^3)-2^{1/3}a\tilde\eta'} \frac{\tilde\mu e^{-2^{1/3}t(\tilde \zeta - \tilde \eta')}}{ e^t - \tilde \mu} e^{2^{1/3}(a-x)\tilde\eta} e^{2^{1/3}x\tilde \zeta } dx dt d\tilde \zeta.
\end{equation*}
This kernel can be factored as a product $ABC$ where
\begin{eqnarray*}
 A:L^2(a,\infty)\rightarrow L^2(\tilde\Gamma_{\eta}),\qquad B:L^2(\tilde\Gamma_{\zeta})\rightarrow L^2(a,\infty),\qquad C:L^2(\tilde\Gamma_{\eta})\rightarrow L^2(\tilde\Gamma_{\zeta}),
\end{eqnarray*}
and the operators are given by their kernels
\begin{eqnarray*}
&&  A(\tilde\eta,x) = e^{2^{1/3}(a-x)\tilde\eta}, \qquad B(x,\tilde\zeta) = e^{2^{1/3}x \tilde\zeta},  \\  
\nonumber&&  C(\tilde\zeta,\tilde\eta) = -2^{2/3}\int_{-\infty}^{\infty} \exp\left\{-\frac{T}{3}(\tilde\zeta^3-\tilde\eta^3)-2^{1/3}a\tilde\eta\right\}\frac{\tilde\mu e^{-2^{1/3}t(\tilde \zeta - \tilde \eta)}}{ e^t - \tilde \mu}dt.
\end{eqnarray*}
Since $\det(I-ABC) = \det(I-BCA)$ we consider $BCA$ acting on $L^2(a,\infty)$ with kernel 
\begin{equation*}
-2^{2/3}\int_{-\infty}^{\infty}\int_{\Gamma_{\tilde\zeta}}\int_{\Gamma_{\tilde\eta}}
e^{-\frac{T}{3}(\tilde\zeta^3-\tilde\eta^3)+2^{1/3}(x-t)\tilde\zeta -2^{1/3}(y-t)\tilde\eta}\frac{\tilde\mu}{e^t - \tilde\mu} d\tilde\eta d\tilde\zeta dt.
\end{equation*}
Using the formula for the Airy function given by 
\begin{equation*}
\Ai(r) = \int_{\tilde\Gamma_{\zeta}} \exp\{-\frac{1}{3}z^3 +rz\} dz
\end{equation*}
and replacing $t$ with $-t$ we find that our kernel equals
\begin{equation*}
2^{2/3}T^{-2/3}\int_{-\infty}^{\infty}\frac{\tilde\mu}{\tilde\mu - e^{-t}} \Ai\big(T^{-1/3}2^{1/3}(x+t)\big)\Ai\big(T^{-1/3}2^{1/3}(y+t)\big)dt.
\end{equation*}
We may now change variables in $t$ as well as in $x$ and $y$ to absorb the factor of $T^{-1/3}2^{1/3}$. 
To rescale $x$ and $y$ use $\det(I - K(x,y))_{L^2( ra,\infty)} = \det (I- rK(rx,ry))_{L^2(a,\infty)}$. This completes the proof.

\subsection{Proof of the Gumbel convolution formula}\label{gumbel_convolution_sec}
%
Before starting we remark that throughout this proof we will dispense with the tilde with respect to $\tilde\mu$ and $\mathcal{\tilde C}$. We choose to prove this formula directly from the form of the Fredholm determinant given in the crossover Airy kernel formula of Theorem \ref{main_result_thm}. However, we make note that it is possible, and in some ways simpler (though a little messier) to prove this directly from the $\csc$ form of the kernel. Our starting point is the formula for $F_T(s)$ given in equation (\ref{sigma_Airy_kernel_formula}). The integration in $\mu$ occurs along a complex contour and even though we haven't been writting it explicitly, the integral is divided by $2\pi i$. We now demonstrate how to squish this contour to the the positive real line (at which point we will start to write the $2\pi i$). The pole in the term $\sigma_{T,\mu}(t)$ for $\mu$ along $\R^+$ means that the integral along the positive real axis from above will not exactly cancel the integral from below.

Define a family of contour $\mathcal{ C}_{\delta_1,\delta_2}$ parametrized by $\delta_1,\delta_2>0$ (small). The contours are defined in terms of three sections \begin{equation*}
\mathcal{C}_{\delta_1,\delta_2} = \mathcal{ C}_{\delta_1,\delta_2}^{-}\cup \mathcal{C}_{\delta_1,\delta_2}^{circ}\cup \mathcal{C}_{\delta_1,\delta_2}^{+}
\end{equation*}
traversed counterclockwise, where 
\begin{equation*}
 \mathcal{C}_{\delta_1,\delta_2}^{circ} = \{\delta_2 e^{i\theta}:\delta_1\leq \theta\leq 2\pi -\delta_1\}
\end{equation*}
and where $\mathcal{C}_{\delta_1,\delta_2}^{\pm}$ are horizontal lines extending from $\delta_1 e^{\pm i\delta_2}$ to $+\infty$.

We can deform the original $\mu$ contour $\mathcal{\mu}$ to any of these contours without changing the value of the integral (and hence of $F_T(s)$). To justify this we use Cauchy's theorem. However this requires the knowledge that the determinant is an analytic function of $\mu$ away from $\R^+$. This may be proved similarly to the proof of Lemma \ref{deform_mu_to_C} and relies on Lemma \ref{Analytic_fredholm_det_lemma}. As such we do not include this computation here. 

Fixing $\delta_2$ for the moment we wish to consider the limit of the integrals over these contours as $\delta_1$ goes to zero. The resulting integral be we written as $I_{\delta_2}^{circ} + I_{\delta_2}^{line}$ where
\begin{eqnarray*}
I_{\delta_2}^{circ} &=& \oint_{|\mu|=\delta_2} \frac{d\mu}{\mu} e^{-\mu} \det(I-K_{T,\mu})_{L^2(\kappa_T^{-1}a,\infty)},\\
I_{\delta_2}^{line} &=& -\lim_{\delta_{1}\rightarrow 0} \int_{\delta_2}^{\infty} \frac{d\mu}{\mu} e^{-\mu} [\det(I-K_{T,\mu+i\delta_i})-\det(I-K_{T,\mu-i\delta_i})] 
\end{eqnarray*}
\begin{claim}
 $I_{\delta_2}^{circ}$ exists and $\lim_{\delta_{2}\rightarrow 0}I_{\delta_2}^{circ} = 1$.
\end{claim}

\begin{proof}
It is easiest, in fact, to prove this claim by replacing the determinant by the $\csc$ determinant: equation (\ref{k_csc_definition}). From that perspective the $\mu$ at 0 and at $2\pi$ are on opposite sides of the branch cut for $\log(-\mu)$, but are still defined (hence the $I_{\delta_2}^{circ}$ is clearly defined). As far as computing the limit, one can do the usual Hilbert-Schmidt estimate and show that, uniformly over the circle $|\mu|=\delta_2$, the trace norm goes to zero as $\delta_2$ goes to zero. Thus the determinant goes uniformly to 1 and the claim follows. 
\end{proof}

Turning now to $I_{\delta_2}^{line}$, that this limit exists can be seen by going to the equivalent $\csc$ kernel (where this limit is trivially just the kernel on different levels of the $\log(-\mu)$ branch cut). Notice now that we can write the operator $K_{T,\mu+i\delta_1}=K_{\delta_1}^{\rm sym}+K_{\delta_1}^{\rm asym}$ and likewise $K_{T,\mu-i\delta_1}=K_{\delta_1}^{\rm sym}-K_{\delta_1}^{\rm asym}$ where $K_{\delta_1}^{\rm sym}$ and $K_{\delta_1}^{\rm asym}$ also act on $L^2(\kappa_T^{-1}a,\infty)$ and are given by their kernels
\begin{eqnarray*}
K_{\delta_1}^{{\rm sym}}(x,y) &=& \int_{-\infty}^{\infty}\frac{\mu(\mu-b)+\delta_1^2}{(\mu-b)^2+\delta_1^2} \Ai(x+t)\Ai(y+t)dt\\
K_{\delta_1}^{{\rm asym}}(x,y) &=& \int_{-\infty}^{\infty}\frac{-i\delta_1 b}{(\mu-b)^2+\delta_1^2} \Ai(x+t)\Ai(y+t)dt,
\end{eqnarray*}
where $b=b(t)=e^{-\kappa_T t}$. 

From this it follows that 
\begin{equation*}
 K^{\rm sym}(x,y):= \lim_{\delta_1\rightarrow 0} K_{\delta_1}^{\rm sym}(x,y)  =\mathrm{P.V.}\int \frac{\mu}{\mu-e^{-\kappa_T t}} \Ai(x+t)\Ai(y+t)dt.
\end{equation*}
As far as $K_{\delta_1}^{{\rm asym}}$, since $\mu-b$ has a unique root at $t_0=-\kappa_T^{-1}\log\mu$, it follows from the Plemelj formula \cite{D} that 
\begin{equation*}
\lim_{\delta_1\rightarrow 0}K_{\delta_1}^{\rm asym}(x,y) = -\frac{\pi i} {\kappa_T} \Ai(x+t_0)\Ai(y+t_0).
\end{equation*}
With this in mind we define
\begin{equation*}
 K^{\rm asym}(x,y) = \frac{2\pi i} {\kappa_T} \Ai(x+t_0)\Ai(y+t_0).
\end{equation*}
We see that $K^{\rm asym}$ is a multiple of the projection operator onto the shifted Airy functions.

We may now collect the calculations from above and we find that
\begin{eqnarray*}
\nonumber I_{\delta_2}^{line} &=& -\frac{1}{2\pi i}\int_{\delta_2}^{\infty} \frac{d\mu}{\mu}e^{-\mu} [\det(I-K^{\rm sym}+\tfrac{1}{2}K^{\rm asym})-\det(I-K^{\rm sym}-\tfrac{1}{2}K^{\rm asym})]\\
&=& -\frac{1}{2\pi i}\int_{\delta_2}^{\infty} \frac{d\mu}{\mu}e^{-\mu} \det(I-K^{\rm sym})\mathrm{tr}\left((I-K^{\rm sym})^{-1}K^{\rm asym}\right)
\end{eqnarray*}
where both $K^{\rm sym}$ and $K^{\rm asym}$ act on $L^2(\kappa_T^{-1}a,\infty)$ and where we have used the fact that $K^{\rm asym}$ is rank one, and if you have $A$ and $B$, where $B$ is rank one, then
\begin{equation*}
\det(I-A+B) = \det(I-A)\det(I+ (I-A)^{-1}B) = \det(I-A)\left[1+\mathrm{tr}\left((I-A)^{-1}B\right)\right].
\end{equation*}
As stated above we've only shown the pointwise convergence of the kernels to $K^{\rm sym}$ and $K^{\rm asym}$. However, using the decay properties of the Airy function and the exponential decay of $\sigma$ this can be strengthened to trace-class convergence. 

We may now take $\delta_2$ to zero and find that
\begin{eqnarray*}\nonumber
F_T(s) & = &  \lim_{\delta_2\rightarrow 0 } (I_{\delta_2}^{circ} + I_{\delta_2}^{line}) \\ &= &1-\frac{1}{2\pi i}\int_0^{\infty} \frac{d\mu}{\mu}e^{-\mu} \det(I-K^{\rm sym})\mathrm{tr}\left((I-K^{\rm sym})^{-1}K^{\rm asym}\right)
\end{eqnarray*}
with $K^{\rm sym}$ and $K^{\rm asym}$ as above acting on $L^2(\kappa_T^{-1}a,\infty)$ and where the integral is improper at zero.

We can simplify our operators so that by changing variables and replacing $x$ by $x+t_0$ and $y$ by $y+t_0$. We can also change variables from $\mu$ to $e^{-r}$. With this in mind we redefine the operators $K^{\rm sym}$ and $K^{\rm asym}$ to act on $L^2(\kappa_T^{-1}(a-r),\infty)$ with kernels
\begin{eqnarray*}
K^{\rm sym}(x,y) &=& \mathrm{P.V.} \int \sigma(t)\Ai(x+t)\Ai(y+t)dt\\
\nonumber K^{\rm asym}(x,y) &=&  \Ai(x)\Ai(y),
\end{eqnarray*}
where $\sigma(t) = \frac{1}{1-e^{-\kappa_T t}}$. In terms of these operators we have
\begin{equation*}
 F_T(s) = 1-\int_{-\infty}^{\infty} e^{-e^{-r}} f(a-r) dr
\end{equation*}
where 
\begin{equation*}
f(r) =\kappa_T^{-1} \det(I-K^{\rm sym})_{L^2(\kappa_T^{-1}r,\infty)}\mathrm{tr}\left((I-K^{\rm sym})^{-1}K^{\rm asym}\right)_{L^2(\kappa_T^{-1}r,\infty)}.
\end{equation*}
 Calling $G(r)=e^{-e^{-r}}$ and observing that $K^{\rm sym}=K_{\sigma_T}$ and $K^{\rm asym}=\rm{P}_{\Ai}$ this completes the proof of the first part of the Gumbel convolution formula.
Turning now to the Hilbert transform formula, we may isolate the singularity of $\sigma_T(t)$ from the above kernel $K^{\rm sym}$ (or $K_{\sigma_T}$) as follows.
%
%
Observe that we may write $\sigma_T(t)$as
\begin{equation*}
\sigma_T(t) = \tilde\sigma_T(t) + \frac{1}{\kappa_T t}
\end{equation*}
where $\tilde\sigma_T(t)$ (given in equation (\ref{tilde_sigma_form})) is a smooth function, non-decreasing on the real line, with $\tilde\sigma_T(-\infty)=0$ and $\tilde\sigma_T(+\infty)=1$. Moreover, $\tilde\sigma_T^{\prime}$ is an approximate delta function with width $\kappa_T^{-1}= 2^{1/3}T^{-1/3}$. The principle value integral of the $\tilde\sigma_T(t)$ term can be replaced by a simple integral. The new term gives
\begin{equation*}
\mathrm{P.V.}\int \frac{1}{\kappa_T t}\Ai(x+t)\Ai(y+t). 
\end{equation*}
This is $\kappa_T^{-1}$ times the Hilbert transform of the product of Airy functions, which is explicitly computable \cite{Varlamov} with the result begin 
\begin{equation*}
\mathrm{P.V.}\int \frac{1}{\kappa_T t}\Ai(x+t)\Ai(y+t) = \kappa_T^{-1} \pi G_{\frac{x-y}{2}}(\frac{x+y}{2})
\end{equation*}
where $G_a(x)$ is given in equation (\ref{tilde_sigma_form}).

\section{Formulas for a class of generalized integrable integral operators}\label{int_int_op_sec}

Presently we will consider a certain class of Fredholm determinants and make two computations involving these determinants. The second of these computations closely follows the work of Tracy and Widom and is based on a similar calculation done in \cite{TWAiry}. In that case the operator in question is the Airy operator. We deal with the family of operators which arise in considering $F_{T}(s)$.

Consider the class of Fredholm determinants $\det(I-K)_{L^2(s,\infty)}$ with operator $K$ acting on $L^2(s,\infty)$ with kernel 
\begin{equation}\label{gen_int_int_ker}
K(x,y) =\int_{-\infty}^{\infty} \sigma(t) \Ai(x+t)\Ai(y+t)dt,
\end{equation}
where $\sigma(t)$ is  a  function  which is
smooth except at a finite number of points at which it has bounded jumps and which approaches   $0$ at $-\infty$ and $1$ at $\infty$, exponentially fast. These operators are, in a certain sense, generalizations of the class of integrable integral operators (see \cite{BorodinDeift}). 

The kernel can be expressed alternatively as 
\begin{equation}\label{kernel}
K(x,y) = \int_{-\infty}^{\infty}\sigma'(t) \frac{\varphi(x+t) \psi(y+t) - \psi(x+t) \varphi(y+t)}{
x-y} dt,
\end{equation}
 with $\varphi(x)=\textrm{Ai}(x)$ and $\psi(x)= \textrm{Ai}^{\prime}(x)$ and $\Ai(x)$ the Airy function.

This, and the entire generalization we will now develop is analogous to what is known for the Airy operator which is defined by its kernel $K_{\Ai}(x,y)$ on $L^2(-\infty,\infty)$
\begin{equation*}
K_{\Ai}(x,y) = \int_{-\infty}^{\infty} \chi(t) \Ai(x+t)\Ai(y+t)dt = \frac{\Ai(x)\Ai^\prime(x)-\Ai(y)\Ai^\prime(x)}{x-y},
\end{equation*}
where presently $\chi(t)=\mathbf{1}_{\{t\geq 0\}}$.

Note that the $\sigma(t)$ in our main result is not exactly of this type.  However, one can smooth out the $\sigma$, and apply the results
of this section to obtain formulas, which then can be shown to converge to the desired formulas as the smoothing is removed.  
It is straightforward to control the convergence in terms of trace norms, so we will not provide further details here.

\subsection{Symmetrized determinant expression}\label{symmetrized_sec}

It is well known that 
\begin{equation*}
\det(I-K_{\Ai})_{L^2(s,\infty)} = \det(I-\sqrt{\chi_{s}}K_{\Ai}\sqrt{\chi_{s}})_{L^2(-\infty,\infty)}
\end{equation*}
where $\chi_{s}$ is the multiplication operator by $\mathbf{1}_{\{\bullet\geq s\}}$ (i.e., $(\chi_s f)(x) = \mathbf{1}(x\geq s)f(x)$).

The following proposition shows that for our class of determinants the same relation holds, and
provides the proof of formula (\ref{sym_F_eqn}) of Theorem 1.

\begin{proposition}
For $K$ in the class of operators with kernel as in (\ref{gen_int_int_ker}), 
\begin{equation*}
\det(I-K)_{L^2(s,\infty)} = \det(I-\hat{K}_s)_{L^2(-\infty,\infty)},
\end{equation*}
where the kernel for $\hat{K}_s$ is given by 
\begin{equation*}
\hat{K}_s(x,y) = \sqrt{\sigma(x-s)}K(x,y) \sqrt{\sigma(y-s)}.
\end{equation*}
\end{proposition}
 
\begin{proof}
Define $L_s: L^2(s,\infty)\to L^2(-\infty,\infty)$ by
\begin{equation*}
(L_sf)(x) = \int_s^\infty  {\rm Ai}(x+y)f(y)dy.
\end{equation*}
Also define 
$\sigma:  L^2(-\infty,\infty)\to L^2(-\infty,\infty)$ by
$(\sigma f)(x) = \sigma(x) f(x)$,
and similarly 
$\chi_s:  L^2(-\infty,\infty)\to L^2(s,\infty)$ by
$(\chi_s f)(x) = \mathbf{1}(x\ge s) f(x)
$.
Then
\begin{equation*}
K = \chi_s L_{-\infty}\sigma L_s. 
\end{equation*}
We have 
\begin{equation*}
\det(I - K)_{L^2(s,\infty)} = \det(I - \tilde K_s)_{L^2(-\infty,\infty)}
\end{equation*}
where
\begin{equation*}
\tilde K_s = \sqrt{\sigma} L_s \chi_s L_{-\infty}\sqrt{\sigma}.
\end{equation*}
The key point is that
\begin{equation*}
L_s \chi_sL_{-\infty} (x,y) = K_{\Ai}(x+s,y+s) 
\end{equation*}
where $K_{\Ai}$ is the Airy kernel. One can also see now that this operator is self-adjoint on the real line.
\end{proof}

\subsection{Painlev\'{e} II type integro-differential equation}\label{integro_differential}
We now develop an integro-differential equation expression for $\det(I-K)_{L^2(s,\infty)}$. 
This provides the proof of Proposition \ref{prop2}.  

Recall that $F_{\mathrm{GUE}}(s)=\det(I-K_{\Ai})_{L^2(s,\infty)}$ can be expressed in terms of a non-linear version of the Airy function, known as Painlev\'{e} II, as follows. Let $q$ be the unique (Hastings-McLeod) solution to Painlev\'{e} II:
\begin{equation*}
\frac{d^2}{ds^2}q(s) = (s+2q^2(s))q(s)
\end{equation*}
subject to $q(s)\sim \Ai(s)$ as $s\rightarrow \infty$. Then
\begin{equation*}
 \frac{d^2}{ds^2} \log \det(I-K_{\Ai})_{L^2(s,\infty)} = q^2(s).
\end{equation*}
From this one shows that 
\begin{equation*}
 F_{\mathrm{GUE}}(s) = \exp \left( -\int_{s}^{\infty} (x-s) q^2(x)dx\right).
\end{equation*}
See  \cite{TWAiry} for details. We now show that an analogous expression exists for the class of operators described in (\ref{gen_int_int_ker}).

\begin{proposition}\label{pII_prop}
For $K$ in the class of operators with kernel as in (\ref{gen_int_int_ker}), let $q(t,s)$  be the solution to
\begin{equation}\label{5.15}
\frac{d^2}{ds^2} q_t(s) = \left(s+t + 2\int_{-\infty}^{\infty} \sigma^\prime(r)q_r^2(s)dr\right) q_t(s)
\end{equation}
subject to $q_t(s)\sim \Ai(t+s)$ as $s\rightarrow \infty$. Then we have
\begin{eqnarray}\label{sicks}
\frac{d^2}{ds^2} \log\det(I-K)_{L^2(s,\infty)} &=& \int_{-\infty}^{\infty} \sigma^\prime(t)q_t^2(s)dt,\\
\nonumber \det(I-K)_{L^2(s,\infty)} &=& \exp\left(-\int_s^{\infty}(x-s)\int_{-\infty}^{\infty} \sigma^\prime(t)q_t^2(x)dtdx\right)
\end{eqnarray}
\end{proposition}

\begin{proof}
As mentioned, we follow the work of Tracy and Widom \cite{TWAiry} very closely, and make the necessary modifications to our present setting. Consider an operator $K$ of the type described in (\ref{gen_int_int_ker}). We use the notation $K\doteq
K(x,y)$ to indicate that operator $K$ has kernel $K(x,y)$. 
It will  be convenient to think of our operator $K$ as acting, not on
$(s,\infty)$, but
on $(-\infty,\infty)$ and to have kernel
\begin{equation*} 
K(x,y) \chi_s(y) 
\end{equation*}
where $\chi$ is the characteristic function of $(s,\infty)$.  Since the integral
operator $K$ is trace class and depends smoothly on
the parameter $s$, we have  the well known formula
\begin{equation}\label{dLog}
\frac{d}{ds}\log\det\left(I-K\right)=-\textrm{tr}\left(\left(I-K\right)^{-1}
 \frac{\partial K}{\partial s}\right).
\end{equation}
By calculus
\begin{equation}\label{Kderiv}
\frac{\partial K}{\partial s}\doteq -K(x,s)\delta(y-s). 
\end{equation}
Substituting this into the above expression gives
\begin{equation*} 
\frac{d}{ds} \log\det\left(I-K\right)= - R(s,s)
\end{equation*}
where $R(x,y)$ is the resolvent kernel of $K$, i.e.\  $R=(I-K)^{-1}K\doteq
R(x,y)$. The resolvent
kernel $R(x,y)$ is smooth in $x$ but discontinuous in $y$ at $y=s$.  The
quantity
$R(s,s)$ is interpreted to mean the limit of $R(s,y)$ as $y$ goes to $s$ from above:
\begin{equation*} 
\lim_{y\rightarrow s^+}R(s,y).
\end{equation*}

\subsubsection{Representation for $R(x,y)$}
If $M$ denotes the multiplication operator, $(Mf)(x)=x f(x)$,  then
\begin{eqnarray*} \nonumber\left[M,K\right] & \doteq & x K(x,y)- K(x,y) y  = (x-y) K(x,y) \\ &=& \int_{-\infty}^{\infty} \sigma'(t)\{
 \varphi(x+t) \psi(y+t) - \psi(x+t) \varphi(y+t)\} dt \end{eqnarray*}
where $\varphi(x) = {\rm Ai}(x)$ and $\psi(x)= {\rm Ai}'(x)$.   As an operator equation this is
\begin{equation*} 
\left[M,K\right]=\int_{-\infty}^{\infty} \sigma'(t) \{\tau_t\varphi\otimes \tau_t\psi - \tau_t\psi\otimes \tau_t\varphi\} dt, 
\end{equation*}
where   $a\otimes b\doteq a(x) b(y)$ and $\left[\cdot,\cdot\right]$
denotes the commutator. The operator $\tau_{t}$ acts as $(\tau_{t}f)(x) = f(x+t)$.
Thus
\begin{eqnarray}\label{comm1}
\left[M,\left(I-K\right)^{-1}\right]&=&\left(I-K\right)^{-1} \left[M,K\right]
 \left(I-K\right)^{-1}
 \nonumber \\
 	&=&\int\sigma'(t)\{ \left(I-K\right)^{-1}\left(\tau_t\varphi\otimes \tau_t\psi - \tau_t\psi\otimes
\tau_t\varphi\right)
	\left(I-K\right)^{-1}\} dt\nonumber\\
	&=&\int\sigma'(t)\{ Q_t\otimes P_t - P_t\otimes Q_t \} dt,
	\end{eqnarray}
where we have introduced
\begin{equation*}
Q_t(x;s)=Q_t(x)= \left(I-K\right)^{-1} \tau_t\varphi \ \ \ \textrm{and} \ \ \
P_t(x;s)=P_t(x)=
 \left(I-K\right)^{-1}
\tau_t\psi.
\end{equation*}
Note an important point here that as $K$ is self-adjoint we can use the
transformation $\tau_t\varphi\otimes \tau_t\psi(I-K)^{-1}=  \tau_t\varphi\otimes (I-K)^{-1}\tau_t\psi$.  
On the other hand since $(I-K)^{-1}\doteq \rho(x,y)=\delta(x-y)+R(x,y)$,
\begin{equation}\label{comm2}
 \left[M,\left(I-K\right)^{-1}\right]\doteq (x-y)\rho(x,y)=(x-y) R(x,y). 
\end{equation}
Comparing (\ref{comm1}) and (\ref{comm2}) we see that
\begin{equation*}
R(x,y) = \int_{-\infty}^{\infty} \sigma'(t) \{ \frac{Q_t(x) P_t(y) - P_t(x) Q_t(y)}{x- y}\} dt, \ \ x,y\in (s,\infty). 
\end{equation*}
Taking $y\rightarrow x$ gives
\begin{equation*}
R(x,x)= \int_{-\infty}^{\infty} \sigma'(t) \{Q_t^\prime(x) P_t(x) - P_t^\prime(x) Q_t(x)\} dt 
\end{equation*}
 where the ${}^\prime$ denotes differentiation with respect to $x$.
 
Introducing
\begin{equation*}
q_t(s)=Q_t(s;s) \ \ \ \textrm{and} \ \ \ p_t(s) = P_t(s;s), 
\end{equation*}
we have
\begin{equation}\label{RDiag} 
R(s,s) = \int_{-\infty}^{\infty} \sigma'(t) \{ Q_t^\prime(s;s) p_t(s) - P_t^\prime(s;s) q_t(s)\} dt,\ \  s<x,y<\infty.
\end{equation}

\subsubsection{Formulas for $Q_t^\prime(x)$ and $P_t^\prime(x)$}
As we just saw, we need expressions for $Q_t^\prime(x)$ and $P_t^\prime(x)$. If
$D$ denotes the differentiation operator, $d/dx$,  then
\begin{eqnarray}\label{Qderiv1}
Q_t^\prime(x;s)&=& D \left(I-K\right)^{-1} \tau_t\varphi \nonumber=  \left(I-K\right)^{-1} D\tau_t\varphi +
\left[D,\left(I-K\right)^{-1}\right]\tau_t\varphi\nonumber\\
&=& \left(I-K\right)^{-1} \tau_t\psi +
\left[D,\left(I-K\right)^{-1}\right]\tau_t\varphi\nonumber\\
&=& P_t(x) + \left[D,\left(I-K\right)^{-1}\right]\tau_t\varphi. 
\end{eqnarray}

We need the commutator
\begin{equation*}
\left[D,\left(I-K\right)^{-1}\right]=\left(I-K\right)^{-1} \left[D,K\right]
\left(I-K\right)^{-1}. 
\end{equation*}
Integration by parts shows
\begin{equation*} 
\left[D,K\right] \doteq \left( \frac{\partial K}{\partial x} + \frac{\partial K}{\partial y}\right)
+ K(x,s) \delta(y-s). 
\end{equation*}
The $\delta$ function comes from differentiating the
characteristic function $\chi$.
Using the specific form for $\varphi$ and $\psi$  ($\varphi^\prime=\psi$,
$\psi^\prime=x\varphi$),
\begin{equation*} 
\left( \frac{\partial K}{\partial x} + \frac{\partial K}{\partial y}\right) =
\int_{-\infty}^{\infty}\sigma'(t)\tau_{t}\varphi(x) \tau_{t}\varphi(y)dt. \end{equation*}
Thus
\begin{equation}\label{DComm} 
\left[D,\left(I-K\right)^{-1}\right]\doteq - \int_{-\infty}^{\infty}\sigma'(t) Q_t(x) Q_t(y)dt + R(x,s) \rho(s,y).
\end{equation}
(Recall $(I-K)^{-1}\doteq \rho(x,y)$.)  We now use this in (\ref{Qderiv1})
\begin{eqnarray*}
 Q_t^\prime(x;s)&=&P_t(x) -\int_{-\infty}^{\infty} \sigma'(\tilde t) Q_{\tilde t}(x) \left(Q_{\tilde t},\tau_t\varphi\right) d\tilde t+ R(x,s) q_t(s) \\
 \nonumber&=& P_t(x) - \int_{-\infty}^{\infty} \sigma'(\tilde t)Q_{\tilde t}(x) u_{t,\tilde t}(s) + R(x,s) q_t(s)
 \end{eqnarray*}
 where the inner product $\left(Q_{\tilde t},\tau_t\varphi\right)$ is denoted by $u_{t,\tilde t}(s)$ and  $u_{t,\tilde t}(s)=u_{\tilde t,t}(s)$.
 Evaluating  at $x=s$  gives
\begin{equation*}
 Q_t^\prime(s;s) = p_t(s) - \int_{-\infty}^{\infty} \sigma'(\tilde t)q_{\tilde t}(s) u_{t,\tilde t}(s) +R(s,s) q_t(s).  
\end{equation*}
We now apply the same procedure to compute $P^\prime$ encountering the one new feature that since $\psi^\prime(x)=x\varphi(x)$ we need to introduce an additional commutator term.
 \begin{eqnarray*}
\nonumber P_t^\prime(x;s)&=& D \left(I-K\right)^{-1} \tau_t\psi= \left(I-K\right)^{-1} D\tau_t\psi + \left[D,\left(I-K\right)^{-1}\right]\tau_t\psi\\
&=& (M+t) \left(I-K\right)^{-1} \tau_t\varphi +
\left[\left(I-K\right)^{-1},M\right]\tau_t\varphi+
 \left[D,\left(I-K\right)^{-1}\right]\tau_t\psi.
\end{eqnarray*}
Writing it explicitly, we get $(x+t) Q_t(x) +R(x,s) p_t(s) + \Xi$ where
\begin{eqnarray*}
 \nonumber \Xi &=& \int_{-\infty}^{\infty} \sigma'(\tilde t)\left(P_{\tilde t}\otimes Q_{\tilde t}-Q_{\tilde t}\otimes P_{\tilde t}\right)\tau_t\varphi d\tilde t-\int_{-\infty}^{\infty}\sigma(\tilde t) (Q_{\tilde t}\otimes Q_{\tilde t})\tau_t\psi d\tilde t\\
&=&  \int_{-\infty}^{\infty} \sigma'(\tilde t)\left\{ 
 P_{\tilde t}(x)\left(Q_{\tilde t},\tau_t\varphi\right) -  Q_{\tilde t}(x) \left(P_{\tilde t},\tau_t\varphi\right)- Q_{\tilde t}(x) \left(Q_{\tilde t},\tau_t\psi\right)\right\}d\tilde t\\
 \nonumber&=&  \int_{-\infty}^{\infty} \sigma'(\tilde t)\left\{ 
 P_{\tilde t}(x)u_{t,\tilde t}(s)-  Q_{\tilde t}(x) v_{t,\tilde t}(s)- Q_{\tilde t}(x) v_{\tilde t,t}(s)\right\}d\tilde t,
 \end{eqnarray*}
with the notation $v_{t,\tilde t}(s)=\left(P_{\tilde t},\tau_t\varphi\right)=\left(\tau_{\tilde t}\psi,Q_t\right)$.
 Evaluating  at $x=s$ gives
 \begin{eqnarray*}
P^{\prime}(s;s) & = & (s+t) q_t(s) + \int_{-\infty}^{\infty} \sigma'(\tilde t)\left\{ 
 p_{\tilde t}(s)u_{t,\tilde t}(s)-  q_{\tilde t}(s) v_{t,\tilde t}(s)- q_{\tilde t}(s) v_{\tilde t,t}(s)\right\}d\tilde t \nonumber
\\&&  +R(s,s) p_t(s).
\end{eqnarray*}
 Using this and the expression for $Q^\prime(s;s)$ in (\ref{RDiag})
 gives,
 \begin{equation*}
 R(s,s)= \int_{-\infty}^{\infty} \sigma'(t) \{ p_t^2-s q_t^2  -\int_{-\infty}^{\infty} \sigma'(\tilde t) \{  [q_{\tilde t} p_t+p_{\tilde t} q_t] u_{t,\tilde t}  -q_{\tilde t}q_{ t}[ v_{t,\tilde t}+ v_{\tilde t, t}]\} \}d\tilde t dt.
\end{equation*}
 \subsubsection{First order equations for $q$, $p$, $u$ and $v$}
 By the chain rule
 \begin{equation}\label{qDeriv}
 \frac{dq_t}{ds} = \left( \frac{\partial}{\partial x}+\frac{\partial}{\partial
s}\right) Q_t(x;s)\left\vert_{x=s}. \right.
\end{equation}
 We have already computed the partial of $Q(x;s)$ with respect to $x$.
 The partial with respect to $s$ is, using (\ref{Kderiv}),
 \begin{equation*}
  \frac{\partial}{\partial s} Q_t(x;s)= \left(I-K\right)^{-1} \frac{\partial K}{\partial s}
   \left(I-K\right)^{-1}\tau_t \varphi= - R(x,s) q_t(s).
  \end{equation*}
 Adding the two partial derivatives  and
evaluating at $x=s$
  gives,
 \begin{equation}\label{qEqn}  
 \frac{dq_t}{ds} = p_t - \int_{-\infty}^{\infty} \sigma'(\tilde t)q_{\tilde t} u_{t,\tilde t} d\tilde t. 
\end{equation}
  A similar calculation gives,
  \begin{equation*}
 \frac{dp}{ds}=  (s+t) q_t + \int_{-\infty}^{\infty} \sigma'(\tilde t)\left\{ 
 p_{\tilde t}u_{t,\tilde t}-  q_{\tilde t} [v_{t,\tilde t}+ v_{\tilde t,t}]\right\}d\tilde t.
\end{equation*}
  We derive first order differential equations for  $u$ and $v$  by
differentiating
  the inner products.  $ u_{t,\tilde t}(s) = \int_s^\infty \tau_t\varphi(x) Q_{\tilde t}(x;s)\, dx$,
  \begin{eqnarray*}
  \frac{du_{t,\tilde t}}{ds}&=& -\tau_t\varphi(s) q_{\tilde t}(s) + \int_s^\infty \tau_t\varphi(x) \frac{\partial Q_{\tilde t}(x;s)}{\partial s}\, dx \\
  &=& -\left(\tau_t\varphi(s)+\int_s^\infty R(s,x) \tau_t\varphi(x)\,dx\right) q_{\tilde t}(s)\\
  &=& -\left(I-K\right)^{-1} \tau_t\varphi(s) \, q_{\tilde t}(s)= - q_tq_{\tilde t}.
  \end{eqnarray*}
   Similarly,
  $\frac{dv_{t,\tilde t}}{ds} = - q_tp_{\tilde t}.$
  \subsubsection{Integro-differential equation for $q_t$}
  From the first order differential equations for $q_t$, $u_t$ and $v_{t,\tilde t}$ it follows
  that the derivative in $s$ of
$\int_{-\infty}^{\infty} \sigma'(t') u_{t,t'}u_{t',\tilde t} dt' - [ v_{t,\tilde t}+v_{\tilde t, t}] -q_tq_{\tilde t}    
$
 is zero.  Examining the behavior near $s=\infty$ to check that
  the constant of integration is zero then gives,
  \begin{equation*}  
\int_{-\infty}^{\infty} \sigma'(t') u_{t,t'}u_{t',\tilde t} dt' - [ v_{t,\tilde t}+v_{\tilde t, t}]=q_tq_{\tilde t} , 
\end{equation*}
  a \textit{first integral}.
 Differentiate (\ref{qEqn}) with respect to $s$, to get
 \begin{equation*}
 q_t''= (s+t)q_t + \int_{-\infty}^{\infty} \sigma'(\tilde t) \Big\{ \int_{-\infty}^{\infty} \sigma'(t') q_{t'} u_{\tilde t, t'} dt' u_{t,\tilde t} - q_{\tilde t}[ v_{t,\tilde t} + v_{\tilde t, t} ] 
 + q_t q_{\tilde t}^2 \Big\} d\tilde t
 \end{equation*}
  and then use the  first integral to deduce that $q$ satisfies   (\ref{5.15}).
  
   Since the kernel of $[D,(I-K)^{-1}]$ is $(\partial/\partial
x+\partial/\partial y)R(x,y)$,
  (\ref{DComm}) says
  \begin{equation*} 
\left(\frac{\partial}{\partial x}+\frac{\partial}{\partial
y}\right)R(x,y)=- \int_{-\infty}^{\infty}\sigma'(t) Q_t(x) Q_t(y)dt + R(x,s) \rho(s,y). 
\end{equation*}
  In computing $\partial Q(x;s)/\partial s$ we showed that
  \begin{equation*} 
\frac{\partial}{\partial s} \left(I-K\right)^{-1}\doteq \frac{\partial}{\partial
s}R(x,y) = -R(x,s)\rho(s,y). 
\end{equation*}
  Adding these two expressions,
  \begin{equation*} 
\left(\frac{\partial}{\partial x}+\frac{\partial}{\partial y}+
  \frac{\partial}{\partial s}\right)R(x,y)=- \int_{-\infty}^{\infty}\sigma'(t) Q_t(x) Q_t(y)dt , 
\end{equation*}
  and then evaluating at $x=y=s$ gives
  \begin{equation*}
 \frac{d}{ds}R(s,s)=- \int_{-\infty}^{\infty}\sigma'(t)q_t^2(s) dt.
\end{equation*}
  Hence 
$ q_t^{\prime\prime}=\Big\{ s+t   - 2R'\Big\}  q_t.$
  Integrating and recalling (\ref{dLog}) gives,
   \begin{equation*} 
\frac{d}{ds}\log\det\left(I-K\right)=-\int_s^\infty \int_{-\infty}^{\infty}\sigma'(t)q_t^2(x) dt \, dx; 
\end{equation*}
  and hence,
  \begin{equation*} 
\log\det\left(I-K\right)=-\int_s^\infty\left(\int_y^\infty
\int_{-\infty}^{\infty}\sigma'(t)q_t^2(x) dt\,dx\right)\, dy.
\end{equation*}
Rearranging gives (\ref{sicks}).
This completes the proof of Proposition \ref{pII_prop}.
\end{proof}

\section{Proofs of Corollaries to Theorem \ref{main_result_thm}}\label{corollary_sec}
\subsection{Large time $F_{GUE}$ asymptotics (Proof of Corollary \ref{TW})}\label{twasymp}
We describe how to turn the idea described after Corollary \ref{TW} into a rigorous proof.
The first step is to cut the $\tilde\mu$ contour off outside of a compact region around the origin. Proposition \ref{reinclude_mu_lemma} shows that for a fixed $T$, the tail of the $\tilde\mu$ integrand is exponentially decaying in $\tilde\mu$. A quick inspection of the proof shows that increasing $T$ only further speeds up the decay. This justifies our ability to cut the contour at minimal cost. Of course, the larger the compact region, the smaller the cost (which goes to zero).

We may now assume that $\tilde\mu$ is on a compact region. We will show the following critical point: that $\det(I-K_{a}^{\csc})_{L^2(\tilde\Gamma_{\eta})}$ converges (uniformly in $\tilde\mu$) to the Fredholm determinant with kernel
\begin{equation}\label{limiting_kernel}
\int_{\Gamma_{\tilde\zeta}}e^{-\frac{1}{3}(\tilde\zeta^3-\tilde\eta'^3)+2^{1/3}s (\tilde\zeta-\tilde\eta')} \frac{d\tilde\zeta}{(\zeta-\eta')(\zeta-\eta)}.
\end{equation}
This claim shows that we approach, uniformly, a limit which is independent of $\tilde\mu$. Therefore, for large enough $T$ we may make the integral arbitrarily close to the integral of $\frac{e^{-\tilde\mu}}{\tilde\mu}$ times the above determinant (which is independent of $\tilde\mu$), over the cutoff $\tilde\mu$ contour. The $\tilde\mu$ integral approaches $1$ as the contour cutoff moves towards infinity, and the determinant is equal to $F_{\mathrm{GUE}}(2^{1/3}s)$ which proves the corollary. A remark worth making is that the complex contours on which we are dealing are not the same as those of \cite{TW3}, however, owing to the decay of the kernel and the integrand (in the kernel definition), changing the contours to those of \cite{TW3} has no effect on the determinant.

All that remains, then, is to prove the uniform convergence of the Fredholm determinant claimed above.

The proof of the claim follows in a rather standard manner. We start by taking a change of variables in the equation for $K_{a}^{\csc}$ in which we replace $\tilde\zeta$ by $T^{-1/3}\tilde\zeta$ and likewise for $\tilde\eta$ and $\tilde\eta'$. The resulting kernel is then given by
\begin{equation*}
T^{-1/3} \int_{\tilde\Gamma_{\zeta}} e^{-\frac{1}{3}(\tilde\zeta^3-\tilde\eta'^3)+2^{1/3}(s+a')(\tilde\zeta-\tilde\eta')} \frac{ \pi 2^{1/3}(-\tilde\mu)^{-2^{1/3}T^{-1/3}(\tilde\zeta-\tilde\eta')}}{\sin(\pi 2^{1/3}T^{-1/3}(\tilde\zeta-\tilde\eta'))} \frac{d\tilde\zeta}{\tilde\zeta-\tilde\eta}.
\end{equation*}
Notice that the $L^2$ space as well as the contour of $\tilde\zeta$ integration should have been dilated by a factor of $T^{1/3}$. However, it is possible (using Lemma \ref{TWprop1}) to show that we may deform these contours back to their original positions without changing the value of the determinant. We have also used the fact that $a=T^{1/3}s-\log\sqrt{2\pi T}$ and hence $T^{-1/3}a = s+a'$ where $a'=-T^{-1/3}\log\sqrt{2\pi T}$.

We may now factor this, just as in Proposition \ref{reinclude_mu_lemma}, as $AB$ and likewise we may factor our limiting kernel (\ref{limiting_kernel}) as $K'=A'B'$ where
\begin{eqnarray*}
&& A(\tilde\zeta,\tilde\eta)  = \frac{e^{-|\im(\tilde\zeta)|}}{\tilde\zeta-\tilde\eta}\\
\nonumber  B(\tilde\eta,\tilde\zeta) &= &e^{|\im(\tilde\zeta)|} e^{-\frac{1}{3}(\tilde\zeta^3-\tilde\eta^3)+2^{1/3}(s+a')(\tilde\zeta-\tilde\eta)} \frac{\pi2^{1/3}T^{-1/3} (-\tilde\mu)^{-2^{1/3}T^{-1/3}(\tilde\zeta-\tilde\eta)}}{ \sin(\pi 2^{1/3}T^{-1/3}(\tilde\zeta-\tilde\eta))}
\end{eqnarray*}
and similarly 
\begin{eqnarray*}
A'(\tilde\zeta,\tilde\eta) &=& \frac{e^{-|\im(\tilde\zeta)|}}{\tilde\zeta-\tilde\eta}\\
\nonumber B'(\tilde\eta,\tilde\zeta) &=& e^{|\im(\tilde\zeta)|} e^{-\frac{1}{3}(\tilde\zeta^3-\tilde\eta'^3)+2^{1/3}s)(\tilde\zeta-\tilde\eta')} \frac{1}{\tilde\zeta-\tilde\eta}
\end{eqnarray*}
Notice that $A=A'$. Now we use the estimate 
\begin{equation*}
|\det(I-K_a^{\csc})-\det(I-K')|\leq ||K_a^{\csc}-K'||_1 \exp\{1+||K_a^{\csc}||_1+||K'||_1\}. 
\end{equation*}
Observe that $||K_a^{\csc}-K'||_1\leq ||AB-AB'||_1\leq ||A||_2||B-B'||_2$. Therefore it suffices to show that $||B-B'||_2$ goes to zero (the boundedness of the trace norms in the exponential also follows from this). 
This is an explicit calculation and is easily made by taking into account the decay of the exponential terms, and the fact that $a'$ goes to zero. The uniformness of this estimate for compact sets of $\tilde\mu$ follows as well. This completes the proof of Corollary \ref{TW}.

\subsection{Small time Gaussian asymptotics }\label{Gaussian_asymptotics}
\begin{proposition}
As $T\beta^{4} \searrow0$, $2^{1/2}\pi^{-1/4}\beta^{-1}T^{-1/4}\mathcal{F}_\beta(T,X)$ converges in distribution to a standard Gaussian.
\end{proposition}
\begin{proof}  We have from (\ref{nine}),
\begin{equation*}
\mathcal{F}_\beta(T,X) = \log \left(1+ \beta T^{1/4} G(T,X)  + \beta^2 T^{1/2} \Omega(\beta,T,X)\right)
\end{equation*}
where
\begin{equation*}G(T,X) = T^{-1/4}\int_0^T \int_{-\infty}^\infty \frac{ p(T-S,X-Y)p(S,Y)}{p(T,X)} \mathscr{W}(dY,dS)
\end{equation*}
and
\begin{equation*}
\Omega(\beta,T,X) = T^{-1/2}
\sum_{n=2}^\infty \int_{\Delta_n(T)} \int_{\mathbb R^n}(-\beta)^{n-2}
p_{ t_1, \ldots, t_n}(x_1,\ldots,x_n) \mathscr{W} (dt_1 dx_1) \cdots \mathscr{W} (dt_n dx_n).
\end{equation*}
It is elementary to show that for each $T_0<\infty$ there is a $C=C(T_0)<\infty$ such that, for $T<T_0$
\begin{equation*}
E[\Omega^2(\beta,T,X)] \le C.
\end{equation*}
$G(T,X)$ is Gaussian and 
\begin{equation*}
E[ G^2(T,X)] = T^{-1/2}\int_0^T \int_{-\infty}^\infty \frac{ p^2(T-S,X-Y)p^2(S,Y)}{p^2(T,X)} dYdS
= \frac12 \sqrt{\pi }.
\end{equation*}
Hence by Chebyshev's inequality,
\begin{eqnarray*}
F_T(2^{-1/2}\pi^{1/4} \beta T^{1/4} s)  & =  & P( \beta T^{1/4} G(T,X)  + \beta^2 T^{1/2} \Omega(\beta,T,X)\le e^{2^{-1/2}\pi^{1/4} \beta T^{1/4} s}-1)
\nonumber
\\ & = & \int_{-\infty}^s \frac{e^{-x^2/2}}{\sqrt{2\pi}} dx + O( \beta T^{1/4}).\end{eqnarray*}
\end{proof}

        

\appendix
\section{Analytic properties of Fredholm Determinants}\label{PD_appendix}
The following appendix addresses the question of analytic properties of Fredholm determinants and is based on communications from Percy Deift. For a general discussion of Fredholm determinants see \cite{BS:book,RS:book}.

Suppose $A(z)$ is an analytic map from the region $D\in \C$ into the trace-class operators on a (separable) Hilbert space $\Hi$. Then we have the following result.

\begin{theorem}\label{Fredholm_Analytic}
 With $A:D\rightarrow \mathcal{B}_1(\Hi)$ as above, the map
$
z\mapsto \det(1+A(z))$
is analytic on $D$ and
\begin{eqnarray*}
\nonumber\frac{d}{dz}\det(1+A(z)) &=&
\tr A' +\tr(A'\otimes A+ A \otimes A') \\
&&+\tr(A'\otimes A \otimes A + A\otimes A' \otimes A +A\otimes A \otimes A' )+\cdots.
\end{eqnarray*}
\end{theorem}

We first prove the following  useful 

\begin{lemma}\label{Gamma_lemma}
 Suppose $A_1,\ldots,A_k \in \mathcal{B}_1(\Hi)$. Then
$
 \Gamma(A_1,\ldots, A_k) = \sum_{\pi\in S_k}A_{\pi(1)}\otimes \cdots \otimes A_{\pi(k)}
$ maps $\bigwedge^k (\Hi)$ to $\bigwedge^k (\Hi)$ and $\Gamma(A_1,\ldots, A_k) \in \mathcal{B}_1(\bigwedge^k (\Hi))$ with norm
\begin{equation}\label{Gamma_ineq}
||\Gamma(A_1,\ldots, A_k)||_1 \leq ||A_1||_1 ||A_2||_1\cdots ||A_k||_1.
\end{equation}
\end{lemma}

\begin{proof}
Since $A_j$ are trace class, they are also compact. Compact operators have singular value decompositions, which is to say that for each $j\in 1,\ldots, k$ there exists a decomposition of $A_j$ as
\begin{equation*}
A_j = \sum_{i\geq 1} a_{ji} (\alpha_{ji},\bullet)\alpha_{ji}',
\end{equation*}
where $a_{ji}\geq 0$, $\sum_{i=1}^{\infty} a_{ji}<\infty$, and $\{\alpha_{ji}\}$ as well as $\{\alpha_{ji}'\}$ are orthonormal. For $u_1,\ldots, u_k\in \Hi$, we write
\begin{equation*}
u_i\wedge \cdots \wedge u_k = \frac{1}{\sqrt{k!}} \sum_{\sigma\in S_k} \sgn(\sigma) u_{\sigma(1)}\otimes \cdots \otimes u_{\sigma(k)} \in \bigwedge^k (\Hi).
\end{equation*}
We will show in a moment that
\begin{equation}\label{blah}
 \Gamma(A_1,\ldots, A_k) u_1\wedge u_2\wedge \cdots \wedge u_k=\sum_{i_1,\ldots, i_k\geq 1}\prod_{l=1}^{k} a_{l,i_l}  \left((\bigwedge_{l=1}^{k} \alpha_{l,i_l}) , (\bigwedge_{l=1}^{k} u_l)\right) \bigwedge \alpha_{l,i_l}'.
\end{equation}
Hence, as linear combinations of $u_1\wedge \cdots \wedge u_k$ are dense in $\bigwedge^k (\Hi)$, we have
\begin{equation*}
\Gamma(A_1,\ldots,A_k) = \sum_{i_1,\ldots, i_k\geq 1} a_{1,i_1}\cdots a_{k,i_k} (\alpha_{1,i_1}\wedge \cdots \wedge \alpha_{k,i_k} ,\bullet) \alpha_{1,i_1}'\wedge \cdots \wedge \alpha_{k,i_k}',
\end{equation*}
which is the generalization of the singular value decomposition to the alternating product of operators.
As $||(u,\bullet)v||_{\mathcal{B}_1} = |(u,v)| \leq ||u||\cdot ||v||$ for any rank 1 operator in a Hilbert space, 
\begin{equation*}
||\Gamma(A_1,\ldots,A_k)||_{\mathcal{B}_1(\bigwedge^k(\Hi))} \leq \sum_{i_1,\ldots ,i_k\geq 1} a_{1,i_1}\cdots a_{k,i_k}  = ||A_1||_{\mathcal{B}_1} \cdots ||A_k||_{\mathcal{B}_1},\end{equation*}
 as
\begin{align}\nonumber
& ||(\alpha_{1,i_1}\wedge \cdots \wedge \alpha_{k,i_k},\bullet) \alpha_{1,i_1}'\wedge \cdots \wedge \alpha_{k,i_k}'||_{\mathcal{B}_1(\bigwedge^k (\Hi))}\\& \qquad \leq ||\alpha_{1,i_1}\wedge \cdots\wedge \alpha_{k,i_k}||\cdot ||\alpha_{1,i_1}'\wedge \cdots \wedge \alpha_{k,i_k}'||\leq 1.
\end{align}
This proves equation (\ref{Gamma_ineq}).
It remains to proves (\ref{blah}). Note that the left hand side can be written as 
\begin{align}
& \frac{1}{\sqrt{k!}} \sum_{\sigma\in S_k} \sum_{\pi\in S_k} \sgn(\sigma) (A_{\pi(1)}\otimes \cdots \otimes A_{\pi(k)}) u_{\sigma(1)}\otimes \cdots \otimes u_{\sigma(k)}\\ \nonumber
&= \sum_{i_1,\ldots, i_k\geq 1} \frac{1}{\sqrt{k!}} \sum_{\sigma,\pi \in S_k} \sgn(\sigma) \prod_{l=1}^{k} a_{\pi(l),i_l} \bigotimes_{l=1}^k ((\alpha_{\pi(l),i_l},\bullet)\alpha_{pi(l),i_l}')\bigotimes_{l=1}^{k} u_{\sigma(l)}
\end{align}
We recognize that $\sum_{\sigma\in S_k} \sgn(\sigma)\prod_{l=1}^{k} (\alpha_{\pi(l),i_l},u_{\sigma(l)})=\det\left[ (\alpha_{\pi(l),i_l},u_m)\right]_{l,m=1}^{k}$ so that, after a permutation, the whole thing becomes
\begin{align*}
&\sum_{i_1,\ldots, i_k\geq 1} \frac{1}{\sqrt{k!}} \sum_{\pi \in S_k} \sgn(\pi) \prod_{l=1}^{k} a_{l,i_l} \det\left[ (\alpha_{l,i_l},u_m)\right]_{l,m=1}^{k} \bigotimes_{l=1}^{k} \alpha_{\pi(l),i_{\pi(l)}}'\\
&=\sum_{i_1,\ldots, i_k\geq 1}\prod_{l=1}^{k} a_{l,i_l}  \left((\bigwedge_{l=1}^{k} \alpha_{l,i_l}) , (\bigwedge_{l=1}^{k} u_l)\right) \frac{1}{\sqrt{k!}} \sum_{\pi\in S_k} \sgn(\pi) \bigotimes_{l=1}^{k} \alpha_{\pi(l),i_{\pi(l)}}'\\
&=\sum_{i_1,\ldots, i_k\geq 1}\prod_{l=1}^{k} a_{l,i_l}  \left((\bigwedge_{l=1}^{k} \alpha_{l,i_l}) , (\bigwedge_{l=1}^{k} u_l)\right) \bigwedge \alpha_{l,i_l}'.
\end{align*}
which gives  (\ref{blah}).
\end{proof}

Now let $A,B\in \mathcal{B}_1(\Hi)$. For $l,m\geq 0$, $k=l+m$, define
\begin{equation*}
\Gamma^{(l,m)} (A,B) = \frac{1}{l!m!} \Gamma(A,\ldots, A, B,\ldots B),
\end{equation*}
where there are $l$ $A$'s and $m$ $B$'s. Clearly $\Gamma^{(l,m)} (A,B) = \sum c_1\otimes \cdots \otimes c_k$ where the sum is over all $m+l \choose m$ ways of designating $l$ of the $c_i$'s as $A$ and the other $m$ as $B$. As an example, $\Gamma^{(1,2)}(A,B) = A\otimes B\otimes B + B\otimes A\otimes B + B\otimes B\otimes A$.

\begin{corollary}[Corollary to Lemma \ref{Gamma_lemma}]
\begin{equation*}
||\Gamma^{(l,m)} (A,B)||_{\mathcal{B}_1(\bigwedge^k(\Hi))} \leq \frac{||A||_1^l}{l!}\frac{||B||_1^m}{m!}.
\end{equation*}
\end{corollary}

We can now proceed with:
\begin{proof}[Proof of Theorem \ref{Fredholm_Analytic}]
Fix $z\in D$ and let $A(z+h) = A(z) + \delta = A +\delta$. For $k\geq 1$,
\begin{align*}
\nonumber&A(z+h)\otimes \cdots \otimes A(z+h)= A+\delta \otimes \cdots \otimes A+\delta= A\otimes\cdots \otimes A + \Gamma^{(1,k-1)}(\delta, A) +\\ & +\Gamma^{(2,k-2)}(\delta,A) +\cdots + A^{(l,k-l)}(\delta, A)+\cdots + \delta \otimes \cdots \otimes \delta.
\end{align*}
Thus
\begin{equation*}\nonumber
h^{-1} \left\{ A(z+h)\otimes \cdots \otimes A(z+h) - A(z)\otimes \cdots \otimes A(z)\right\} = A^{(1,k-1)}({\delta}/{h},A)+\Delta(h),
\end{equation*}
where, by the Corollary,
\begin{equation*}
||\Delta(h)||_{\mathcal{B}_1(\bigwedge^k(\Hi))}\leq \frac{1}{|h|}  \frac{||\delta||_1^2}{2} \frac{||A||_1^{k-2}}{(k-2)!} +\cdots+\frac{1}{|h|}\frac{||\delta||_1^k}{k!}.
\end{equation*}
Observe that $||\delta||_1 = ||A(z+h)-A(z)||_1 = O(h)$. Write
\begin{equation*}
 A^{(1,k-1)}(\frac{\delta}{h},A) = \Gamma^{(1,k-1)}(A',A) + \Gamma^{(1,k-1)}(\frac{A(z+h)-A(z)}{h}-A'(z),A),
\end{equation*}
and then observe that by the Corollary
\begin{align}
&||\Gamma^{(1,k-1)}(\frac{A(z+h)-A(z)}{h}-A'(z),A(z))||_{\mathcal{B}_1(\bigwedge^k(\Hi))}\\
\nonumber&\leq ||\frac{A(z+h)-A(z)}{h}-A'(z)||_{\mathcal{B}_1} \frac{1}{(k-1)!}||A(z)||_{\mathcal{B}_1}^{k-1} = O(h).
\end{align}
Combining these observations shows that
\begin{equation*}\nonumber
h^{-1}\left\{A(z+h)\otimes \cdots \otimes A(z+h) - A(z)\otimes \cdots \otimes A(z)\right\}  = \Gamma^{(1,k-1)}(A',A)+O(h),
\end{equation*}
and hence the function $z\mapsto A(z)\otimes \cdots \otimes A(z) = \Gamma^{(k)}(A(z))$ is an analytic map from $D$ to $\mathcal{B}_1(\bigwedge^k (\Hi))$ for all $k\geq 1$ and
\begin{equation*}\nonumber 
\frac{d}{dz} A(z)\otimes \cdots \otimes A(z) = \Gamma^{(1,k-1)}(A',A) = A'\otimes A\otimes \cdots \otimes A +\cdots + A\otimes \cdots \otimes A\otimes A'.
\end{equation*}
 It then follows that $z\mapsto \tr \Gamma^{(k)}(A(z))$ is analytic for $k\geq 1$ from $D$ to $\C$.
Hence for any $n\geq 1$,
$
1+\sum_{k=1}^{n} \tr \Gamma^{(k)}(A(z))
$
 is analytic in $D$ and
\begin{equation*}\nonumber 
|1+\sum_{k=1}^{n} \tr \Gamma^{(k)}(A(z))|\leq 1+\sum_{k=1}^{n} ||\Gamma^{(k)}(A(z))||_{\mathcal{B}_1(\bigwedge^k(\Hi))}\leq 1+ \sum_{k=1}^{n}\frac{||A(z)||_{\mathcal{B}_1(\bigwedge^k(\Hi))}^k}{k!},
\end{equation*}
which is bouned by  $e^{||A(z)||}$, and so for $z$ in a compact subset of $D$, the functions $1+\sum_{k=1}^{n} \tr \Gamma^{(k)}(A(z))$ are uniformly bounded in $n$. It follows that $z\mapsto \det(I+A(z))=\lim_{n\rightarrow \infty} \sum_{k=0}^{n} \tr\Gamma^{(k)}(A(z))$ is analytic in $D$ and
\begin{align}
\nonumber&\frac{d}{dz} \det(I+A(z)) = \lim_{n\rightarrow \infty} \sum_{k=0}^{n} \frac{d}{dz} \tr \Gamma^{(k)}(A(z))=\sum_{k=1}^{\infty} \tr(\Gamma^{(1,k-1)}(A'(z),A(z)))\\
\nonumber&= \sum_{k=1}^{\infty} \tr (A'(z)\otimes A(z)\otimes \cdots \otimes A(z) + \cdots +A(z)\otimes \cdots\otimes A'(z)).
\end{align}
\end{proof}



\ack 
We would like to thank Percy Deift for his ongoing support and assistance with this project, as well as travel funding he provided to IC. We thank Craig Tracy and Harold Widom for discussing this matter during a visit in the summer of 2009, and further thank Tracy for ongoing interest and support.  JQ and GA wish to thank Kostya Khanin and Balint Vir\'ag for many interesting discussions and encouragement, and Tom Bloom and Ian Graham for their helpful suggestions on function theory. IC wishes to thank G\'{e}rard Ben Arous and Antonio Auffinger for helpful discussions and comments as well as Sunder Setheraman for an early discussion about the WASEP crossover. IC also wishes to thank all of the participants of the ASEP seminar which occurred during the 2008-2009 year. This collaboration was initially encouraged by Ron Peled, and we thank him graciously for playing matchmaker. We also thank Alex Bloemendal and David and Nora Ihilchik for hosting IC during his visits to Toronto. GA and JQ are supported by the Natural Science and Engineering Research Council of Canada.
IC is funded by the NSF graduate research fellowship, and has also received support from the PIRE grant OISE-07-30136.


\frenchspacing
\bibliographystyle{plain}

\end{document}